\documentclass[10pt,reqno]{amsart}
\usepackage{graphicx}
\usepackage[utf8]{inputenc}
\usepackage[T1]{fontenc} 
\usepackage{indentfirst,csquotes}

\usepackage{amssymb,amsthm,amsmath}
\usepackage{xcolor,paralist,hyperref,titlesec,fancyhdr,etoolbox}
\newtheorem{theorem}{Theorem}
\newtheorem{definition}{Definition}

\newtheorem{lemma}{Lemma}
\newtheorem{proposition}{Proposition}

\titleformat{\section}[display]{\normalfont\huge\bfseries\centering}{\centering\chaptertitlename\thechapter}{10pt}{\Large}
\titlespacing*{\section}{0pt}{0ex}{0ex}

\hypersetup{ colorlinks=true, linkcolor=black, filecolor=black, urlcolor=black }

\usepackage{lipsum}

\begin{document}
\title[Inverse  coefficient problem for  a fully fractional diffusion equation..]{Inverse  coefficient problem for  a fully fractional diffusion equation with nonlinear and source nonlocal initial condition} 
\author[D.K.Durdiev]{D.K.Durdiev}
\author[H.H. Turdiev]{H.H. Turdiev}
\date{\today}
\address{$ ^1$Bukhara Branch of Romanovskii Institute of Mathematics,
Uzbekistan Academy of Sciences, Bukhara, Uzbekistan}
\address{$^2$Bukhara State University, Bukhara, 705018 Uzbekistan}
\email{durdiyevdd@gmail.com}
\email{hturdiev@mail.ru, h.h.turdiyev@gmail.com}
\maketitle

\let\thefootnote\relax

{\bf Abstract}
In this work, we consider an inverse problem of determining a time dependent coefficient in a fully fractional diffusion equation with a nonlinear source term. The  nonlocal initial-boundary value problem refers to the forward model: the fractional diffusion equation equipped with a nonlocal initial condition and homogeneous Dirichlet boundary conditions. We first establish the existence and uniqueness of the mild solution to this nonlocal initial boundary value problem, together with the corresponding regularity properties of the solution. These results are obtained via the Fourier method, tools from fractional calculus, and key properties of the Mittag-Leffler function.
Subsequently, by applying a fixed-point argument in suitable Sobolev spaces, we prove a theorem on the local existence and uniqueness of the solution to the inverse problem. In this way, we establish the well-posedness of the  problem solution.

\bigskip


\section*{1 Introduction}

 In the domain $\Omega_{T}:=\{(x,t): 0 < x <l, \,\, 0<t\leq T\}$ consider the fully fractional diffusion- wave equation

\begin{equation}\label{eqth1}
\partial_t^\alpha u(t, x)+L^\beta u(t, x)+k(t)u(t, x)=F(t, x, u(t, x)), \,\, (t, x) \in (0,T) \times \Omega
\end{equation}
with the boundary condition
\begin{equation}\label{eqth2}
\textbf{B} u(t, x)=0, \quad (t, x) \in (0,T) \times \partial \Omega
\end{equation}
and the nonlocal initial condition
\begin{equation}\label{eqth3}
u(T, x)=\kappa u(0, x)+\varphi(x), \quad x \in \Omega,
\end{equation}
where $\Omega$ is an open and bounded domain in $\mathbb{R}^N,(N \geq 1)$ with boundary $\partial \Omega$, the Caputo fractional differential operator
$\partial_t^{\alpha}$ of the order $0<\alpha<1$  is defined in [[1], pp. 90-99]:
\begin{equation*}
\partial_t^{\alpha}u(x, t):=I_{t}^{1-\alpha}u_{t}(x, t)=\frac{1}{\Gamma(1-\alpha)}\int_0^t\frac{u_{\tau}(x, \tau)}{(t-\tau)^{\alpha}}d\tau,
\end{equation*}
\begin{equation*}
I_{t}^{\alpha}u(x, t):=\frac{1}{\Gamma(\alpha)}\int_0^t\frac{u(x, \tau)}{(t-\tau)^{1-\alpha}}d\tau, \, \, \alpha\in(0, 1),
\end{equation*}
the function $\Gamma(\cdot)$ represents the Gamma function, while $I_{t}^{\alpha}u(x, t)$ denotes the Riemann Liouville fractional integral of the function $u(x, t)$ with respect to the variable $t.$

The fractional power $L^\beta(0<\beta \leq 1)$ of the   operator $L$ on $\Omega$ is defined by its spectrum. The symmetric uniformly elliptic operator is defined on the space $L_2(\Omega)$ by
$$
L u(x)=-\sum_{i=1}^N \frac{\partial}{\partial x_i}\left(\sum_{j=1}^N A_{i j}(x) \frac{\partial}{\partial x_j} u(x)\right)+b(x) u(x),
$$
provided that $A_{i j} \in C^1(\bar{\Omega}), b \in C(\bar{\Omega}), b(x) \geq 0$ for all $x \in \bar{\Omega}, A_{i j}= A_{j i}, 1 \leq i, j \leq N$, and $\xi^T\left[A_{i j}(x)\right] \xi \geq L_0|\xi|^2$ for some $L_0>0, x \in \bar{\Omega}$, $\xi=\left(\xi_1, \xi_2, \ldots, \xi_N\right) \in \mathbb{R}^N$. Equation (\ref{eqth1}) is equipped with $\textbf{B} v=v$ or $\textbf{B} v=(L\nabla v) \cdot \vec{n}$, where $L=\left[A_{i j}(x)\right]_{i, j=1}^N$ is a $N \times N$ matrix and $n$ is the outer normal vector of $\partial D$. Then the operator $L$ is self-adjoint under this impedance boundary condition.
Here $\varphi$ is a given function, $F(t, x, u(t, x))$ is forcing or source function.

\textit{Inverse problem.} The inverse problem in this paper is to reconstruct  $k(t)$ according to the additional data
\begin{equation}\label{eqth4}
u(t, x_0)=h(t), \quad t \in [0,T],
\end{equation}
where $x_0\in \Omega$ is fixed point, $h(t)$ is given function.

In recent years, fractional differential equations have found many successful applications in various areas of modern science and engineering [2].  Such as, in mechanics, biological, physical science, and applied science, where the generalization of the classical diffusion-wave equation is the fractional diffusion-wave equation [1], [2].

Mathematical studies on initial-boundary problems for the fractional diffusion-wave equations have been extensively developed by many authors, for example, see [3]-[7].  Gorenlfo, Luchko and Yamamoto [8] proposed a modified definition of the Caputo derivative in the fractional Sobolev spaces, and proved the maximal regularity of the solutions to some initial boundary value problems for the time-fractional diffusion equation with the definition of Caputo derivative in the fractional Sobolev spaces.
Sakamoto and Yamamoto [4] investigated the existence, uniqueness, and regularity of solution of a  initial value/boundary value problem for fractional diffusion equation in space $C\left([0, T] ; L^2(\Omega)\right)$ with respect to the assumptions $a \in H_0^1(\Omega) \cap H^2(\Omega)$ and $f \in C^\theta\left([0, T] ; L^2(\Omega)\right)$.
Mu, Ahmad and Huang [9] investigated the existence and regularity of classical solutions on a weighted H\''older continuous function space, with respect to the assumption $f \in \mathcal{F}^{\beta, \theta} \left((0, T]L^2(\Omega)\right)$, which is weaker than H\"older continuous. Carvaho et al. [10] established a local theory of mild solutions for the problem (\ref{eqth1}) and (\ref{eqth3}) where $L^\beta$ is a sectorial (nonpositive) operator. Luchko et al. [11] discussed uniqueness and reconstruction of an unknown semilinear term in a time-fractional reaction diffusion equation

When $\kappa=0$ in problem (\ref{eqth1})--(\ref{eqth3}), Tuan et al. [12] studied a backward problem for a time-space fractional diffusion equation. For the problem with linear source function, regularity properties of solutions were studied. The existence, uniqueness and regularity to solution were proved in the case of nonlinear source.
Tuan, Huynh, Ngoc and Zhou [13] investigated a backward problem for a time fractional diffusion equation with a nonlinear source term
$$
\begin{cases}\partial_t^\alpha u(x, t)=A u(x, t)+F(x, t, u(x, t)), & (x, t) \in \Omega \times(0, T), \\ u(x, t)=0, & (x, t) \in \partial \Omega \times(0, T), \\ u(x, T)=f(x), & x \in \Omega, \end{cases}
$$
where $\Omega$ is a bounded domain in $\mathbb{R}^d$ with sufficiently smooth boundary $\partial \Omega, \frac{\partial^\alpha}{\partial t^\alpha}$ is the Caputo fractional derivative of order $0<\alpha<1$, $F$ is a given function-satisfying some hypotheses. The operator $A$ is a linear densely defined self-adjoint and positive definite elliptic operator on the connected bounded domain $\Omega$ with zero Dirichlet boundary condition.
Under certain assumptions, they established the existence and uniqueness of mild local solutions to the nonlinear problem. They also showed that the backward problem is ill-posed in the sense of Hadamard and proposed a regularization method to approximate its solution. Tuan et al. [14] studied a backward (final) value problem for the time-fractional diffusion equation with inhomogeneous source. They showed that the ill-posedness of the problem and provided a regularized method. Chen et al. in [15] develops weak solutions of time-space fractional diffusion equations on bounded domains.
Jia and Li [16] gave the maximum principles for the classical solution and weak solution.

In 2015, Gusvanto and Suzuki [17] considered the existence and uniqueness of a local mild solution of the fractional-order abstract Cauchy problem
$$
\left\{\begin{array}{lll}
\partial_t^{\alpha} u=A u+f(u), t>0,0<\alpha<1, \\
u(0)=u_0,
\end{array}\right.
$$
where $A: D(A) \rightarrow H$ is a sectorial linear operator, $H$ is a Banach space, $u_0 \in H$, and $f: H \rightarrow H$.

Inverse problems for fractional diffusion-wave equations have become a promising research area because of their wide applications in many scientific and engineering fields.  Currently, scientists have a great interest in inverse problems related to fractional derivatives. In most of these studies, a fractional time derivative is considered, and the inverse problem typically involves determining this derivative, a source term,  the coefficient, and the memory function under some additional conditions.  However, the majority of existing research focuses on linear models (see  [18]-[21]. Moreover, Wei and Wang [22] identified a space-dependent source term by a modified quasi-boundary value method from the final data.
In [23], the initial-boundary value problem for the time-fractional wave equation, as well as the associated inverse problem concerning the determination of the zero-order term, were investigated. First, the direct problem was studied under a Robin-type boundary condition. Utilizing the properties of the direct problem, the inverse problem was subsequently addressed. The local solvability of the inverse problem was established, and its stability was demonstrated.

Inverse problems for nonlinear and nonlocal diffusion wave equations represent a relatively new research direction, and only a limited number of experimental and analytical studies can be found in this field. Janno [24] investigated an inverse problem for identifying time-dependent sources and parameters in nonlocal diffusion and wave equations from final data.

\section*{Structure of the work}

To provide an overview of the present paper, we give the outline of the work, including some summaries of the mathematical contributions.

Section 2 provides some basic knowledge about function spaces, special operators, the mild formula, and some linear estimates.

Section 3,  we consider the regularity property of the solutions in the linear case $F$. We seek to address the following question: If the given functions are regular, how regular is the solution? Our task in this part is to find a suitable Banach space for the given data $(\varphi, F)$ in order to obtain regularity results. Particular, in subsection 3.1, we obtain an mild solution of the initial boundary value problem and the operator form of this solution. In subsection 3.2, we prove the existence, uniqueness, and regularity of the mild solution of a nonlocal initial problem with a linear source. We will investigate the fractional derivative of a mild solution over time in subsection  3.3.

Section 4, we consider the regularity property of solutions for the nonlinear case $F$.  In this case, in subsection 4.1, we discuss existence, uniqueness, and regularity for the solutions to (\ref{eqth1})--(\ref{eqth3}) for the nonlinear problem.   In subsection 4.2, we will investigate the existence of  time-space fractional derivatives of a mild solution of a nonlocal initial problem with a nonlinear source.

In section 5, we will investigate the inverse problem (\ref{eqth1})--(\ref{eqth4}). First, the inverse problem reduces to an equivalent Volterra-type integral equation. In section 6, the correctness of the inverse problem is proven.  The local existence and global uniqueness of the solution of the inverse problem are proven using the contraction mapping principle.

\section*{2 Preliminaries}

In this section, we present some useful definitions and results , which will be use in the future.

The function $E_\alpha(z)$ is named after the great Swedish mathematician G\''osta Magnus Mittag-Leffler (1846-1927) who defined it by a power series
$$
E_\alpha(z)=\sum_{k=0}^{\infty} \frac{z^k}{\Gamma(\alpha k+1)}, \quad \alpha \in \mathbb{C}, \operatorname{Re} \alpha>0,
$$
and studied its properties in 1902-1905 [25]-[28]  in connection with his summation method for divergent series.

{\bf Two parameter Mittag-Leffler (M-L) function.} The two parameter M-L function $E_{\alpha,\sigma}(z)$ is defined by the following series:
\begin{equation*}
E_{\alpha,\sigma}(z)=\sum\limits_{k=0}^{\infty}\frac{z^k}{\Gamma(\alpha k+\sigma)},
\end{equation*}
where $\alpha,\beta,z\in\mathbb{C}$ with $\mathfrak{R}(\alpha)>0$, $\mathfrak{R}(\alpha)$ denotes the real part of the complex number $\alpha$, $\Gamma(\cdot)$ is Euler's Gamma function. The Mittag-Leffler function has been studied by many authors who have proposed and studied
various generalizations and applications. A very interesting work that meets many results about
this function is due to Kilbas et al. (see [[1], pp. 42-44].
For $\sigma=1$, $\sigma$ reduces to the Mittag-Leffler function of single parameter, i.e. $E_{\alpha,1}(z)=E_{\alpha}(z).$

\begin{proposition}\label{proposition1} [[1], pp. 40-45], [[3], pp. 29-37]
 Let $0<\alpha<2$ and $\sigma\in\mathbb{R}$ be arbitrary. We suppose that $\nu$ is such that $\pi\alpha/2<\nu<\min\{\pi,\pi\alpha\}$. Then there exists a constant $C=C(\alpha,\sigma,\nu)>0$ such that
\begin{equation*}
\left|E_{\alpha,\beta}(z)\right|\leq\frac{C}{1+|z|},\quad \nu\leq|\mbox{arg}(z)|\leq\pi.
\end{equation*}
\end{proposition}

\begin{proposition}\label{propthth1} The Mittag-Leffler function of the negative argument $E_\alpha(-x)$ is a monotonically decreasing function for all $0<\alpha<1$ and
$$
0<E_\alpha(-x)<1.
$$
\end{proposition}

\begin{proposition}\label{proposition2} [[3], pp. 16-37] For $\alpha \in(0,1)$ and $t \in \mathbb{R}$ :

(i) $E_\alpha(t)>0,\,\, E_{\alpha,\alpha}(t)>0$,

(ii) $\left(E_\alpha(t)\right)'=\frac{1}{\alpha} E_{\alpha,\alpha}(t)$,

(iii) $\lim\limits_{t \rightarrow-\infty} E_\alpha(t)=\lim\limits_{t \rightarrow-\infty} E_{\alpha,\alpha}(t)=0$,

(iv) $\partial_t^\alpha E_\alpha\left(\omega t^\alpha\right)=\omega E_\alpha\left(\omega t^\alpha\right), \,\, \partial_t^{\alpha-1}\left(t^{\alpha-1} E_{\alpha,\alpha}\left(\omega t^\alpha\right)\right)=E_\alpha\left(\omega t^\alpha\right)$,

(v) $\partial_t^\alpha \left\{t^{\alpha-1}E_{\alpha, \alpha}\left(\omega t^\alpha\right)\right\}=\omega t^{\alpha-1}E_{\alpha, \alpha}\left(\omega t^\alpha\right),\,\,  \omega \in \mathbb{C}.$
\end{proposition}

\begin{proposition}\label{thproth3}  [[3],pp. 29-37], [13]. Let $0<\alpha_0<\alpha_1<1$. Then there exist positive constants $M_1, M_2$ depending only on $\alpha_0, \alpha_1$ such that for all $\alpha \in\left[\alpha_0, \alpha_1\right]$,
$$
\frac{M_1}{1+z} \leq E_\alpha(-z) \leq \frac{M_2}{1+z}, \quad E_{\alpha, \sigma}(-z) \leq \frac{M_2}{1+z},
$$
for all $z \geq 0,\,\, \sigma \in \mathbb{R}$.
\end{proposition}

\begin{proposition}\label{proposition4}  [[1], pp. 40-49]. Let $0<\alpha<1$ and $\lambda>0$. Then:

(i) $\frac{d}{d t} E_\alpha\left(-\lambda t^\alpha\right)=-\lambda t^{\alpha-1} E_{\alpha,\alpha}\left(-\lambda t^\alpha\right)$, for $t>0$,

(ii) $\frac{d}{d t}\left(t^{\alpha-1} E_{\alpha,\alpha}\left(-\lambda t^\alpha\right)\right)=t^{\alpha-2} E_{\alpha, \alpha-1}\left(-\lambda t^\alpha\right)$, for $t>0$,

(iii) $\int_0^{\infty} e^{-s t} E_\alpha\left(-\lambda t^\alpha\right) d t=\frac{s^{\alpha-1}}{s^\alpha+\lambda}$, for $\operatorname{Re}(s)>\lambda^{1 / \alpha}$.
\end{proposition}

\begin{theorem}\label{ththththth1} [29].  Suppose $0<\alpha<1$, the functions $u, g$ are nonnegative continuous functions defined on $t \geq 0, T \geq 0$ is a constant. If the following inequality is satisfied
$$
u(t) \leq C+\frac{1}{\Gamma(\alpha)} \int\limits_0^t(t-s)^{\alpha-1} g(s) u(s) d s+\frac{1}{\Gamma(\alpha)} \int\limits_0^T(T-s)^{\alpha-1} g(s) u(s) d s,\,\, t \in[0, T],
$$
then we have the following explicit estimate for $u(t)$ :
$$
u(t) \leq \frac{C \exp \left[\int\limits_0^{\frac{t^\alpha}{\Gamma(1+\alpha)}} g\left((s \Gamma(1+\alpha))^{\frac{1}{\alpha}}\right) d s\right]}{2-\exp \left[\int\limits_0^{\frac{T^\alpha}{\Gamma(1+\alpha)}} g\left((s \Gamma(1+\alpha))^{\frac{1}{\alpha}}\right) d s\right]}, \quad t \in[0, T]
$$
provided that $\exp \left[\int\limits_0^{\frac{T^\alpha}{\Gamma(1+\alpha)}} g\left((s \Gamma(1+\alpha))^{\frac{1}{\alpha}}\right) d s\right]<2$.
\end{theorem}

Next, we introduce some functional spaces for solutions of direct problem (\ref{eqth1})-(\ref{eqth3}). By $\left\{\lambda_j\right\}_{j \geq 1}$ and $\left\{e_j(x)\right\}_{j \geq 1}$, we denote the eigenvalues and sequence of eigenfunctions of $L$ which satisfy $e_j \in\left\{v \in H^2(\Omega): \textbf{B} v=0\right\}, L e_j(x)= \lambda_j e_j(x), 0<\lambda_1 \leq \lambda_2 \leq \ldots \leq \lambda_j \leq \ldots$, and $\lim\limits_{j \rightarrow \infty} \lambda_j=\infty$. The sequence $\left\{e_j(x)\right\}_{j \geq 1}$ forms an orthonormal basis of the space $L_2(\Omega)$. For a given real number $p \geq 0$, the Hilbert scale space $H^{2 p}(\Omega)$ is defined by
$$
\left\{v \in L_2(\Omega):\|v\|_{H^{2 p}(\Omega)}^2:=\sum_{j=1}^{\infty}\left(v, e_j\right)^2 \lambda_j^{2 p}<\infty\right\},
$$
where ( $\cdot, \cdot$ ) is the usual inner product of $L_2(\Omega)$. The fractional power $L^\beta, \beta \geq 0$, of the operator $L$ on $\Omega$ is defined by
\begin{equation}\label{eqth4444}
L^\beta v(x):=\sum_{j=1}^{\infty}\left(v, e_j\right) \lambda_j^\beta e_j(x).
\end{equation}
Then, $\left\{\lambda_j^\beta\right\}_{j \geq 1}$ is the  eigenvalues of the operator $L^\beta$. We denote by $D(L^\beta)$ the domain of $L^\beta$, and then
$$
D\left(L^{\beta}\right)=\left\{v \in L_2(\Omega):\left\|L^\beta v\right\|<\infty\right\},
$$
where $\|\cdot\|$ is the usual norm of $L_2(\Omega)$, and $D(L^\beta)$ is a Banach space with respect to the norm $\|v\|_{D(L^\beta)}=\left\|L^\beta v\right\|$. Moreover, the inclusion $D(L^\beta) \subset H^{2 \beta}(\Omega)$ holds for $\beta>0$. We identify the dual space $\left(L_2(\Omega)\right)'=L_2(\Omega)$ and define the domain $D\left(L^{-\beta}\right)$ by the dual space of $D(L^\beta)$, i.e., $D\left(L^{-\beta}\right)=\left(D(L^\beta)\right)'$. Then, $D\left(L^{-\beta}\right)$ is a Hilbert space endowed with the norm
$$
\|v\|_{D\left(L^{-\beta}\right)}:=\left\{\sum_{j=1}^{\infty}\left(v, e_j\right)_{-\beta, \beta}^2 \lambda_j^{-2 \beta}\right\}^{1 / 2},
$$
where $(\cdot, \cdot)_{-\beta, \beta}$ denotes the dual inner product between $D\left(L^{-\beta}\right)$ and $D(L^\beta)$. We note that the Sobolev embedding $D(L^\beta) \hookrightarrow L_2(\Omega) \hookrightarrow D\left(L^{-\beta}\right)$ holds for $0<\beta<1$, and $(\bar{v}, v)_{-\beta, \beta}=(\bar{v}, v)$, for $\bar{v} \in L_2(\Omega), v \in D(L^\beta)$. Hence, we have
\begin{equation}\label{eqth5555}
\left(e_i, e_j\right)_{-\beta, \beta}=\left(e_i, e_j\right)=\delta_{i j}
\end{equation}
where $\delta_{i j}$ is the Kronecker delta for $i,j \in \mathbb{N}.$

 Moreover, for given $p_1 \geq 1$ and $0<\eta<1$, we denote by $\mathcal{D}_{p_1, \eta}((0,T) \times \Omega)$ the set of all functions $f$ from $(0,T)$ to $L_{p_1}(\Omega)$ such that
\begin{equation}\label{eqhtt55}
\|f\|_{\mathcal{D}_{p_1, \eta}}:=\underset{0 \leq t \leq T}{\operatorname{ess} \sup } \int\limits_0^t\|f(\cdot,\tau)\|_{p_1}(t-\tau)^{\eta-1} d \tau<\infty,
\end{equation}
where $\|\cdot\|_{p_1}$ is the norm of $L_{p_1}(\Omega)$. Note that, for fixed $t>0$, the H\''older's inequality shows that
$$
\int\limits_0^t\|f(\cdot,\tau)\|_{p_1}(t-\tau)^{\eta-1} d \tau \leq\left[\int\limits_0^t\|f(\cdot,\tau)\|_{p_1}^{p_2} d \tau\right]^{\frac{1}{p_2}}\left[\int\limits_0^t(t-\tau)^{\frac{p_2(\eta-1)}{p_2-1}} d \tau\right]^{\frac{p_2-1}{p_2}}.
$$

In the above inequality, we note that the function $\tau \rightarrow(t-\tau)^{\frac{p_2(\eta-1)}{p_2-1}}$ is integrable for $p_2>\frac{1}{\eta}$. Therefore, if we let $L_{p_2}\left(0, T ; L_{p_1}(\Omega)\right),\,  p_1, p_2 \geq 1$, be the space of all Bochner's measurable functions $f$ from $(0,T)$ to $L_{p_1}(\Omega)$ such that
$$
\|f\|_{L_{p_2}\left(0, T ; L_{p_1}(\Omega)\right)}:=\left[\int\limits_0^t\|f(\tau, \cdot)\|_{p_1}^{p_2} d \tau\right]^{\frac{1}{p_2}}<\infty,
$$
then the following inclusion holds
\begin{equation}\label{theqt6}
L_{p_2}\left(0, T ; L_{p_1}(\Omega)\right) \subset \mathcal{D}_{p_1, \eta}((0,T) \times \Omega), \quad \text { for } p_2>\frac{1}{\eta}
\end{equation}
and there exists a positive constant $C>0$ such that
\begin{equation}\label{theqt7}
\left\|\|f\|_{\mathcal{D}_{p_1, \eta}} \leq C\right\| f \|_{L_{p_2}\left(0, T ; L_{p_1}(\Omega)\right)},
\end{equation}
here, $C$ depends only on $p_2, \eta$, and $T$. Moreover, for a given number $s$ such that $0<s<\eta$, we have $\mathcal{D}_{p_1, \eta-s}((0,T) \times \Omega) \subset \mathcal{D}_{p_1, \eta}((0,T) \times \Omega)$ since $\|f\|_{\mathcal{D}_{p_1, \eta}} \leq T^s \mid\|f\|_{\mathcal{D}_{p_1, \eta-s}}$. Let $B$ be a Banach space, and we denote by $C([0, T], B)$ the space of all continuous functions from $[0, T]$ to $B$ endowed with the norm $\|v\|_{C([0, T] ; B)}:=\sup\limits_{0 \leq t \leq T}\|v(t)\|_B$, and by $C^\theta([0, T], B)$ the subspace of $C([0, T] ; B)$ which includes all H\''older-continuous functions, and is equipped with the norm
$$
\||v|\|_{C^\theta([0, T], B)}:=\sup _{0 \leq t_1<t_2 \leq T} \frac{\left\|v\left(t_2\right)-v\left(t_1\right)\right\|_B}{\left|t_2-t_1\right|^\theta}.
$$

In some cases, a given function might not be continuous at $t=0$. Hence, it is useful to consider the set $C((0, T] ; B)$ which consists of all continuous functions from $(0, T]$ to $B$. We define by $C^\rho((0, T] ; B)$ the Banach space of all functions $v$ in $C((0, T] ; B)$ such that
$$
\|v\|_{C^\rho((0, T] ; B)}:=\sup _{0<t \leq T} t^\rho\|v(t)\|_B<\infty.
$$

Now, we investigate the problem of constructing solutions to the nonlocal value problem for a ordinary fractional equation.
\begin{equation}\label{eqth6666}
\partial_t^\alpha \omega(t)=g(t, \omega(t))-\Lambda \omega(t), t \in (0,T), \quad \text { and } \quad \omega(T)=\kappa \omega(0)+\psi,
\end{equation}
where $\Lambda, \kappa$ are given real numbers and  $\psi\in L_2(\Omega)$. Here, we wish to find a representation formula for $v$ in terms of the given function $g$ and $\psi$. By writing $\partial_t^\alpha=I_t^{1-\alpha} D_t$, and applying the fractional integral $I_t^\alpha$ on both sides of Equation (\ref{eqth6666}), we obtain
$$
\omega(t)=\omega(0)+I_t^\alpha[g(t, \omega(t))-\Lambda \omega(t)] .
$$
The Laplace transform yields that  [[1], pp. 18-22, 85]
$$
\widehat{\omega}=\frac{\lambda^{\alpha-1}}{\lambda^\alpha+\Lambda} \omega(0)+\frac{1}{\lambda^\alpha+\Lambda} \widehat{g}(\omega),
$$
where $\widehat{\omega}$ is the Laplace transform of $\omega$. Hence, the inverse Laplace transform implies
\begin{equation}\label{eqth7777}
\omega(t)=\omega(0) E_{\alpha, 1}\left(-\Lambda t^\alpha\right)+g(t, \omega(t)) \star\left[t^{\alpha-1} E_{\alpha, \alpha}\left(-\Lambda t^\alpha\right)\right],
\end{equation}
where $\star$ denotes the convolution, $
f_1 \star f_2=\int\limits_0^t f_1(\tau)  f_2(t-\tau)d \tau
$ for suitable functions $f_1$ and $f_2$.

Now, a representation of the solution of nonlocal value problem (\ref{eqth6666}) can be obtained by substituting $t=T$ into (\ref{eqth7777}), and using the nonlocal initial condition $\omega(T)=\kappa \omega(0)+\psi$, i.e.,
$$
v(t)=g(t, v(t)) \star \left\{t^{\alpha-1} E_{\alpha, \alpha}\left(-\Lambda t^\alpha\right)\right\}
$$$$
+ \left[\psi-g(T, v(T)) \star \left\{T^{\alpha-1} E_{\alpha, \alpha}\left(-\Lambda T^\alpha\right)\right\}\right] E_{\alpha, 1}\left(-\Lambda t^\alpha\right)\frac{E_{\alpha, 1}\left(-\Lambda t^\alpha\right)}{E_{\alpha, 1}\left(-\Lambda T^\alpha\right)-\kappa}.
$$

\section*{3 Investigation of initial boundary value problem (\ref{eqth1})--(\ref{eqth3})}

 \section*{3.1 Mild solution of the initial boundary value problem and representation of this solution in operator form}

Let $u_{j}(t)=(u(t, \cdot),e_{j}),\,\, j\geq 1.$ Assume that problem (\ref{eqth1})--(\ref{eqth3}) has a unique solution $u(t, x)$ given by
\begin{equation}\label{eqth8}
u(t, x)=\sum\limits_{j=1}^{\infty}u_{j}(t)e_{j}(x),
\end{equation}
where $u_{j} \, (j=1,2,\ldots)$  are solutions of nonlocal value problem problems:
\begin{equation}\label{ttheq1}
\partial_t^\alpha u_j(t)+\lambda_j^\beta u_j(t)+k(t)u_i(t)=F_j(t, u(t)),
\end{equation}
\begin{equation}\label{ttheq2}
u_{j}(T)=\kappa u_{j}(0)+\varphi_{j},
\end{equation}
where $\varphi_j=\left(\varphi, e_j\right)$, $F_j(t, u(t))=\left(F(t, x, u(t, x)), e_j\right) .$

Using the method of solving the nonlocal value probleminitial problem (\ref{eqth6666}), we obtain from problem (\ref{ttheq1}) - (\ref{ttheq2}):
$$
u_j(t)  =\bigg[\varphi_j+ \{k(T)u_j(T)\} \star \left\{T^{\alpha-1} E_{\alpha, \alpha}\left(-\lambda_j^\beta T^\alpha\right)\right\}
$$$$
-F_j(T, u(T)) \star \left\{T^{\alpha-1} E_{\alpha, \alpha}\left(-\lambda_j^\beta T^\alpha\right)\right\} \bigg]\Psi_j(T)E_{\alpha, 1}\left(-\lambda_j^\beta t^\alpha\right)
$$
\begin{equation}\label{eqth9}
+F_j(t, u(t)) \star \left\{t^{\alpha-1} E_{\alpha, \alpha}\left(-\lambda_j^\beta t^\alpha\right)\right\}-\{k(t) u_j(t)\} \star \left\{t^{\alpha-1} E_{\alpha, \alpha}\left(-\lambda_j^\beta t^\alpha\right)\right\},
\end{equation}
where
\begin{equation}\label{eqhttt11}
\Psi_j(T)=\frac{1}{E_{\alpha, 1}\left(-\lambda_j^{\beta}T^\alpha\right)-\kappa}.
\end{equation}

Using (\ref{eqth8}) and (\ref{eqth9}), we obtain the spectral representation of $u (x,t)$ as follows:
$$
u(t, x)=\sum_{j=1}^{\infty} \Psi_j(T) e_j(x) \bigg[\varphi_j-F_j(T, u(T)) \star \left\{T^{\alpha-1} E_{\alpha, \alpha}\left(-\lambda_j^\beta T^\alpha\right)\right\}\bigg]  E_{\alpha, 1}\left(-\lambda_j^\beta t^\alpha\right)
$$
$$
+\sum_{j=1}^{\infty} \Psi_j(T) e_j(x) \{k(T)u_j(T)\}  \star \left\{T^{\alpha-1} E_{\alpha, \alpha}\left(-\lambda_j^\beta T^\alpha\right)\right\}  E_{\alpha, 1}\left(-\lambda_j^\beta t^\alpha\right)
$$
$$
 +\sum_{j=1}^{\infty}e_j(x)  F_j(t, u(t)) \star \left\{t^{\alpha-1} E_{\alpha, \alpha}\left(-\lambda_j^\beta t^\alpha\right)\right\}
$$
\begin{equation}\label{eqht10}
-\sum_{j=1}^{\infty}e_j(x) \{k(t) u_j(t)\} \star \left\{t^{\alpha-1} E_{\alpha, \alpha}\left(-\lambda_j^\beta t^\alpha\right)\right\}.
\end{equation}

Let us first get acquainted with the analysis of function $\Psi_j(T)$ by (\ref{eqhttt11}).

$1^{\text {st }}$ case: $\kappa=0$. Then, $E_{\alpha, 1}\left(-\lambda_j^\beta T^\alpha\right) \neq 0$, but the Mittag-Leffler function can asymptotically tend towards zero (see Proposition \ref{proposition1}). Therefore, in this case we have:
$$
\left|\Psi_j(T)\right| \leq C_\kappa \lambda_j^{\beta} T^\alpha,
$$
where $C_\kappa$ is a constant that depends on $\kappa$.

$2^{\text {nd }}$ case: $0<\kappa<1$. Then, in view of Proposition \ref{propthth1}, there is a unique $\lambda_0>0$ such that $E_{\alpha, 1}\left(-\lambda_0^{\beta} T^\alpha\right)=\kappa$. If $\lambda_j \neq \lambda_0$ for all $n \in \mathbf{N}$, then
$$
\left|\Psi_j(T)\right| \leq C_\kappa
$$
the estimate is held with some constant $C_\kappa>0$. Therefore, if $\kappa \notin(0,1)$ or $\kappa \in(0,1)$, but $\lambda_j \neq \lambda_0$ for all $n \in \mathbf{N}$, then the formal solution of nonlocal value problemproblem (\ref{eqth1})-(\ref{eqth3}) has the form

$$
\omega(t, x)=Z_\kappa \psi(x), \quad (t, x) \in (0,T)\times \Omega.
$$

$3^{\text {rd }}$ case: $0<\kappa<1$ and $\lambda_j=\lambda_0$ for $j=n_0, n_0+1, \ldots, n_0+p_0-1$ where $p_0$ is the multiplicity of the eigenvalue $\lambda_{n_0}$. Then the nonlocal value problemproblem (\ref{eqth1})-(\ref{eqth3}) has a solution if the boundary function $\psi(x)$ satisfies the following orthogonality conditions

$$
\psi_j=\left(\psi, e_j\right)=0, \quad j \in \mathcal{K}, \quad \mathcal{K}=\left\{n_0, n_0+1, \ldots, n_0+p_0-1\right\} .
$$

For all other $j$ we have
$$
\left|\Psi_j(T)\right| \leq C_\kappa, \quad j \notin \mathcal{K} .
$$

Further in this article, we consider (\ref{eqth1})-(\ref{eqth3}) for the case when
\begin{equation}
\kappa \notin(0,1), \, \text{ or } \, \kappa \in(0,1),\, \text{ but } \, \lambda_j \neq \lambda_0, \, \text{ for all } \,  n \in \mathbb{N}.
\end{equation}

For $h_1(t, x)\in L_2(0, T;L_2(\Omega))$ and $h_2(x)\in L_2(\Omega)$, let us define the following operators
\begin{equation}\label{eqht1110}
\mathcal{G}_1(t, x) h_1:=\sum_{j=1}^{\infty} h_{1j}(t) \star \left\{t^{\alpha-1} E_{\alpha, \alpha}\left(-\lambda_j^\beta t^\alpha\right)\right\} e_j(x),
\end{equation}
\begin{equation}\label{eqht111011}
\mathcal{G}_2(t, x) h_{2}:=\sum_{j=1}^{\infty} h_{2j} \Psi_j(T)E_{\alpha, 1}\left(-\lambda_j^\beta t^\alpha\right) e_j(x),
\end{equation}
\begin{equation}\label{eqht1111}
\mathcal{G}_3(t, x)h_3=-\mathcal{G}_2(t, x) \mathcal{G}_1(T, x)h_3,
\end{equation}
for $(t, x) \in (0,T) \times \Omega$.
Then, the solution $u$ can be represented as
\begin{equation}\label{eqth10}
u(t, x)=\mathcal{G}_2(t, x) \varphi+\mathcal{G}_1(t, x) F+\mathcal{G}_3(t, x) F-\mathcal{G}_3(t, x) \{k,  u\}-\mathcal{G}_1(t, x)\{k,  u\}.
\end{equation}
where we understand $F(t, u)=F(t, \cdot, u(t, \cdot))$ is a function of $x$ for fixed $t$.

\section*{3.2 Nonlocal initial problem with a linear source}

In this subsection, we study the regularity of mild solutions of nonlocal value problem (\ref{eqth1})-(\ref{eqth3}) corresponding to the linear source function $F$, i.e., $F(t, x, u(t, x))=F(t, x)$ which does not include $u$.

We will investigate the regularity of the following nonlocal value problem
\begin{equation}\label{eqth11}
\left\{\begin{array}{lll}
\partial_t^\alpha u(t, x)+L^\beta u(t, x)+k(t)u(t, x)=F(t, x), \,\, (t, x) \in (0,T) \times \Omega, \\
\textbf{B} u(t, x)=0, \quad(t, x) \in (0,T) \times \partial \Omega, \\
u(T, x)=\kappa u(0, x)+\varphi(x), \quad x \in \Omega.
\end{array}\right.
\end{equation}

In order to consider this problem, it it necessary to give a definition of mild solutions based on (\ref{eqth10}) as follows.
\begin{definition}\label{difth4.1}
If a function $u(t, x)$ belongs to $L^p\left(0, T ; L^q(\Omega)\right)$, for some $p, q \geq 1$, and satisfies the equation
\begin{equation}\label{eqth12}
u(t, x)=\mathcal{G}_2(t, x) \varphi+\mathcal{G}_1(t, x) F+\mathcal{G}_3(t, x) F-\mathcal{G}_3(t, x) \{k,  u\}-\mathcal{G}_1(t, x)\{k,  u\}.
\end{equation}
then $u(t, x)$ is said to be a mild solution of nonlocal value problem (\ref{eqth11}).
\end{definition}

The following assertion is true.
\begin{lemma}\label{lemmth1}
Let $p, q$ $(0<p, q<1)$ be defined by $p+q=1$, and $u(t, x)$ satisfies (\ref{eqth12}). If $\varphi \in D\left(L^{\beta q}\right)$, $k\in C[0,T]$, $F \in \mathcal{D}_{2, \alpha q}((0,T) \times \Omega)$  and the inequality $\exp{\left\{ \frac{\Phi \|k\|_{C[0,T]}T^{2\alpha q}}{2\alpha q}\right\}}<2$ is satisfied, then there exists a constant $\Phi >0$ such that
\begin{equation}\label{est25}
\|t^{\alpha q} u(t, x)\|\leq \frac{\Phi \left(\|\varphi\|_{D\left(L^{\beta p}\right)}+  \|F\|_{\mathcal{D}_{2, \alpha q}}\right)}{2-\exp{\left\{ \frac{\Phi \|k\|_{C[0,T]}T^{2\alpha q}}{2\alpha q}\right\}}} \exp{\left\{ \frac{\Phi \|k\|_{C[0,T]}t^{2\alpha q}}{2\alpha q}\right\}},
\end{equation}
where $\Phi=\max\left\{ C_{\kappa} \frac{M_2 T^{\alpha q}}{\lambda_1^{2\beta p}}, \frac{ T^{\alpha q}}{\lambda_1^{\beta p}} M_2, \frac{C_{\kappa} M_2^2 T^{\alpha q}}{\lambda_1^{\beta p}}, \frac{C_{\kappa} M_2^2 }{\lambda_1^{\beta p}}, \frac{1}{\lambda_1^{\beta p}} M_2\right\}$.
\end{lemma}
\begin{proof}
 Boundedness of operators is one of the most important means of proving the correctness of differential equations. The inequalities in Proposition \ref{thproth3} show  that
$$
E_{\alpha, \alpha}\left(-\lambda_j^\beta(t-\tau)^\alpha\right) \leq \frac{M_2}{1+\lambda_j^\beta(t-\tau)^\alpha}  \leq \frac{M_2}{\left[1+\lambda_j^\beta(t-\tau)^\alpha\right]^{p}}
\leq \frac{M_2}{\lambda_j^{\beta p}(t-\tau)^{\alpha p}}.
$$
For the operator $\mathcal{G}_1(t, x) F$ defined by formula (\ref{eqht1110}), we have that
$$
\left\|t^{\alpha q}\mathcal{G}_1(t, \cdot) F\right\|  \leq  \int\limits_0^t \tau^{\alpha q} \left\|\sum_{j=1}^{\infty} F_j(\tau) (t-\tau)^{\alpha-1} E_{\alpha, \alpha}\left(-\lambda_j^\beta (t-\tau)^\alpha\right) e_j\right\| d \tau
$$$$
\leq   \int\limits_0^t \tau^{\alpha q} \left\{\sum_{j=1}^{\infty} \bigg( F_j (\tau) E_{\alpha, \alpha} \left(-\lambda_j^\beta(t-\tau)^\alpha\right)(t-\tau)^{ \alpha-1}\bigg)^2\right\}^{\frac{1}{2}}\left\{\sum_{j=1}^{\infty} |e_j|^2\right\}^{\frac{1}{2}} d \tau
$$$$
=  \int\limits_0^t \tau^{\alpha q} \left\{\sum_{j=1}^{\infty} F_j^2(\tau) E_{\alpha, \alpha}^2\left(-\lambda_j^\beta(t-\tau)^\alpha\right)(t-\tau)^{2 \alpha-2}\right\}^{\frac{1}{2}} d \tau
$$$$
 \leq M_2 \int\limits_0^t \tau^{\alpha q} \left\{\sum_{j=1}^{\infty} \frac{1}{\lambda_j^{2\beta p}(t-\tau)^{2\alpha p}}F_j^2(\tau) (t-\tau)^{2 \alpha-2}\right\}^{\frac{1}{2}} d \tau
$$
\begin{equation}\label{est20}
\leq T^{\alpha q} \frac{1}{\lambda_1^{\beta p}} M_2 \int\limits_0^t (t-\tau)^{ \alpha-1-\alpha p} \|F(\tau, \cdot)\|_{L_2(\Omega)}d \tau.
\end{equation}
Using the above inequality and noting (\ref{eqhtt55}), we obtain the following estimate
\begin{equation}\label{est21}
\left\| t^{\alpha q} \mathcal{G}_1(t, \cdot) F\right\| \leq \frac{1}{\lambda_1^{\beta p}} M_2   \|F\|_{\mathcal{D}_{2, \alpha q}}T^{\alpha q}.
\end{equation}

Now, $\left\| t^{\alpha q} \mathcal{G}_2(t, x) \varphi\right\|$ can be estimated as follows:
$$
\left\| t^{\alpha q} \mathcal{G}_2(t, \cdot) \varphi\right\|=  \left\|t^{\alpha q} \sum_{j=1}^{\infty}  \varphi_j \Psi_j(T)E_{\alpha, 1}\left(-\lambda_j^\beta t^\alpha\right) e_j(x)  \right\|
$$
$$
\leq T^{\alpha q} \left\{
\sum_{j=1}^{\infty} \left(\varphi_j \Psi_j(T)E_{\alpha, 1}\left(-\lambda_j^\beta t^\alpha\right) \right)^2\right\}^{1 / 2} \left\{\sum_{j=1}^{\infty} |e_j|^2\right\}^{\frac{1}{2}}
$$$$
\leq C_{\kappa} M_2 T^{\alpha q}
\left\{\sum_{j=1}^{\infty} \varphi_j^2 \frac{1}{\left(1+\lambda_j^\beta t^\alpha\right)^2}\right\}^{1 / 2} \leq C_{\kappa} M_2
T^{\alpha q}\left\{\sum_{j=1}^{\infty} \varphi_j^2 \frac{1}{\left(1+\lambda_j^\beta t^\alpha\right)^{2q}}\right\}^{1 / 2}
$$
\begin{equation}\label{est22}
\leq C_{\kappa}\frac{M_2 T^{\alpha q}}{\lambda_1^{2\beta p}}
\left\{\sum_{j=1}^{\infty} \varphi_j^2 \lambda_j^{2\beta p}\right\}^{1 / 2} = C_{\kappa} \frac{M_2 T^{\alpha q}}{\lambda_1^{2\beta p}}   \|\varphi\|_{D\left(L^{\beta p}\right)}.
\end{equation}

Now, by using the same techniques as in  (\ref{est20}) and (\ref{est22}), we estimates $\mathcal{G}_3(t, x)F$ and $\mathcal{G}_3(t, x)\{k,  u\}$.
As a result
$$
\left\|t^{\alpha q} \mathcal{G}_3(t, x)F\right\|=\left\|t^{\alpha q} \mathcal{G}_2(t, x) \mathcal{G}_1(T, x)F\right\|
$$
$$
=\left\|t^{\alpha q} \sum_{j=1}^{\infty} \Psi_j(T) e_j(x) F_j(T) \star \left\{T^{\alpha-1} E_{\alpha, \alpha}\left(-\lambda_j^\beta T^\alpha\right)\right\} E_{\alpha, 1}\left(-\lambda_j^\beta t^\alpha\right)\right\|
$$
$$
\leq \int\limits_0^T \tau^{\alpha q}\left\|\sum_{j=1}^{\infty} F_j(\tau) (T-\tau)^{\alpha-1} E_{\alpha, \alpha}\left(-\lambda_j^\beta (T-\tau)^\alpha\right) \Psi_j(T) e_j(x) E_{\alpha, 1}\left(-\lambda_j^\beta t^\alpha\right)\right\| d \tau
$$
$$
\leq T^{\alpha q}\int\limits_0^T\left\{\sum_{j=1}^{\infty} \bigg( F_j(\tau) (T-\tau)^{\alpha-1} E_{\alpha, \alpha}\left(-\lambda_j^\beta (T-\tau)^\alpha\right) \Psi_j(T)  E_{\alpha, 1}\left(-\lambda_j^\beta t^\alpha\right)\bigg)^2\right\}^{\frac{1}{2}} d \tau
$$
$$
\leq \frac{C_{\kappa} M_2^2 T^{\alpha q}}{\lambda_1^{\beta p}}    \int\limits_0^T \left\{\sum_{j=1}^{\infty} F_j^2(\tau) (T-\tau)^{2\alpha-2-2\alpha p}
\right\}^{\frac{1}{2}} d \tau
$$
\begin{equation}\label{est23}
\leq\frac{C_{\kappa} M_2^2 T^{\alpha q}}{\lambda_1^{\beta p}}    \int\limits_0^T  \| F_j(\tau,\cdot)\|_{L_2(\Omega)} (T-\tau)^{\alpha q-1}   d \tau  \leq \frac{C_{\kappa} M_2^2 T^{\alpha q}}{\lambda_1^{\beta p}}    \|F\|_{\mathcal{D}_{2, \alpha q}},
\end{equation}
$$
\left\|t^{\alpha q} \mathcal{G}_3(t, x)\{k,  u\}\right\|=\left\|t^{\alpha q} \mathcal{G}_2(t, x) \mathcal{G}_1(T, x)\{k,  u\}\right\|
$$
$$
=\left\|t^{\alpha q} \sum_{j=1}^{\infty} \Psi_j(T) e_j(x) k(T)  u_j(T) \star \left\{T^{\alpha-1} E_{\alpha, \alpha}\left(-\lambda_j^\beta T^\alpha\right)\right\} E_{\alpha, 1}\left(-\lambda_j^\beta t^\alpha\right)\right\|
$$
$$
\leq \int\limits_0^T \tau^{\alpha q}\left\|\sum_{j=1}^{\infty} k(\tau) u_j(\tau) (T-\tau)^{\alpha-1} E_{\alpha, \alpha}\left(-\lambda_j^\beta (T-\tau)^\alpha\right) \Psi_j(T) e_j(x) E_{\alpha, 1}\left(-\lambda_j^\beta t^\alpha\right)\right\| d \tau
$$
$$
\leq \int\limits_0^T \tau^{\alpha q} \left\{\sum_{j=1}^{\infty} \bigg( k(\tau) u_j(\tau) (T-\tau)^{\alpha-1} E_{\alpha, \alpha}\left(-\lambda_j^\beta (T-\tau)^\alpha\right) \Psi_j(T)  E_{\alpha, 1}\left(-\lambda_j^\beta t^\alpha\right)\bigg)^2\right\}^{\frac{1}{2}} d \tau
$$
$$
\leq \frac{C_{\kappa} M_2^2 }{\lambda_1^{\beta p}}    \int\limits_0^T \tau^{\alpha q} |k(\tau)| \left\{\sum_{j=1}^{\infty} u_j^2(\tau) (T-\tau)^{2\alpha-2-2\alpha p}
\right\}^{\frac{1}{2}} d \tau
$$
\begin{equation}\label{est233223}
\leq\frac{C_{\kappa} M_2^2 }{\lambda_1^{\beta p}}    \int\limits_0^T (T-\tau)^{\alpha q-1} \tau^{\alpha q} |k(\tau)| \| u(\tau,\cdot)\|_{L_2(\Omega)}   d \tau.
\end{equation}
Similar to $\mathcal{G}_1(t, x) F$, we estimate $\mathcal{G}_1(t, x)\{k,  u\}$
$$
\left\|t^{\alpha q} \mathcal{G}_1(t, x)\{k,  u\}\right\|  \leq \int\limits_0^t \tau^{\alpha q} \left\|\sum_{j=1}^{\infty} k(\tau) u_j(\tau) (t-\tau)^{\alpha-1} E_{\alpha, \alpha}\left(-\lambda_j^\beta (t-\tau)^\alpha\right) e_j\right\| d \tau
$$$$
\leq \int\limits_0^t \tau^{\alpha q} \left\{\sum_{j=1}^{\infty} \bigg( k(\tau)u_j (\tau) E_{\alpha, \alpha} \left(-\lambda_j^\beta(t-\tau)^\alpha\right)(t-\tau)^{ \alpha-1}\bigg)^2\right\}^{\frac{1}{2}}\left\{\sum_{j=1}^{\infty} |e_j|^2\right\}^{\frac{1}{2}} d \tau
$$$$
\leq \int\limits_0^t \tau^{\alpha q}  |k(\tau)| \left\{\sum_{j=1}^{\infty} u_j^2(\tau) E_{\alpha, \alpha}^2\left(-\lambda_j^\beta(t-\tau)^\alpha\right)(t-\tau)^{2 \alpha-2}\right\}^{\frac{1}{2}} d \tau
$$$$
 \leq M_2 \int\limits_0^t  \tau^{\alpha q}  |k(\tau)| \left\{\sum_{j=1}^{\infty} \frac{1}{\lambda_j^{2\beta p}(t-\tau)^{2\alpha p}}u_j^2(\tau) (t-\tau)^{2 \alpha-2}\right\}^{\frac{1}{2}} d \tau
$$
\begin{equation}\label{est24}
\leq \frac{1}{\lambda_1^{\beta p}} M_2   \int\limits_0^t (t-\tau)^{ \alpha q-1} \tau^{\alpha q}  |k(\tau)| \|u(\tau, \cdot)\|_{L_2(\Omega)}d \tau.
\end{equation}

In conclusion, it follows from (\ref{est20}) - (\ref{est24}) and the  integral equation solution (\ref{eqth12})  satisfyies estimates
\begin{equation*}
\|t^{\alpha q} u(t, x)\|\leq\|t^{\alpha q} \mathcal{G}_2(t, x) \varphi\|+\|t^{\alpha q} \mathcal{G}_1(t, x) F\|+\|t^{\alpha q} \mathcal{G}_3(t, x) F\|
+\|t^{\alpha q} \mathcal{G}_3(t, x) \{k,  u\}\|
\end{equation*}
$$
+\|t^{\alpha q}\mathcal{G}_1(t, x)\{k,  u\}\|   \leq C_{\kappa} \frac{M_2 T^{\alpha q}}{\lambda_1^{2\beta p}}   \|\varphi\|_{D\left(L^{\beta p}\right)}+ \frac{1}{\lambda_1^{\beta p}} M_2  \|F\|_{\mathcal{D}_{2, \alpha q}}T^{\alpha q}
$$
$$
+\frac{C_{\kappa} M_2^2 T^{\alpha q}}{\lambda_1^{\beta p}}    \|F\|_{\mathcal{D}_{2, \alpha q}}+
\frac{C_{\kappa} M_2^2 }{\lambda_1^{\beta p}}    \int\limits_0^T (T-\tau)^{\alpha q-1} \tau^{\alpha q} |k(\tau)| \| u(\tau,\cdot)\|_{L_2(\Omega)}   d \tau
$$
$$
+\frac{1}{\lambda_1^{\beta p}} M_2   \int\limits_0^t (t-\tau)^{ \alpha q-1} \tau^{\alpha q}  |k(\tau)| \|u(\tau, \cdot)\|_{L_2(\Omega)}d \tau.
$$
From this
$$
\|t^{\alpha q} u(t, x)\|\leq \Phi \left(\|\varphi\|_{D\left(L^{\beta p}\right)}+  \|F\|_{\mathcal{D}_{2, \alpha q}}\right)
$$
$$
+
\Phi \int\limits_0^T (T-\tau)^{\alpha q-1} \tau^{\alpha q} |k(\tau)| \| u(\tau,\cdot)\|_{L_2(\Omega)}   d \tau
+  \Phi \int\limits_0^t (t-\tau)^{ \alpha q-1} \tau^{\alpha q}  |k(\tau)| \|u(\tau, \cdot)\|_{L_2(\Omega)}d \tau.
$$
Then, applying the weak singular variant of the Gronwall-Bellman-Gamidov inequality given in Theorem \ref{ththththth1} to the above integral inequality, we obtain (\ref{est25}).  Lemma \ref{lemmth1} is proven.
\end{proof}

We require the following conditions to be satisfied for $p, q, s, r, p', q'$

 A1) $0<p, q<1$ such that $p+q=1$;

  A2) $0<r \leq \frac{1-\alpha q}{\alpha q}$;

  A3) $0<s<\min (\alpha q, 1-\alpha q)$;

  A4) $0<p' \leq p-\frac{s}{\alpha}, \quad q'=1-p', \quad 0<r \leq \frac{1-\alpha q'}{\alpha q'}$;

  A5) $0 \leq \hat{q} \leq \min \left(p, q, \frac{s}{\alpha}\right), \quad \hat{p}=1-\hat{q}, \quad 0<\hat{r} \leq \frac{1-\alpha}{\alpha}$.

In what follows, we introduce some assumptions on the nonlocal value problem value data $\varphi$ and the linear source function $F$.

 Based on Lemma \ref{lemmth1}, we consider the existence, uniqueness, and regularity of the solution in the following lemma.

\begin{lemma}\label{lemmth2}
Let  conditions A1), A2) be satisfied. If $\varphi \in D\left(L^{\beta p}\right)$, $k\in C[0,T]$, $F \in \mathcal{D}_{2, \alpha q}((0,T) \times \Omega)$ and the inequality $\exp{\left\{ \frac{\Phi \|k\|_{C[0,T]}T^{2\alpha q}}{2\alpha q}\right\}}<2$ is satisfied, then problem (\ref{eqth11}) has a unique solution $u(t, x)$ in $L^{\frac{1}{\alpha q}-r}\left(0, T ; L_2(\Omega)\right)$. Moreover, there exists a positive constant $C_1$ such that
\begin{equation}\label{eqth2727}
\|u(t, x)\|_{L^{\frac{1}{\alpha q}-r}(0,T;L_2(\Omega))} \leq \frac{C_1 \left(\|\varphi\|_{D\left(L^{\beta p}\right)}+  \|F\|_{\mathcal{D}_{2, \alpha q}}\right)}{2-\exp{\left\{ \frac{\Phi \|k\|_{C[0,T]}T^{2\alpha q}}{2\alpha q}\right\}}} \exp{\left\{ \frac{\Phi \|k\|_{C[0,T]}T^{2\alpha q}}{2\alpha q}\right\}},
\end{equation}
where $C_1$ is a constant depending on $\alpha,\, \beta,\, p,\, q$ and $T$.
\end{lemma}

\begin{proof}
The proof of Lemma \ref{lemmth2} can be easily obtained using Lemma \ref{lemmth1}. Indeed, from the inequality (\ref{est25}), we have:
$$
\|u(t, x)\|_{L^{\frac{1}{\alpha q}-r}(0,T;L_2(\Omega))} \leq \frac{\Phi \left(\|\varphi\|_{D\left(L^{\beta p}\right)}+  \|F\|_{\mathcal{D}_{2, \alpha q}}\right)}{2-
\exp{\left\{ \frac{\Phi \|k\|_{C[0,T]}T^{2\alpha q}}{2\alpha q}\right\}}}
$$$$
\times
\exp{\left\{ \frac{\Phi \|k\|_{C[0,T]}t^{2\alpha q}}{2\alpha q}\right\}}\left(\int\limits_0^T t^{-\alpha q (\frac{1}{\alpha q}-r)}d t\right)^{\frac{1}{\alpha q}-r}.
$$
Since $-\alpha q (\frac{1}{\alpha q}-r)>-1$, the integral in the above inequality exists, i.e., $t^{-\alpha q}$ belongs to $L^{\frac{1}{\alpha q}-r}(0,T;L_2(\mathbb{R}))$. Hence, problem (\ref{eqth11}) has a solution $u(t, x)$ in $L^{\frac{1}{\alpha q}-r}(0,T;L_2(\Omega))$. The uniqueness of $u(t, x)$ is obvious. Moreover, the inequality (\ref{eqth2727}) is derived by letting $C_1 =  \Phi \left(\int\limits_0^T t^{-\alpha q (\frac{1}{\alpha q}-r)}d t\right)^{\frac{1}{\alpha q}-r}$.
\end{proof}

\begin{lemma} \label{lemmth3}
 Let  conditions A1), A3) and A4) be satisfied.
If $\varphi \in D\left(L^{\beta p}\right)$, $k\in C[0,T]$, $F \in \mathcal{D}_{2, \alpha q-s}((0,T) \times \Omega)$ and the inequality $\exp{\left\{ \frac{\Phi_1 \|k\|_{C[0,T]}T^{\alpha q +\alpha q'}}{ \alpha q +\alpha q'}\right\}}<2$ is satisfied, then problem (\ref{eqth11}) has a unique solution $u(t, x)$, such that $u \in L^{\frac{1}{\alpha q}-r}\left(0, T ; D\left(L^{\beta\left(p-q\right)}\right)\right)$ and
\begin{equation*}
\|u(t, x)\|_{L^{\frac{1}{\alpha q'}-r}\left((0,T;D\left(L^{\beta\left(p-p'\right)}\right)\right)}\leq  \frac{C_2 \left(\|\varphi\|_{D\left(L^{\beta p}\right)}+  \|F\|_{\mathcal{D}_{2, \alpha q-s}}\right)}{2-\exp{\left\{ \frac{\Phi_1 \|k\|_{C[0,T]}T^{\alpha q +\alpha q'}}{ \alpha q +\alpha q'}\right\}}}
\end{equation*}
\begin{equation}\label{theq28}
\times \exp{\left\{ \frac{\Phi_1 \|k\|_{C[0,T]}T^{\alpha q +\alpha q'}}{ \alpha q +\alpha q'}\right\}},
\end{equation}
where $C_2$ is a constant depending on $\alpha,\, \beta,\, p,\, q,\, r$, $p'$, $q'$, $\kappa$ and $T$.
\end{lemma}
\begin{proof}
  We prove $u(t, x) \in L^{\frac{1}{\alpha q}-r}\left(0, T ; D\left(L^{-\beta\left(p-p'\right)}\right)\right)$. Firstly, by the same argument as in the proof of (\ref{est20}), we derive the following chain of inequalities
$$
\left\|t^{\alpha q'} \mathcal{G}_1(t, \cdot) F\right\|_{D\left(L^{\beta\left(p-p'\right)}\right)}
$$$$
\leq \int\limits_0^t \tau^{\alpha q'} \left\|L^{\beta\left(p-p'\right)}\sum_{j=1}^{\infty} F_j(\tau) (t-\tau)^{\alpha-1} E_{\alpha, \alpha}\left(-\lambda_j^\beta (t-\tau)^\alpha\right) e_j\right\| d \tau
$$
$$
 \leq M_2 \int\limits_0^t \tau^{\alpha q'} \left\{\sum_{j=1}^{\infty} \frac{1}{\lambda_j^{2\beta p}(t-\tau)^{2\alpha p}}F_j^2(\tau) (t-\tau)^{2 \alpha-2} \lambda_j^{2\beta (p-p')} \right\}^{\frac{1}{2}} d \tau
$$
$$
\leq \frac{M_2}{\lambda_1^{\beta p'}} T^{\alpha q'} \int\limits_0^t (t-\tau)^{ \alpha-1-\alpha p} \|F(\tau, \cdot)\|_{L_2(\Omega)}d \tau
$$
\begin{equation}\label{est2828}
\leq  \frac{M_2 T^{\alpha q'}}{\lambda_1^{\beta p'}}   \|F\|_{\mathcal{D}_{2, \alpha q }}
\leq \frac{M_2 T^{\alpha q'+s}}{\lambda_1^{\beta p'}}   \|F\|_{\mathcal{D}_{2, \alpha q-s}},
\end{equation}
where the inequality $\|F\|_{\mathcal{D}_{2, \alpha q}} \leq T^s\|F\|_{\mathcal{D}_{2, \alpha q-s}}$ holds.

Similarly, from $\left\|t^{\alpha q'}\mathcal{G}_2(t, \cdot) \varphi\right\|_{D\left(L^{\beta\left(p-p'\right)}\right)}=\left\|t ^{\alpha q'} L^{\beta\left(p-p'\right)} \mathcal{G}_2(t,\cdot) \varphi\right\|$, and the same way as in the proof of (\ref{est2828}), we have
$$
\left\|t^{\alpha q'} \mathcal{G}_2(t, \cdot) \varphi\right\|_{D\left(L^{\beta\left(p-p'\right)}\right)}
\leq T^{\alpha q'}\left\{
\sum_{j=1}^{\infty} \left(\varphi_j \Psi_j(T)E_{\alpha, 1}\left(-\lambda_j^\beta t^\alpha\right) \right)^2 \lambda_j^{2 \beta\left(p-p'\right)}\right\}^{1 / 2}
$$
\begin{equation}\label{est29}
\leq \frac{C_{\kappa} M_2 T^{\alpha q'}}{\lambda_1^{\beta p'}}
\left\{\sum_{j=1}^{\infty} \varphi_j^2 \lambda_j^{2 \beta p}\right\}^{1 / 2}
\leq \frac{C_{\kappa} M_2 T^{\alpha q'}}{\lambda_1^{\beta p'}}
\|\varphi\|_{D\left(L^{\beta p}\right)}.
\end{equation}
Using the above methods, we estimate $\mathcal{G}_3(t, x)F$ and $\mathcal{G}_3(t, x) \{k, u\}$
$$
\left\|t^{\alpha q'} \mathcal{G}_3(t, x)F\right\|_{D\left(L^{\beta\left(p-p'\right)}\right)}=\left\|t^{\alpha q'} L^{\beta\left(p-p'\right)}\mathcal{G}_2(t, x) \mathcal{G}_1(T, x)F\right\|
$$
$$
=\left\|t^{\alpha q'} L^{\beta\left(p-p'\right)}\sum_{j=1}^{\infty} \Psi_j(T) e_j(x) F_j(T) \star \left\{T^{\alpha-1} E_{\alpha, \alpha}\left(-\lambda_j^\beta T^\alpha\right)\right\} E_{\alpha, 1}\left(-\lambda_j^\beta t^\alpha\right)\right\|
$$
$$
\leq M_2 \int\limits_0^T \tau^{\alpha q'}\left\{\sum_{j=1}^{\infty} \bigg( F_j(\tau) (T-\tau)^{\alpha-1} E_{\alpha, \alpha}\left(-\lambda_j^\beta (T-\tau)^\alpha\right) \Psi_j(T)   \lambda_j^{\beta(p-p')}\bigg)^2\right\}^{\frac{1}{2}} d \tau
$$
$$
\leq \frac{C_{\kappa} M_2^2}{\lambda_1^{\beta p'}}    \int\limits_0^T \left\{\sum_{j=1}^{\infty} F_j^2(\tau) (T-\tau)^{2\alpha-2-2\alpha p}
\right\}^{\frac{1}{2}} d \tau
$$
\begin{equation}\label{est30}
\leq \frac{C_{\kappa} M_2^2 T^{\alpha q'}}{\lambda_1^{\beta p'}}    \int\limits_0^T  \| F(\tau,\cdot)\|_{L_2(\Omega)} (T-\tau)^{\alpha q-1}  d \tau  \leq \frac{C_{\kappa} M_2^2 T^{\alpha q'+s}}{\lambda_1^{\beta p'}}   \|F\|_{\mathcal{D}_{2, \alpha q-s}},
\end{equation}
$$
\left\|t^{\alpha q'} \mathcal{G}_3(t, x) \{k, u\} \right\|_{D\left(L^{\beta\left(p-p'\right)}\right)}=\left\|t^{\alpha q'} L^{\beta\left(p-p'\right)}\mathcal{G}_2(t, x) \mathcal{G}_1(T, x) \{k, u\} \right\|
$$
$$
=\left\|t^{\alpha q'} L^{\beta\left(p-p'\right)}\sum_{j=1}^{\infty} \Psi_j(T) e_j(x) k(T) u_j(T) \star \left\{T^{\alpha-1} E_{\alpha, \alpha}\left(-\lambda_j^\beta T^\alpha\right)\right\} E_{\alpha, 1}\left(-\lambda_j^\beta t^\alpha\right)\right\|
$$
$$
\leq M_2 \int\limits_0^T \tau^{\alpha q'}\left\{\sum_{j=1}^{\infty} \bigg( k(\tau)u_j(\tau) (T-\tau)^{\alpha-1} E_{\alpha, \alpha}\left(-\lambda_j^\beta (T-\tau)^\alpha\right) \Psi_j(T)   \lambda_j^{\beta(p-p')}\bigg)^2\right\}^{\frac{1}{2}} d \tau
$$
\begin{equation}\label{thest303030}
\leq \frac{C_{\kappa} M_2^2}{\lambda_1^{\beta p'}}    \int\limits_0^T (T-\tau)^{\alpha q -1} \tau^{\alpha q'} |k(\tau)|   \|u(\tau, \cdot)\|_{D\left(L^{\beta\left(p-p'\right)}\right)}d \tau.
\end{equation}

Using the same methods for estimating $\mathcal{G}_1 (t, x) F$, we estimate $\mathcal{G}_1 (t, x) \{k, u\}$
$$
\left\|t^{\alpha q'} \mathcal{G}_1(t, x)\{k,  u\}\right\|_{D\left(L^{\beta\left(p-p'\right)}\right)}
$$$$
\leq \int\limits_0^t \tau^{\alpha q'} \left\|L^{\beta\left(p-p'\right)}\sum_{j=1}^{\infty} k(\tau) u_j(\tau) (t-\tau)^{\alpha-1} E_{\alpha, \alpha}\left(-\lambda_j^\beta (t-\tau)^\alpha\right) e_j\right\| d \tau
$$$$
\leq \int\limits_0^t \tau^{\alpha q'} \left\{\sum_{j=1}^{\infty} \bigg( k(\tau)u_j (\tau) E_{\alpha, \alpha} \left(-\lambda_j^\beta(t-\tau)^\alpha\right)(t-\tau)^{ \alpha-1}\lambda_j^{\beta(p-p')}\bigg)^2\right\}^{\frac{1}{2}}
$$$$
\leq T^{\alpha q'} \int\limits_0^t  \tau^{\alpha q'} |k(\tau)|   \left\{\sum_{j=1}^{\infty} u_j^2(\tau) E_{\alpha, \alpha}^2\left(-\lambda_j^\beta(t-\tau)^\alpha\right)(t-\tau)^{2 \alpha-2}\lambda_j^{2\beta(p-p')}\right\}^{\frac{1}{2}} d \tau
$$$$
 \leq \frac{M_2 }{\lambda_1^{\beta p}}   \int\limits_0^t \tau^{\alpha q'}  |k(\tau)|  \left\{\sum_{j=1}^{\infty} \frac{1}{(t-\tau)^{2\alpha p}}u_j^2(\tau) (t-\tau)^{2 \alpha-2}\right\}^{\frac{1}{2}} d \tau
$$
\begin{equation}\label{est31}
\leq \frac{M_2 }{\lambda_1^{\beta p}}  \int\limits_0^t (t-\tau)^{ \alpha q-1} \tau^{\alpha q'} |k(\tau)|  \|u(\tau, \cdot)\|_{D\left(L^{\beta\left(p-p'\right)}\right)}d \tau.
\end{equation}

Finally, it follows from (\ref{est2828})-(\ref{est31}), and the identity (\ref{eqth12}) that
\begin{equation*}
\|t^{\alpha q'} u(t, x)\|_{D\left(L^{\beta\left(p-p'\right)}\right)}\leq\|t^{\alpha q'} \mathcal{G}_2(t, x) \varphi\|_{D\left(L^{\beta\left(p-p'\right)}\right)}+\|t^{\alpha q'} \mathcal{G}_1(t, x) F\|_{D\left(L^{\beta\left(p-p'\right)}\right)}
\end{equation*}
$$
+\|t^{\alpha q'} \mathcal{G}_3(t, x) F\|_{D\left(L^{\beta\left(p-p'\right)}\right)}+\|t^{\alpha q'} \mathcal{G}_1(t, x)\{k,  u\}\|_{D\left(L^{\beta\left(p-p'\right)}\right)}
$$
$$
\leq \frac{C_{\kappa} M_2 T^{\alpha q'}}{\lambda_1^{\beta p'}}
\|\varphi\|_{D\left(L^{\beta p}\right)}+
\frac{M_2 T^{\alpha q'+s}}{\lambda_1^{\beta p'}}   \|F\|_{\mathcal{D}_{2, \alpha q-s}}+
\frac{C_{\kappa} M_2^2 T^{\alpha q'+s}}{\lambda_1^{\beta p'}}  \|F\|_{\mathcal{D}_{2, \alpha q-s}}
$$
$$
+\frac{C_{\kappa} M_2^2}{\lambda_1^{\beta p'}}    \int\limits_0^T (T-\tau)^{\alpha q -1} \tau^{\alpha q'} |k(\tau)|   \|u(\tau, \cdot)\|_{D\left(L^{\beta\left(p-p'\right)}\right)}d \tau
$$
\begin{equation*}
+\frac{1}{\lambda_1^{\beta p}} M_2  \int\limits_0^t (t-\tau)^{ \alpha-1-\alpha p} \tau^{\alpha q'} |k(\tau)|  \|u(\tau, \cdot)\|_{D\left(L^{\beta\left(p-p'\right)}\right)}d \tau.
\end{equation*}
We can write the above inequality as follows
\begin{equation*}
\|t^{\alpha q'} u(t, x)\|_{D\left(L^{\beta\left(p-p'\right)}\right)}\leq \Phi_1\left(\|\varphi\|_{D\left(L^{\beta p}\right)}+
2\|F\|_{\mathcal{D}_{2, \alpha q-s}}\right)
\end{equation*}
$$
+  \Phi_1  \int\limits_0^T (T-\tau)^{\alpha q -1} \tau^{\alpha q'} |k(\tau)|   \|u(\tau, \cdot)\|_{D\left(L^{\beta\left(p-p'\right)}\right)}d \tau
$$
\begin{equation}\label{est32}
+\Phi_1 \int\limits_0^t (t-\tau)^{ \alpha q-1} \tau^{\alpha q'} |k(\tau)|  \|u(\tau, \cdot)\|_{D\left(L^{\beta\left(p-p'\right)}\right)}d \tau,
\end{equation}
where $\Phi_1=\max\left\{ \frac{C_{\kappa} M_2 T^{\alpha q'}}{\lambda_1^{\beta p'}}, \frac{M_2 T^{\alpha q'+s}}{\lambda_1^{\beta p'}}, \frac{C_{\kappa} M_2^2 T^{\alpha q'+s}}{\lambda_1^{\beta p'}}, \frac{C_{\kappa} M_2^2}{\lambda_1^{\beta p'}}, \frac{1}{\lambda_1^{\beta p}} M_2\right\}$.

According to the Gronwall-Bellman-Gamidov inequality given in Theorem \ref{ththththth1}, from the integral inequality (\ref{est32}) we get
\begin{equation}\label{thineq35}
\|t^{\alpha q'} u(t, x)\|_{D\left(L^{\beta\left(p-p'\right)}\right)} \leq \frac{\Phi_1 \left(\|\varphi\|_{D\left(L^{\beta p}\right)}+  \|F\|_{\mathcal{D}_{2, \alpha q-s}}\right)}{2-\exp{\left\{ \frac{\Phi_1 \|k\|_{C[0,T]}T^{\alpha q +\alpha q'}}{ \alpha q +\alpha q'}\right\}}} \exp{\left\{ \frac{\Phi_1 \|k\|_{C[0,T]}t^{\alpha q +\alpha q'}}{ \alpha q +\alpha q'}\right\}}.
\end{equation}

Since the function $t\rightarrow t^{-\alpha q'}$
is clearly contained in the space $L^{\frac{1}{\alpha q'}-r}(0,T; \mathbb{R})$, we can take the $L^{\frac{1}{\alpha q'}-r}(0,T; \mathbb{R})$-norm on both sides of the above inequality
\begin{equation*}
\|u(t, x)\|_{L^{\frac{1}{\alpha q'}-r}\left((0,T;D\left(L^{\beta\left(p-p'\right)}\right)\right)}\leq  \frac{C_2 \left(\|\varphi\|_{D\left(L^{\beta p}\right)}+  \|F\|_{\mathcal{D}_{2, \alpha q-s}}\right)}{2-\exp{\left\{ \frac{\Phi_1 \|k\|_{C[0,T]}T^{\alpha q +\alpha q'}}{ \alpha q +\alpha q'}\right\}}}
\end{equation*}
$$ \times \exp{\left\{ \frac{\Phi_1 \|k\|_{C[0,T]}t^{\alpha q +\alpha q'}}{ \alpha q +\alpha q'}\right\}},
$$
where $C_2 =  \Phi_1 \left(\int\limits_0^T t^{-\alpha q' (\frac{1}{\alpha q'}-r)}d t\right)^{\frac{1}{\alpha q'}-r}$.
From this it follows that the inequality (\ref{theq28}) in Lemma \ref{lemmth3} holds.
\end{proof}

 Now, we will prove that the solution belongs to the class of Hilbert continuous functions.

\begin{lemma}\label{lemmth4}
 Let  satisfied  conditions A1), A3) and A4).
If $\varphi \in D\left(L^{\beta p}\right)$, $k\in C[0,T]$ and $F \in \mathcal{D}_{2, \alpha q-s}((0,T) \times \Omega)$, then problem (\ref{eqth11}) has a unique solution $u(t, x)$, such that $u \in  C^{\alpha q}\left((0, T] ; L_2(\Omega)\right)$ and
\begin{equation}\label{theq29}
\|u(t, x)\|_{C^{\alpha q}\left((0,T];L_2(\Omega)\right)}\leq \frac{\Phi_2 \left(\|\varphi\|_{D\left(L^{\beta p}\right)}+  \|F\|_{\mathcal{D}_{2, \alpha q-s}}\right)}{2-\exp{\left\{ \frac{\Phi_1 \|k\|_{C[0,T]}T^{\alpha q +\alpha q'}}{ \alpha q +\alpha q'}\right\}}} \exp{\left\{ \frac{\Phi_1 \|k\|_{C[0,T]}T^{\alpha q +\alpha q'}}{ \alpha q +\alpha q'}\right\}}.
\end{equation}
\end{lemma}
\begin{proof}

We prove $u(t, x) \in  C^{\alpha q}((0, T];L_2(\Omega))$. Let us consider $0<t_1<t_2 \leq T$. By  (\ref{eqth12}), the difference $u(t_2, x) -u(t_1, x)$ can be calculated as
$$
u(t_2, x) -u(t_1, x)
=\sum_{j=1}^{\infty} \Psi_j(T) e_j(x) \varphi_j  E_{\alpha, 1}\left(-\lambda_j^\beta t^\alpha\right)\Bigg|_{t=t_1}^{t=t_2}
$$
$$
-\sum_{j=1}^{\infty} \Psi_j(T) e_j(x)  F_j(T) \star \left\{T^{\alpha-1} E_{\alpha, \alpha}\left(-\lambda_j^\beta T^\alpha\right)\right\}  E_{\alpha, 1}\left(-\lambda_j^\beta t^\alpha\right)\Bigg|_{t=t_1}^{t=t_2}
$$
$$
-\sum_{j=1}^{\infty} \Psi_j(T) e_j(x) \{k(T) u_j(T)\} \star \left\{T^{\alpha-1} E_{\alpha, \alpha}\left(-\lambda_j^\beta T^\alpha\right)\right\}  E_{\alpha, 1}\left(-\lambda_j^\beta t^\alpha\right)\Bigg|_{t=t_1}^{t=t_2}
$$
$$
 +\sum_{j=1}^{\infty}e_j(x)  F_j(t) \star \left\{t^{\alpha-1} E_{\alpha, \alpha}\left(-\lambda_j^\beta t^\alpha\right)\right\} \Bigg|_{t=t_1}^{t=t_2}
$$
\begin{equation}\label{eqht34}
-\sum_{j=1}^{\infty}e_j(x) \{k(t) u_j(t)\} \star \left\{t^{\alpha-1} E_{\alpha, \alpha}\left(-\lambda_j^\beta t^\alpha\right)\right\}\Bigg|_{t=t_1}^{t=t_2}.
\end{equation}
Using the differentiation identities given in Proposition \ref{proposition4}(ii),
we get
$$
\left.F_j(t) \star \left\{t^{\alpha-1} E_{\alpha, \alpha}\left(-\lambda_j^\beta t^\alpha\right)\right\} \right|_{t=t_1} ^{t=t_2}
=\int_0^{t_2} F_j(\tau) (t_2-\tau)^{\alpha-1} E_{\alpha, \alpha}\left(-\lambda_j^\beta (t_2-\tau)^{\alpha}\right) d \tau
$$
$$
-\int\limits_0^{t_1} F_j(\tau) (t_1-\tau)^{\alpha-1} E_{\alpha, \alpha}\left(-\lambda_j^\beta (t_1-\tau)^{\alpha}\right) d \tau
$$
$$
= \left.\int\limits_0^{t_1} F_j(\tau) (\omega-\tau)^{\alpha-1}E_{\alpha, \alpha}\left(-\lambda_j^\beta (\omega-\tau)^\alpha\right)\right|_{\omega=t_1} ^{\omega=t_2} d \tau
$$$$
+\int\limits_{t_1}^{t_2} F_j(\tau) (t_2-\tau)^{\alpha-1} E_{\alpha, \alpha}\left(-\lambda_j^\beta (t_2-\tau)^{\alpha}\right) d \tau
$$$$
= \int\limits_0^{t_1} \int\limits_{t_1-\tau}^{t_2-\tau} F_j(\tau) \eta^{\alpha-2} E_{\alpha, \alpha-1}\left(-\lambda_j^\beta \eta^\alpha\right) d \eta d \tau
$$
\begin{equation}\label{eqht35}
+\int\limits_{t_1}^{t_2} F_j(\tau) (t_2-\tau)^{\alpha-1} E_{\alpha, \alpha}\left(-\lambda_j^\beta (t_2-\tau)^{\alpha}\right) d \tau.
\end{equation}
In addition, we have
\begin{equation}\label{eqht36}
\left.E_{\alpha, 1}\left(-\lambda_j^\beta t^\alpha\right)\right|_{t=t_1} ^{t=t_2}=-\lambda_j^\beta \int\limits_{t_1}^{t_2} \tau^{\alpha-1} E_{\alpha, \alpha}\left(-\lambda_j^\beta \tau^\alpha\right) d \tau.
\end{equation}
According to the above (\ref{eqht34})-(\ref{eqht36}), we obtain
$$
u\left(t_2, x\right)-u\left(t_1, x\right)=\sum_{j=1}^{\infty} \int\limits_0^{t_1} \int\limits_{t_1-\tau}^{t_2-\tau} F_j(\tau) \eta^{\alpha-2} E_{\alpha, \alpha-1}\left(-\lambda_j^\beta \eta^\alpha\right) d \eta d \tau e_j(x)
$$
$$
+\sum_{j=1}^{\infty} \int\limits_{t_1}^{t_2} F_j(\tau) (t_2-\tau)^{\alpha-1} E_{\alpha, \alpha}\left(-\lambda_j^\beta (t_2-\tau)^{\alpha}\right) d \tau  e_j(x)
$$
$$
-L^\beta \sum_{j=1}^{\infty} \Psi_j(T)  \varphi_j  \int\limits_{t_1}^{t_2} \tau^{\alpha-1} E_{\alpha, \alpha}\left(-\lambda_j^\beta \tau^\alpha\right) d \tau e_j(x)
$$
$$
+L^\beta\sum_{j=1}^{\infty} \Psi_j(T)   F_j(T) \star \left\{T^{\alpha-1} E_{\alpha, \alpha}\left(-\lambda_j^\beta T^\alpha\right)\right\}  \int\limits_{t_1}^{t_2} \tau^{\alpha-1} E_{\alpha, \alpha}\left(-\lambda_j^\beta \tau^\alpha\right) d \tau e_j(x)
$$
$$
+L^\beta\sum_{j=1}^{\infty} \Psi_j(T)  \{k(T) u_j(T)\} \star \left\{T^{\alpha-1} E_{\alpha, \alpha}\left(-\lambda_j^\beta T^\alpha\right)\right\}  \int\limits_{t_1}^{t_2} \tau^{\alpha-1} E_{\alpha, \alpha}\left(-\lambda_j^\beta \tau^\alpha\right) d \tau e_j(x)
$$
$$
+\sum_{j=1}^{\infty} \int\limits_0^{t_1} \int\limits_{t_1-\tau}^{t_2-\tau} k(\tau) u_j(\tau) \eta^{\alpha-2} E_{\alpha, \alpha-1}\left(-\lambda_j^\beta \eta^\alpha\right) d \eta d \tau e_j(x)
$$
$$
+\sum_{j=1}^{\infty} \int\limits_{t_1}^{t_2} k(\tau) u_j(\tau) (t_2-\tau)^{\alpha-1} E_{\alpha, \alpha}\left(-\lambda_j^\beta (t_2-\tau)^{\alpha}\right) d \tau  e_j(x)
$$
\begin{equation}\label{eqth37}
:=\mathcal{I}_1+\mathcal{I}_2+\mathcal{I}_3+\mathcal{I}_4+\mathcal{I}_5+\mathcal{I}_6+\mathcal{I}_7.
\end{equation}

Now, we will establish an estimate for $u\left(t_2, x\right)-u\left(t_1, x\right)$, and show that $u\left(t_2, x\right)-u\left(t_1, x\right)$ tends to 0 as $t_2-t_1 \rightarrow 0$.

Firstly, according to the inequalities in Proposition \ref{thproth3}, we see that the absolute value of $E_{\alpha,\alpha-1}\left(-\lambda_j^\beta t^\alpha\right)$ is bounded by $\frac{M_2}{\lambda_j^{\beta p}} t^{-\alpha p}$. This implies
$$
t^{\alpha-2}\left|E_{\alpha, \alpha-1}\left(-\lambda_j^\beta t^\alpha\right)\right| \leq t^{\alpha-2}
\frac{M_2}{\left(1+\lambda_j^\beta t^\alpha\right)^p}
\leq \frac{M_2}{ \lambda_1^{\beta p}}   t^{\alpha q-2} .
$$

Moreover, for $0<\tau<t_1$, we have
$$
\int\limits_{t_1-\tau}^{t_2-\tau} \eta^{\alpha q-2} d \eta=\frac{1}{\alpha q-1} \left(\left(t_2-\tau\right)^{\alpha q-1}-\left(t_1-\tau\right)^{\alpha q-1}\right)
$$
$$
=\frac{1}{1-\alpha q}\left(\frac{1}{\left(t_1-\tau\right)^{1-\alpha q}}-\frac{1}{\left(t_2-\tau\right)^{1-\alpha q}}\right)
$$
\begin{equation}\label{eqth38}
 =\frac{1}{1-\alpha q} \frac{\left(t_2-\tau\right)^{1-\alpha q}-\left(t_1-\tau\right)^{1-\alpha q}}{\left(t_1-\tau\right)^{1-\alpha q}\left(t_2-\tau\right)^{1-\alpha q}}.
\end{equation}

We note that the estimates
\begin{equation}\label{esth39}
\left\{\begin{array}{l}
\left(t_2-\tau\right)^{1-\alpha q}-\left(t_1-\tau\right)^{1-\alpha q} \leq\left(t_2-t_1\right)^{1-\alpha q}, \\
\left(t_1-\tau\right)^s\left(t_2-t_1\right)^{1-\alpha q-s}\leq \left(t_2-\tau\right)^{1-\alpha q},
\end{array}\right.
\end{equation}
hold.  Inequality (\ref{esth39}) can be showed easily from $0<\alpha q<1$, and $1-\alpha q>s>0$. This is follows from $f(t)=t^\kappa, 0<\kappa<1: f\left(t_2\right)-f\left(t_1\right) \leq f\left(t_2-t_1\right)$ if $0<t_1<t_2$.

Therefore, we will examine the above statements
$$
\left\|\mathcal{I}_1\right\|   \leq \int_0^{t_1}\left\|\sum_{j=1}^{\infty}   \int\limits_{t_1-\tau}^{t_2-\tau} F_j(\tau) \eta^{\alpha-2} E_{\alpha, \alpha-1}\left(-\lambda_j^\beta \eta^\alpha\right) d \eta d \tau e_j(x)\right\| d \tau
$$$$
\leq  \frac{M_2}{ \lambda_1^{\beta p}}  \int\limits_0^{t_1} \int\limits_{t_1-\tau}^{t_2-\tau} \eta^{\alpha q-2} d \eta \|F(\tau, \cdot)\| d \tau.
$$
Using (\ref{eqth38}) and (\ref{esth39}) in the last integral of the above inequality, we get
$$
\left\|\mathcal{I}_1\right\|   \leq    \frac{M_2 \left(t_2-t_1\right)^s}{ (1-\alpha q)\lambda_1^{\beta p}}   \int_0^{t_1}\|F(\tau, \cdot)\|\left(t_1-\tau\right)^{\alpha q -s-1} d \tau
$$
\begin{equation}\label{est40}
  \leq  \frac{M_2 \left(t_2-t_1\right)^s}{ (1-\alpha q)\lambda_1^{\beta p}} \|F\|_{\mathcal{D}_{2, \alpha q-s}}.
\end{equation}

Secondly, an estimate for the term $\mathcal{I}_2$ can be shown using inequalities in Proposition \ref{thproth3} as follows
$$
\left\|\mathcal{I}_2\right\|  \leq
\int\limits_{t_1}^{t_2}\left\|\sum_{j=1}^{\infty}   F_j(\tau) (t_2-\tau)^{\alpha-1} E_{\alpha, \alpha}\left(-\lambda_j^\beta (t_2-\tau)^{\alpha}\right) e_j(x)\right\| d \tau
$$
$$
\leq \frac{M_2}{ \lambda_1^{\beta p}} \int\limits_{t_1}^{t_2}\|F(\tau, \cdot)\|\left(t_2-\tau\right)^{\alpha q-1} d \tau
$$
\begin{equation}\label{est41}
\leq \frac{M_2 \left(t_2-t_1\right)^s}{ \lambda_1^{\beta p}}  \int\limits_{t_1}^{t_2}\|F(\tau, \cdot)\|\left(t_2-\tau\right)^{\alpha q-s-1} d \tau \leq \frac{M_2 \left(t_2-t_1\right)^s}{ \lambda_1^{\beta p}}\|F\|_{\mathcal{D}_{2, \alpha q-s}}.
\end{equation}

 Third, we will estimate the term $\mathcal{I}_3$. We have
$$
\left\|\mathcal{I}_3\right\|=\left\|L^\beta \sum_{j=1}^{\infty} \Psi_j(T)  \varphi_j  \int\limits_{t_1}^{t_2} \tau^{\alpha-1} E_{\alpha, \alpha}\left(-\lambda_j^\beta \tau^\alpha\right) d \tau e_j(x)\right\|
$$
$$
\leq \left\{
\sum_{j=1}^{\infty}  \varphi_j^2 \Psi_j^2(T)\left[\int\limits_{t_1}^{t_2} \tau^{\alpha-1} E_{\alpha, \alpha}\left(-\lambda_j^\beta \tau^\alpha\right) d \tau\right]^2 \lambda_j^{2 \beta}\right\}^{1 / 2}
$$
$$
\leq C_{\kappa} M_2  \left\{ \sum_{j=1}^{\infty}  \varphi_j^2  \lambda_j^{2 \beta} \left[\int\limits_{t_1}^{t_2} \tau^{\alpha-1} \lambda_j^{-\beta p} \tau^{-\alpha p} d \tau\right]^2 \right\}^{1 / 2}
$$
$$
\leq C_{\kappa} M_2  \left\{ \sum_{j=1}^{\infty}  \varphi_j^2  \lambda_j^{2 \beta q} \left[\int\limits_{t_1}^{t_2} \tau^{\alpha q-1}  d \tau\right]^2 \right\}^{1 / 2}
$$
$$
\leq \frac{C_{\kappa} M_2}{\alpha q}   \left\{ \sum_{j=1}^{\infty}  \varphi_j^2  \lambda_j^{2 \beta q} \left[ t_2^{\alpha q}-t_1^{\alpha q}  \right]^2 \right\}^{1 / 2}.
$$
From inequality (\ref{esth39}), we have
\begin{equation}\label{est42}
\left\|\mathcal{I}_3\right\| \leq
\frac{C_{\kappa} M_2 T^{\alpha q}}{\alpha q}t^{-\alpha q} \left( t_2-t_1 \right)^{\alpha q}
\|\varphi\|_{D\left(L^{\beta q}\right)}.
\end{equation}

Fourth, we proceed to estimate
$\mathcal{I}_4$ and $\mathcal{I}_5$. Using the estimation method for $\mathcal{I}_3$ and inequalities in Proposition \ref{thproth3}, we get the following estimate for $\mathcal{I}_4$ and $\mathcal{I}_5$
$$
\left\|\mathcal{I}_4\right\|
\leq \int\limits_0^T \Bigg\{
\sum_{j=1}^{\infty} \lambda_j^{2 \beta} F_j^2(\tau) \Psi_j^2(T)(T-\tau)^{2\alpha-2} \left(E_{\alpha, \alpha}\left(-\lambda_j^\beta (T-\tau)^\alpha\right)\right)^2
$$
$$
\left[\int\limits_{t_1}^{t_2} \eta^{\alpha-1} E_{\alpha, \alpha}\left(-\lambda_j^\beta \eta^\alpha\right) d \eta\right]^2 \Bigg\}^{1 / 2}  d \tau
\leq  C_{\kappa} M_2 \int\limits_0^T \Bigg\{
\sum_{j=1}^{\infty} \lambda_j^{2 \beta} F_j^2(\tau)
$$
$$
 (T-\tau)^{2\alpha-2} \frac{1}{\lambda_j^{2\beta p}}
(T-\tau)^{-2\alpha p}\left[\int\limits_{t_1}^{t_2} \eta^{\alpha-1} \lambda_j^{-\beta q} \eta^{-\alpha q} d \eta\right]^2 \Bigg\}^{1 / 2}  d \tau
$$
$$
\leq  C_{\kappa} M_2 \int\limits_0^T \Bigg\{
\sum_{j=1}^{\infty} \lambda_j^{2 \beta} F_j^2(\tau) (T-\tau)^{2\alpha q-2}
 \frac{T^\alpha}{(\alpha q)^2 \lambda_j^{2\beta}}  \left[t_1^{-\alpha q}-t_2^{-\alpha q} \right]^2 \Bigg\}^{1 / 2}  d \tau
$$
$$
\leq  \frac{ C_{\kappa} M_2 T^\alpha}{\alpha q }  \int\limits_0^T   \|F(\tau, \cdot)\| (T-\tau)^{\alpha q-1}
\frac{t_2^{\alpha q}-t_1^{\alpha q}}{t_1^{\alpha q}t_2^{\alpha q} }    d \tau
$$
\begin{equation}\label{tfor43}
\leq  \frac{ C_{\kappa} M_2 T^\alpha}{\alpha q }  \frac{t_2^{\alpha q}-t_1^{\alpha q}}{t_1^{2\alpha q}} \|F\|_{\mathcal{D}_{2, \alpha q}}
\leq  \frac{ C_{\kappa} M_2 T^{\alpha+s}}{\alpha q }  t_1^{-2\alpha q}(t_2 -t_1)^{\alpha q}\|F\|_{\mathcal{D}_{2, \alpha q-s}}.
\end{equation}
$$
\left\|\mathcal{I}_5\right\|
\leq \int\limits_0^T \Bigg\{
\sum_{j=1}^{\infty} \lambda_j^{2 \beta} k^2(\tau) u_j^2(\tau) \Psi_j^2(T)(T-\tau)^{2\alpha-2} \left(E_{\alpha, \alpha}\left(-\lambda_j^\beta (T-\tau)^\alpha\right)\right)^2
$$
$$
\left[\int\limits_{t_1}^{t_2} \eta^{\alpha-1} E_{\alpha, \alpha}\left(-\lambda_j^\beta \eta^\alpha\right) d \eta\right]^2 \Bigg\}^{1 / 2}  d \tau
\leq  C_{\kappa} M_2 \int\limits_0^T \Bigg\{
\sum_{j=1}^{\infty} \lambda_j^{2 \beta} k^2(\tau) u_j^2(\tau)
$$
$$
 (T-\tau)^{2\alpha-2} \frac{1}{\lambda_j^{2\beta p}}
(T-\tau)^{-2\alpha p}\left[\int\limits_{t_1}^{t_2} \eta^{\alpha-1} \lambda_j^{-\beta q} \eta^{-\alpha q} d \eta\right]^2 \Bigg\}^{1 / 2}  d \tau
$$
$$
\leq  C_{\kappa} M_2 \int\limits_0^t \Bigg\{
\sum_{j=1}^{\infty} \lambda_j^{2 \beta} k^2(\tau) u_j^2(\tau) (T-\tau)^{2\alpha q-2}
 \frac{T^\alpha}{(\alpha q)^2 \lambda_j^{2\beta}}  \left[t_1^{-\alpha q}-t_2^{-\alpha q} \right]^2 \Bigg\}^{1 / 2}  d \tau
$$
$$
\leq  \frac{ C_{\kappa} M_2 T^\alpha}{\alpha q }  \int\limits_0^T  (T-\tau)^{\alpha q-1}  |k(\tau)|  \|u(\tau, \cdot)\|  \frac{t_2^{\alpha q}-t_1^{\alpha q}}{t_1^{\alpha q}t_2^{\alpha q} }   d \tau
$$
$$
\leq  \frac{ C_{\kappa} M_2 T^{\alpha}}{\alpha q }  t_1^{-2\alpha q}(t_2 -t_1)^{\alpha q}  \int\limits_0^T  (T-\tau)^{\alpha q-1}  |k(\tau)|  \|u(\tau, \cdot)\|    d \tau
$$
\begin{equation*}
\leq  \frac{ C_{\kappa} M_2 T^{\alpha}}{\alpha q }  t_1^{-2\alpha q}(t_2 -t_1)^{\alpha q}  \int\limits_0^T  (T-\tau)^{\alpha q-1}  |k(\tau)|  \|u(\tau, \cdot)\|    d \tau.
\end{equation*}
We note that the Sobolev embedding $D(L^{\beta\left(p-p'\right)}) \hookrightarrow L_2(\Omega)$ hold for $0< \beta\left(p-p'\right) < 1.$ Furthermore, based on this embedding, the inequality $\|\cdot\|_{L_2(\Omega)}\leq \|\cdot\|_{D(L^{\beta\left(p-p'\right)})}$ is satisfied. Using the above inequality and (\ref{thineq35}), we get the following estimate for $\mathcal{I}_5$
$$
\mathcal{I}_5 \leq  \frac{ C_{\kappa}  M_2 T^{\alpha q + \alpha q'+s} \|k\|_{C[[0,T]}}{\alpha q}  t_1^{-2\alpha q}(t_2 -t_1)^{\alpha q} B(\alpha q,1-\alpha q')
$$
\begin{equation}\label{esthth43}
\frac{\Phi_1 \left(\|\varphi\|_{D\left(L^{\beta p}\right)}+  \|F\|_{\mathcal{D}_{2, \alpha q-s}}\right)}{2-\exp{\left\{ \frac{\Phi_1 \|k\|_{C[0,T]}T^{\alpha q +\alpha q'}}{ \alpha q +\alpha q'}\right\}}} \exp{\left\{ \frac{\Phi_1 \|k\|_{C[0,T]}t^{\alpha q +\alpha q'}}{ \alpha q +\alpha q'}\right\}},
\end{equation}
where $B(\cdot, \cdot)$ is the Beta function.

Now let's estimate $\mathcal{I}_6$ and $\mathcal{I}_7$ in the same way. First, we'll estimate $\mathcal{I}_6$
$$
\left\|\mathcal{I}_6\right\|   \leq \int_0^{t_1}\left\|\sum_{j=1}^{\infty}   \int\limits_{t_1-\tau}^{t_2-\tau} k(\tau) u_j(\tau) \eta^{\alpha-2} E_{\alpha, \alpha-1}\left(-\lambda_j^\beta \eta^\alpha\right) d \eta d \tau e_j(x)\right\| d \tau
$$
$$
\leq  \frac{M_2 }{ \lambda_1^{\beta p}}  \int_0^{t_1} |k(\tau)|\|u(\tau, \cdot)\| d \tau \int\limits_{t_1-\tau}^{t_2-\tau} \eta^{\alpha q-2} d \eta.
$$
Using (\ref{eqth38}) and (\ref{esth39}) in the last integral of the above inequality, we get
$$
\left\|\mathcal{I}_6\right\| \leq  \frac{M_2 \|k\|_{C[[0,T]} \left(t_2-t_1\right)^s}{(1-\alpha q) \lambda_1^{\beta p}}  \int_0^{t_1}\|u(\tau, \cdot)\| \left(t_1-\tau\right)^{\alpha q -s-1} d \tau ,
$$
Let's estimate $\mathcal{I}_7$
$$
\left\|\mathcal{I}_7\right\|  \leq
\int\limits_{t_1}^{t_2}\left\|\sum_{j=1}^{\infty}  k(\tau) u_j(\tau) (t_2-\tau)^{\alpha-1} E_{\alpha, \alpha}\left(-\lambda_j^\beta (t_2-\tau)^{\alpha}\right) e_j(x)\right\| d \tau
$$
$$
\leq \frac{M_2 \|k\|_{C[[0,T]}}{ \lambda_1^{\beta p}} \int\limits_{t_1}^{t_2}\|u(\tau, \cdot)\|\left(t_2-\tau\right)^{\alpha q-s-1}\left(t_2-\tau\right)^{s} d \tau
$$
$$
\leq \frac{M_2 \|k\|_{C[[0,T]} \left(t_2-t_1\right)^{s}}{ \lambda_1^{\beta p}} \int\limits_{t_1}^{t_2}\|u(\tau, \cdot)\|\left(t_2-\tau\right)^{\alpha q-s-1} d \tau.
$$
We note that the Sobolev embedding $D(L^\beta) \hookrightarrow L_2(\Omega)$ hold for $0< \beta< 1.$ Furthermore, based on this embedding, the inequality $\|\cdot\|_{L_2(\Omega)}\leq \|\cdot\|_{D(L^\beta)}$ is satisfied. Using the above inequality and (\ref{thineq35}), we get the following estimates for $\mathcal{I}_6$ and $\mathcal{I}_7$
$$
\left\|\mathcal{I}_6\right\|   \leq    \frac{C_1 M_2 \|k\|_{C[[0,T]} \left(t_2-t_1\right)^s}{ (1-\alpha q)\lambda_1^{\beta p}} t_1^{\alpha(q - q') - s}   B(1 - \alpha q', \alpha q- s)
$$
\begin{equation}\label{est44}
\times \frac{\Phi_1 \left(\|\varphi\|_{D\left(L^{\beta p}\right)}+  \|F\|_{\mathcal{D}_{2, \alpha q-s}}\right)}{2-\exp{\left\{ \frac{\Phi_1 \|k\|_{C[0,T]}T^{\alpha q +\alpha q'}}{ \alpha q +\alpha q'}\right\}}} \exp{\left\{ \frac{\Phi_1 \|k\|_{C[0,T]}t^{\alpha q +\alpha q'}}{ \alpha q +\alpha q'}\right\}},
\end{equation}
$$
\left\|\mathcal{I}_7\right\|
\leq \frac{C_1 M_2 \|k\|_{C[[0,T]} \left(t_2-t_1\right)^{s}}{ \lambda_1^{\beta p}} t_2^{\alpha(q- q') - s}  B(1 - \alpha q', \alpha q- s)
$$
\begin{equation}\label{est45}
\times \frac{\Phi_1 \left(\|\varphi\|_{D\left(L^{\beta p}\right)}+  \|F\|_{\mathcal{D}_{2, \alpha q-s}}\right)}{2-\exp{\left\{ \frac{\Phi_1 \|k\|_{C[0,T]}T^{\alpha q +\alpha q'}}{ \alpha q +\alpha q'}\right\}}} \exp{\left\{ \frac{\Phi_1 \|k\|_{C[0,T]}t^{\alpha q +\alpha q'}}{ \alpha q +\alpha q'}\right\}}.
\end{equation}

Using (\ref{est40})-(\ref{est45}), we estimate the difference $u\left(t_2, x\right)-u\left(t_1, x\right)$ and obtain
$$
\|u\left(t_2, x\right)-u\left(t_1, x\right)\| \leq \sum\limits_{j=1}^{7}\left\|\mathcal{I}_j\right\| \leq
\frac{ (2-\alpha q) M_2}{ (1-\alpha q)\lambda_1^{\beta p}} \|F\|_{\mathcal{D}_{2, \alpha q-s}}  \left(t_2-t_1\right)^s
$$
$$
+
\frac{C_{\kappa} M_2 }{\alpha q}\left(T^{\alpha q}t^{-\alpha q}
\|\varphi\|_{D\left(L^{\beta q}\right)}+T^{\alpha+s}t_1^{-2\alpha q}\|F\|_{\mathcal{D}_{2, \alpha q-s}}\right)\left( t_2-t_1 \right)^{\alpha q}
$$
$$
+\Bigg(\frac{ C_{\kappa}  M_2 T^{\alpha q + \alpha q'+s} \|k\|_{C[[0,T]}}{\alpha q}  t_1^{-2\alpha q}(t_2 -t_1)^{\alpha q} B(\alpha q,1-\alpha q')
$$
$$
+  \frac{C_1 M_2 \|k\|_{C[[0,T]} \left(t_2-t_1\right)^s}{ (1-\alpha q)\lambda_1^{\beta p}} t_1^{\alpha(q - q') - s}   B(1 - \alpha q', \alpha q- s)
$$
$$
+\frac{C_1 M_2 \|k\|_{C[[0,T]} \left(t_2-t_1\right)^{s}}{ \lambda_1^{\beta p}} t_2^{\alpha(q- q') - s}  B(1 - \alpha q', \alpha q- s) \Bigg)
$$
$$
\times \frac{\Phi_1 \left(\|\varphi\|_{D\left(L^{\beta p}\right)}+  \|F\|_{\mathcal{D}_{2, \alpha q-s}}\right)}{2-\exp{\left\{ \frac{\Phi_1 \|k\|_{C[0,T]}T^{\alpha q +\alpha q'}}{ \alpha q +\alpha q'}\right\}}} \exp{\left\{ \frac{\Phi_1 \|k\|_{C[0,T]}t^{\alpha q +\alpha q'}}{ \alpha q +\alpha q'}\right\}}.
$$

We derive from the above that $\|u\left(t_2, x\right)-u\left(t_1, x\right)\|$ tends to $0$ as $t_2-t_1$ tends $0$ for $0<t_1<t_2\leq T$.  Thus, $u(t, x)$ belongs to the set $C\left((0,T];L_2(\Omega)\right)$. On the other hand, according to the condition $0<s<\min (\alpha q, 1-\alpha q)$, we have $0< \alpha q-s< \alpha q$ and the Sobolev embedding $\mathcal{D}_{2, \alpha q-s}\left((0,T)\times \Omega\right) \subset
 \mathcal{D}_{2, \alpha q}\left((0,T)\times \Omega\right)$.
Therefore, the functions $\varphi(x)$ and $F(t, x)$, satisfying the conditions of Lemma \ref{lemmth4}, also satisfy the conditions of Lemma \ref{lemmth1}. Hence, the inequality (\ref{est25}) holds, i.e.,
\begin{equation}\label{est46}
t^{\alpha q}\|u(t, x)\|\leq  \frac{\Phi_2 \left(\|\varphi\|_{D\left(L^{\beta p}\right)}+  \|F\|_{\mathcal{D}_{2, \alpha q-s}}\right)}{2-\exp{\left\{ \frac{\Phi_1 \|k\|_{C[0,T]}T^{\alpha q +\alpha q'}}{ \alpha q +\alpha q'}\right\}}} \exp{\left\{ \frac{\Phi_1 \|k\|_{C[0,T]}T^{\alpha q +\alpha q'}}{ \alpha q +\alpha q'}\right\}}, \,\, t>0.
\end{equation}
Therefore, the function $u(t, x) $ belongs to $C^{\alpha q}\left((0,T];L_2(\Omega)\right)$. In addition, by taking the supremum on both sides of (\ref{est46}) at $(0, T]$, we obtain the inequality (\ref{theq29}).  Lemma \ref{lemmth4} is proven.
\end{proof}

\begin{lemma}\label{lemmth5}
 Let  satisfied  conditions A1), A2).
If $\varphi \in D\left(L^{\beta p}\right)$, $k\in C[0,T]$ and $F \in \mathcal{D}_{2, \alpha q-s}((0,T) \times \Omega)$, then problem (\ref{eqth11}) has a unique solution $u(t, x)$, such that  $u \in   C^s\left([0, T] ; D\left(L^{-\beta q}\right)\right)$ and
$$
\|u\left(t, x\right)\|_{C^{s}\left([0,T];D\left(L^{-\beta q'}\right)\right)}
\leq
\frac{C_{\kappa} M_2 T^{\alpha q-s}}{\alpha q \lambda_1^{2 \beta q'}}
\|\varphi\|_{D\left(L^{\beta q}\right)}
+\frac{(2-\alpha q)M_2 }{ (1-\alpha q)\lambda_1^{\beta (p+q')}} \|F\|_{\mathcal{D}_{2, \alpha q-s}}
$$
$$
+\frac{ C_{\kappa} M_2 T^{\alpha + \alpha q-s}}{\alpha q }  t_1^{-2\alpha q} \|F\|_{\mathcal{D}_{2, \alpha q-s}}
+\Bigg(\frac{ C_{\kappa} T^{\alpha + \alpha q+\alpha q'}  }{\alpha q }  t_1^{-2\alpha q}  B(\alpha q,1-\alpha q')
$$$$
+\frac{t_1^{\alpha(q-q') - s} }{ (1-\alpha q)\lambda_1^{\beta (p+q')}}   B(1 - \alpha q', \alpha q- s)
 +\frac{t_2^{\alpha(q- q') - s} }{ \lambda_1^{\beta (p+q')}}   B(1 - \alpha q', \alpha q- s)\Bigg)M_2 \|k\|_{C[0,T]}
$$
\begin{equation}\label{theq30}
\times \frac{\Phi_1 \left(\|\varphi\|_{D\left(L^{\beta p}\right)}+  \|F\|_{\mathcal{D}_{2, \alpha q-s}}\right)}{2-\exp{\left\{ \frac{\Phi_1 \|k\|_{C[0,T]}T^{\alpha q +\alpha q'}}{ \alpha q +\alpha q'}\right\}}} \exp{\left\{ \frac{\Phi_1 \|k\|_{C[0,T]}T^{\alpha q +\alpha q'}}{ \alpha q +\alpha q'}\right\}}.
\end{equation}
\end{lemma}
\begin{proof}
We prove $u \in C^s\left([0, T] ; D\left(L^{-\beta q'}\right)\right)$. In this step, we establish the continuity of the solution on the closed interval $[0, T]$. Now, we consider $0 \leq t_1<t_2 \leq T$. If $t_1=0$, then $\mathcal{I}_1=0$. If $t_1>0$, then combining (\ref{eqth5555}) in the same way as in (\ref{est40}) gives
$$
\left\|\mathcal{I}_1\right\|_{D\left(L^{-\beta q'}\right)}   \leq \int_0^{t_1}\left\|\sum_{j=1}^{\infty}   \int\limits_{t_1-\tau}^{t_2-\tau} F_j(\tau) \eta^{\alpha-2} E_{\alpha, \alpha-1}\left(-\lambda_j^\beta \eta^\alpha\right) d \eta   e_j(x)\right\|_{D\left(L^{-\beta q'}\right)} d \tau
$$
$$
\leq \int_0^{t_1}\left\{\sum_{j=1}^{\infty} \lambda_j^{-2\beta q'}  \left(\int\limits_{t_1-\tau}^{t_2-\tau} F_j(\tau) \eta^{\alpha-2} E_{\alpha, \alpha-1}\left(-\lambda_j^\beta \eta^\alpha\right) d \eta \right)_{-\beta q',\beta q'}^2 \right\}^{\frac{1}{2}} d \tau
$$
$$
\leq \int_0^{t_1}\left\{\sum_{j=1}^{\infty} \lambda_j^{-2\beta q'}  F_j^2(\tau) \left|\int\limits_{t_1-\tau}^{t_2-\tau} \eta^{\alpha-2} E_{\alpha, \alpha-1}\left(-\lambda_j^\beta \eta^\alpha\right) d \eta \right|^2 \right\}^{\frac{1}{2}} d \tau
$$
$$
\leq  \frac{M_2}{ \lambda_1^{\beta (p+q')}}  \int\limits_0^{t_1} \int\limits_{t_1-\tau}^{t_2-\tau} \eta^{\alpha q-2} d \eta \|F(\tau, \cdot)\| d \tau.
$$
Using (\ref{eqth38}) and (\ref{esth39}) in the last integral of the above inequality, we get
$$
\left\|\mathcal{I}_1\right\|_{D\left(L^{-\beta q'}\right)}   \leq    \frac{M_2 \left(t_2-t_1\right)^s}{ (1-\alpha q)\lambda_1^{\beta (p+q')}}   \int_0^{t_1}\|F(\tau, \cdot)\|\left(t_1-\tau\right)^{\alpha q -s-1} d \tau
$$
\begin{equation}\label{est48}
  \leq  \frac{M_2 \left(t_2-t_1\right)^s}{ (1-\alpha q)\lambda_1^{\beta (p+q')}} \|F\|_{\mathcal{D}_{2, \alpha q-s}}.
\end{equation}

On the other hand, the inequality (\ref{est41}) also holds for all $0 \leq t_1<t_2 \leq T$. Hence, the same way as in the proof (\ref{est41}) shows that
$$
\left\|\mathcal{I}_2\right\|_{D\left(L^{-\beta q'}\right)}  \leq
\int\limits_{t_1}^{t_2}\left\|\sum_{j=1}^{\infty}   F_j(\tau) (t_2-\tau)^{\alpha-1} E_{\alpha, \alpha}\left(-\lambda_j^\beta (t_2-\tau)^{\alpha}\right) e_j \right\|_{D\left(L^{-\beta q'}\right)} d \tau
$$
$$
 \leq
\int\limits_{t_1}^{t_2}\left\|\sum_{j=1}^{\infty}   F_j(\tau)  E_{\alpha, \alpha}\left(-\lambda_j^\beta (t_2-\tau)^{\alpha}\right) e_j \right\|_{D\left(L^{-\beta q'}\right)} (t_2-\tau)^{\alpha-1} d \tau
$$
$$
 \leq
\int\limits_{t_1}^{t_2}\left\{\left(\sum_{j=1}^{\infty}  \lambda_j^{-\beta q'} F_j(\tau)  E_{\alpha, \alpha}\left(-\lambda_j^\beta (t_2-\tau)^{\alpha}\right) \right)^2 \right\}^{\frac{1}{2}} (t_2-\tau)^{\alpha-1} d \tau
$$
$$
\leq \frac{M_2}{ \lambda_1^{\beta (p+q')}} \int\limits_{t_1}^{t_2}\|F(\tau, \cdot)\|\left(t_2-\tau\right)^{\alpha q-1} d \tau
$$
\begin{equation}\label{est49}
\leq \frac{M_2 \left(t_2-t_1\right)^s}{ \lambda_1^{\beta (p+q')}}  \int\limits_{t_1}^{t_2}\|F(\tau, \cdot)\|\left(t_2-\tau\right)^{\alpha q-s-1} d \tau
\leq
\frac{M_2 \left(t_2-t_1\right)^s}{ \lambda_1^{\beta (p+q')}}\|F\|_{\mathcal{D}_{2, \alpha q-s}}.
\end{equation}

Now, we will establish estimate for   $\mathcal{I}_3$. We have
$$
\left\|\mathcal{I}_3\right\|_{D\left(L^{-\beta q'}\right)}=\left\|L^\beta \sum_{j=1}^{\infty} \Psi_j(T)  \varphi_j  \int\limits_{t_1}^{t_2} \tau^{\alpha-1} E_{\alpha, \alpha}\left(-\lambda_j^\beta \tau^\alpha\right) d \tau e_j(x)\right\|_{D\left(L^{-\beta q'}\right)}
$$
$$
\leq \left\{
\sum_{j=1}^{\infty}  \varphi_j^2 \Psi_j^2(T)\left[\int\limits_{t_1}^{t_2} \tau^{\alpha-1} E_{\alpha, \alpha}\left(-\lambda_j^\beta \tau^\alpha\right) d \tau\right]^2 \lambda_j^{2 \beta} \lambda_j^{-2 \beta q'}\right\}^{1 / 2}
$$
$$
\leq \frac{C_{\kappa} M_2}{\lambda_1^{2 \beta q'}}  \left\{ \sum_{j=1}^{\infty}  \varphi_j^2  \lambda_j^{2 \beta} \left[\int\limits_{t_1}^{t_2} \tau^{\alpha-1} \lambda_j^{-\beta p} \tau^{\alpha p} d \tau\right]^2 \right\}^{1 / 2}
$$
$$
\leq \frac{C_{\kappa} M_2}{\lambda_1^{2 \beta q'}}  \left\{ \sum_{j=1}^{\infty}  \varphi_j^2  \lambda_j^{2 \beta q} \left[\int\limits_{t_1}^{t_2} \tau^{\alpha q-1}  d \tau\right]^2 \right\}^{1 / 2}
$$
$$
\leq \frac{C_{\kappa} M_2}{\alpha q \lambda_1^{2 \beta q'}}   \left\{ \sum_{j=1}^{\infty}  \varphi_j^2  \lambda_j^{2 \beta q} \left[ t_2^{\alpha q}-t_1^{\alpha q}  \right]^2 \right\}^{1 / 2}.
$$
From inequality (\ref{esth39}) and
$$
\left( t_2-t_1 \right)^{\alpha q} = \left( t_2-t_1 \right)^{\alpha q-s} \left( t_2-t_1 \right)^{s} \leq T^{\alpha q-s} \left( t_2-t_1 \right)^{s},
$$
we have
\begin{equation}\label{est50}
\left\|\mathcal{I}_3\right\|_{D\left(L^{-\beta q'}\right)} \leq
\frac{C_{\kappa} M_2 T^{\alpha q-s}}{\alpha q \lambda_1^{2 \beta q'}} \left( t_2-t_1 \right)^{s}
\|\varphi\|_{D\left(L^{\beta q}\right)}.
\end{equation}

Fourthly, we proceed to estimate
$\mathcal{I}_4$. Using the estimation method for $\mathcal{I}_3$, we get the following estimate for $\mathcal{I}_4$
$$
\left\|\mathcal{I}_4\right\|_{D\left(L^{-\beta q'}\right)}
\leq \int\limits_0^T \Bigg\{
\sum_{j=1}^{\infty} \lambda_j^{2 \beta} \lambda_j^{-2 \beta q'} F_j^2(\tau) \Psi_j^2(T)(T-\tau)^{2\alpha-2} \left(E_{\alpha, \alpha}\left(-\lambda_j^\beta (T-\tau)^\alpha\right)\right)^2
$$
$$
\left[\int\limits_{t_1}^{t_2} \eta^{\alpha-1} E_{\alpha, \alpha}\left(-\lambda_j^\beta \eta^\alpha\right) d \eta\right]^2 \Bigg\}^{1 / 2}  d \tau
\leq  \frac{C_{\kappa} M_2}{ \lambda_1^{\beta q'}} \int\limits_0^T \Bigg\{
\sum_{j=1}^{\infty} \lambda_j^{2 \beta} F_j^2(\tau)
$$
$$
 (T-\tau)^{2\alpha-2} \frac{1}{\lambda_j^{2\beta p}}
(T-\tau)^{-2\alpha p}\left[\int\limits_{t_1}^{t_2} \eta^{\alpha-1} \lambda_j^{-\beta q} \eta^{-\alpha q} d \eta\right]^2 \Bigg\}^{1 / 2}  d \tau
$$
$$
\leq  \frac{C_{\kappa} M_2}{ \lambda_1^{\beta q'}} \int\limits_0^T \Bigg\{
\sum_{j=1}^{\infty} \lambda_j^{2 \beta} F_j^2(\tau) (T-\tau)^{2\alpha q-2}
 \frac{T^\alpha}{(\alpha q)^2 \lambda_j^{2\beta}}  \left[t_1^{-\alpha q}-t_2^{-\alpha q} \right]^2 \Bigg\}^{1 / 2}  d \tau
$$
$$
\leq  \frac{ C_{\kappa} M_2 T^\alpha}{\alpha q  \lambda_1^{\beta q'}}  \int\limits_0^T   \|F(\tau, \cdot)\| (T-\tau)^{\alpha q-1}
\frac{t_2^{\alpha q}-t_1^{\alpha q}}{t_1^{\alpha q}t_2^{\alpha q} }    d \tau
$$
\begin{equation}\label{est51}
\leq  \frac{ C_{\kappa} M_2 T^\alpha}{\alpha q  \lambda_1^{ \beta q'}}  \frac{t_2^{\alpha q}-t_1^{\alpha q}}{t_1^{2\alpha q}} \|F\|_{\mathcal{D}_{2, \alpha q}}
\leq  \frac{ C_{\kappa} M_2 T^{\alpha + \alpha q-s}}{\alpha q }  t_1^{-2\alpha q}(t_2 -t_1)^{s}\|F\|_{\mathcal{D}_{2, \alpha q-s}}.
\end{equation}

Using the above method, we get an estimate of $\mathcal{I}_5$

$$
\left\|\mathcal{I}_5\right\|_{D\left(L^{-\beta q'}\right)}
\leq \int\limits_0^T \Bigg\{
\sum_{j=1}^{\infty} \lambda_j^{2 \beta} \lambda_j^{-2 \beta q'} k^2(\tau) u_j^2(\tau) \Psi_j^2(T)(T-\tau)^{2\alpha-2} \left(E_{\alpha, \alpha}\left(-\lambda_j^\beta (T-\tau)^\alpha\right)\right)^2
$$
$$
\left[\int\limits_{t_1}^{t_2} \eta^{\alpha-1} E_{\alpha, \alpha}\left(-\lambda_j^\beta \eta^\alpha\right) d \eta\right]^2 \Bigg\}^{1 / 2}  d \tau
\leq  \frac{C_{\kappa} M_2}{ \lambda_1^{\beta q'}} \int\limits_0^T \Bigg\{
\sum_{j=1}^{\infty} \lambda_j^{2 \beta} k^2(\tau) u_j^2(\tau)
$$
$$
 (T-\tau)^{2\alpha-2} \frac{1}{\lambda_j^{2\beta p}}
(T-\tau)^{-2\alpha p}\left[\int\limits_{t_1}^{t_2} \eta^{\alpha-1} \lambda_j^{-\beta q} \eta^{-\alpha q} d \eta\right]^2 \Bigg\}^{1 / 2}  d \tau
$$
$$
\leq  \frac{C_{\kappa} M_2}{ \lambda_1^{\beta q'}} \int\limits_0^T \Bigg\{
\sum_{j=1}^{\infty} \lambda_j^{2 \beta} k^2(\tau) u_j^2(\tau) (T-\tau)^{2\alpha q-2}
 \frac{T^\alpha}{(\alpha q)^2 \lambda_j^{2\beta}}  \left[t_1^{-\alpha q}-t_2^{-\alpha q} \right]^2 \Bigg\}^{1 / 2}  d \tau
$$
$$
\leq  \frac{ C_{\kappa} M_2 T^\alpha}{\alpha q  \lambda_1^{\beta q'}}  \int\limits_0^T  |k(\tau)|  \|u(\tau, \cdot)\| (T-\tau)^{\alpha q-1}
\frac{t_2^{\alpha q}-t_1^{\alpha q}}{t_1^{\alpha q}t_2^{\alpha q} }    d \tau
$$
\begin{equation*}
\leq  \frac{ C_{\kappa} M_2 T^{\alpha + \alpha q-s} \|k\|_{C[0,T]}}{\alpha q }  t_1^{-2\alpha q}(t_2 -t_1)^{s}\int\limits_0^T  (T-\tau)^{\alpha q-1}    \|u(\tau, \cdot)\|
  d \tau.
\end{equation*}
We note that the Sobolev embedding $D(L^\beta) \hookrightarrow L_2(\Omega)$ hold for $0< \beta< 1.$ Furthermore, based on this embedding, the inequality $\|\cdot\|_{L_2(\Omega)}\leq \|\cdot\|_{D(L^\beta)}$ is satisfied. Using the above inequality, we get the following estimates for $\mathcal{I}_5$
$$
\left\|\mathcal{I}_5\right\|_{D\left(L^{-\beta q'}\right)}  \leq  \frac{ C_{\kappa} M_2 T^{\alpha + \alpha q+\alpha q'} \|k\|_{C[0,T]}}{\alpha q }  t_1^{-2\alpha q}(t_2 -t_1)^{s} B(\alpha q,1-\alpha q')
$$
\begin{equation}\label{ththest515th}
\times\frac{\Phi_1 \left(\|\varphi\|_{D\left(L^{\beta p}\right)}+  \|F\|_{\mathcal{D}_{2, \alpha q-s}}\right)}{2-\exp{\left\{ \frac{\Phi_1 \|k\|_{C[0,T]}T^{\alpha q +\alpha q'}}{ \alpha q +\alpha q'}\right\}}} \exp{\left\{ \frac{\Phi_1 \|k\|_{C[0,T]}T^{\alpha q +\alpha q'}}{ \alpha q +\alpha q'}\right\}}.
\end{equation}

Now let's estimates $\mathcal{I}_6$ and $\mathcal{I}_7$ in the same way:
$$
\left\|\mathcal{I}_6\right\|_{D\left(L^{-\beta q'}\right)}
$$$$
\leq \int_0^{t_1}\left\|\sum_{j=1}^{\infty}   \int\limits_{t_1-\tau}^{t_2-\tau} k(\tau) u_j(\tau) \eta^{\alpha-2} E_{\alpha, \alpha-1}\left(-\lambda_j^\beta \eta^\alpha\right) d \eta d \tau e_j(x)\right\|_{D\left(L^{-\beta q'}\right)}  d \tau
$$
$$
\leq  \frac{M_2 \|k\|_{C[[0,T]}}{ \lambda_1^{\beta (p+q')}}  \int_0^{t_1}\|u(\tau, \cdot)\| d \tau \int\limits_{t_1-\tau}^{t_2-\tau} \eta^{\alpha q-2} d \eta.
$$
$$
\leq  \frac{M_2 \|k\|_{C[[0,T]} \left(t_2-t_1\right)^s}{(1-\alpha q) \lambda_1^{\beta (p+q')}}  \int_0^{t_1}\|u(\tau, \cdot)\| \left(t_1-\tau\right)^{\alpha q -s-1} d \tau,
$$
$$
\left\|\mathcal{I}_7\right\|_{D\left(L^{-\beta q'}\right)}
$$$$
\leq
\int\limits_{t_1}^{t_2}\left\|\sum_{j=1}^{\infty}  k(\tau) u_j(\tau) (t_2-\tau)^{\alpha-1} E_{\alpha, \alpha}\left(-\lambda_j^\beta (t_2-\tau)^{\alpha}\right) e_j(x)\right\|_{D\left(L^{-\beta q'}\right)} d \tau
$$
$$
\leq \frac{M_2 \|k\|_{C[[0,T]}}{ \lambda_1^{\beta (p+q')}} \int\limits_{t_1}^{t_2}\|u(\tau, \cdot)\|\left(t_2-\tau\right)^{\alpha q-s-1}\left(t_2-\tau\right)^{s} d \tau
$$
$$
\leq \frac{M_2 \|k\|_{C[[0,T]} \left(t_2-t_1\right)^{s}}{ \lambda_1^{\beta (p+q')}} \int\limits_{t_1}^{t_2}\|u(\tau, \cdot)\|\left(t_2-\tau\right)^{\alpha q-s-1} d \tau.
$$
We note that the Sobolev embedding $D(L^\beta) \hookrightarrow L_2(\Omega)$ hold for $0< \beta< 1.$ Furthermore, based on this embedding, the inequality $\|\cdot\|_{L_2(\Omega)}\leq \|\cdot\|_{D(L^\beta)}$ is satisfied. Using the above inequality, we get the following estimates for $\mathcal{I}_6$ and $\mathcal{I}_7$
$$
\left\|\mathcal{I}_6\right\|_{D\left(L^{-\beta q'}\right)}   \leq    \frac{ M_2 \|k\|_{C[[0,T]} \left(t_2-t_1\right)^s}{ (1-\alpha q)\lambda_1^{\beta (p+q')}} t_1^{\alpha(q-q') - s}  B(1 - \alpha q', \alpha q- s)
$$
\begin{equation}\label{est52}
\times \frac{\Phi_1 \left(\|\varphi\|_{D\left(L^{\beta p}\right)}+  \|F\|_{\mathcal{D}_{2, \alpha q-s}}\right)}{2-\exp{\left\{ \frac{\Phi_1 \|k\|_{C[0,T]}T^{\alpha q +\alpha q'}}{ \alpha q +\alpha q'}\right\}}} \exp{\left\{ \frac{\Phi_1 \|k\|_{C[0,T]}T^{\alpha q +\alpha q'}}{ \alpha q +\alpha q'}\right\}},
\end{equation}
$$
\left\|\mathcal{I}_7\right\|_{D\left(L^{-\beta q'}\right)}
\leq \frac{ M_2 \|k\|_{C[[0,T]} \left(t_2-t_1\right)^{s}}{ \lambda_1^{\beta (p+q')}} t_2^{\alpha(q- q') - s}  B(1 - \alpha q', \alpha q- s)
$$
\begin{equation}\label{est53}
\times \frac{\Phi_1 \left(\|\varphi\|_{D\left(L^{\beta p}\right)}+  \|F\|_{\mathcal{D}_{2, \alpha q-s}}\right)}{2-\exp{\left\{ \frac{\Phi_1 \|k\|_{C[0,T]}T^{\alpha q +\alpha q'}}{ \alpha q +\alpha q'}\right\}}} \exp{\left\{ \frac{\Phi_1 \|k\|_{C[0,T]}T^{\alpha q +\alpha q'}}{ \alpha q +\alpha q'}\right\}}.
\end{equation}
Using (\ref{est48}) - (\ref{est53}),  we estimate the difference $u\left(t_2, x\right)-u\left(t_1, x\right)$ and obtain
$$
\|u\left(t_2, x\right)-u\left(t_1, x\right)\|_{D\left(L^{-\beta q'}\right)} \leq \sum\limits_{j=1}^{7}\left\|\mathcal{I}_j\right\|_{D\left(L^{-\beta q'}\right)}
$$
$$
\leq \Bigg[
\frac{C_{\kappa} M_2 T^{\alpha q-s}}{\alpha q \lambda_1^{2 \beta q'}}
\|\varphi\|_{D\left(L^{\beta q}\right)}
+\frac{(2-\alpha q)M_2 }{ (1-\alpha q)\lambda_1^{\beta (p+q')}} \|F\|_{\mathcal{D}_{2, \alpha q-s}}
$$
$$
+\frac{ C_{\kappa} M_2 T^{\alpha + \alpha q-s}}{\alpha q }  t_1^{-2\alpha q} \|F\|_{\mathcal{D}_{2, \alpha q-s}}\Bigg](t_2 -t_1)^{s}
$$
$$
+\Bigg(\frac{ C_{\kappa} T^{\alpha + \alpha q+\alpha q'}  }{\alpha q }  t_1^{-2\alpha q}  B(\alpha q,1-\alpha q')
+\frac{t_1^{\alpha(q-q') - s} }{ (1-\alpha q)\lambda_1^{\beta (p+q')}}   B(1 - \alpha q', \alpha q- s)
$$

$$
 +\frac{t_2^{\alpha(q- q') - s} }{ \lambda_1^{\beta (p+q')}}   B(1 - \alpha q', \alpha q- s)\Bigg)M_2 \|k\|_{C[0,T]}(t_2 -t_1)^{s}
$$
\begin{equation}\label{esth54}
\times \frac{\Phi_1 \left(\|\varphi\|_{D\left(L^{\beta p}\right)}+  \|F\|_{\mathcal{D}_{2, \alpha q-s}}\right)}{2-\exp{\left\{ \frac{\Phi_1 \|k\|_{C[0,T]}T^{\alpha q +\alpha q'}}{ \alpha q +\alpha q'}\right\}}} \exp{\left\{ \frac{\Phi_1 \|k\|_{C[0,T]}T^{\alpha q +\alpha q'}}{ \alpha q +\alpha q'}\right\}}.
\end{equation}

We derive from the above that $\|u\left(t_2, x\right)-u\left(t_1, x\right)\|_{D\left(L^{-\beta q'}\right)}$ tends to $0$ as $t_2-t_1$ tends $0$ for $0<t_1<t_2\leq T$.  Thus, $u(t, x)$ belongs to the set $C\left([0,T];D\left(L^{-\beta q'}\right)\right)$.

Now, we divide both sides of inequality (\ref{esth54}) by $\left(t_2-t_1\right)^s$
$$
\frac{\|u\left(t_2, x\right)-u\left(t_1, x\right)\|_{D\left(L^{-\beta q'}\right)}}{\left(t_2-t_1\right)^s}
$$
$$
\leq
\frac{C_{\kappa} M_2 T^{\alpha q-s}}{\alpha q \lambda_1^{2 \beta q'}}
\|\varphi\|_{D\left(L^{\beta q}\right)}
+\frac{(2-\alpha q)M_2 }{ (1-\alpha q)\lambda_1^{\beta (p+q')}} \|F\|_{\mathcal{D}_{2, \alpha q-s}}
$$
$$
+\frac{ C_{\kappa} M_2 T^{\alpha + \alpha q-s}}{\alpha q }  t_1^{-2\alpha q} \|F\|_{\mathcal{D}_{2, \alpha q-s}}
+\Bigg(\frac{ C_{\kappa} T^{\alpha + \alpha q+\alpha q'}  }{\alpha q }  t_1^{-2\alpha q}  B(\alpha q,1-\alpha q')
$$$$
+\frac{t_1^{\alpha(q-q') - s} }{ (1-\alpha q)\lambda_1^{\beta (p+q')}}   B(1 - \alpha q', \alpha q- s)
 +\frac{t_2^{\alpha(q- q') - s} }{ \lambda_1^{\beta (p+q')}}   B(1 - \alpha q', \alpha q- s)\Bigg)M_2 \|k\|_{C[0,T]}
$$
\begin{equation*}
\times \frac{\Phi_1 \left(\|\varphi\|_{D\left(L^{\beta p}\right)}+  \|F\|_{\mathcal{D}_{2, \alpha q-s}}\right)}{2-\exp{\left\{ \frac{\Phi_1 \|k\|_{C[0,T]}T^{\alpha q +\alpha q'}}{ \alpha q +\alpha q'}\right\}}} \exp{\left\{ \frac{\Phi_1 \|k\|_{C[0,T]}T^{\alpha q +\alpha q'}}{ \alpha q +\alpha q'}\right\}}.
\end{equation*}

On the other hand, according to the definition of the norm in the class of H\"{o}lder-continuous functions and  condition $0<s< \min \{\alpha q, 1-\alpha q\}$, we have  $u(t, x)\in C^{s}\left([0,T];D\left(L^{-\beta q'}\right)\right)$, such that
$$
\|u\left(t, x\right)\|_{C^{s}\left([0,T];D\left(L^{-\beta q'}\right)\right)}
\leq
\frac{C_{\kappa} M_2 T^{\alpha q-s}}{\alpha q \lambda_1^{2 \beta q'}}
\|\varphi\|_{D\left(L^{\beta q}\right)}
+\frac{(2-\alpha q)M_2 }{ (1-\alpha q)\lambda_1^{\beta (p+q')}} \|F\|_{\mathcal{D}_{2, \alpha q-s}}
$$
$$
+\frac{ C_{\kappa} M_2 T^{\alpha + \alpha q-s}}{\alpha q }  t_1^{-2\alpha q} \|F\|_{\mathcal{D}_{2, \alpha q-s}}
+\Bigg(\frac{ C_{\kappa} T^{\alpha + \alpha q+\alpha q'}  }{\alpha q }  t_1^{-2\alpha q}  B(\alpha q,1-\alpha q')
$$$$
+\frac{t_1^{\alpha(q-q') - s} }{ (1-\alpha q)\lambda_1^{\beta (p+q')}}   B(1 - \alpha q', \alpha q- s)
 +\frac{t_2^{\alpha(q- q') - s} }{ \lambda_1^{\beta (p+q')}}   B(1 - \alpha q', \alpha q- s)\Bigg)M_2 \|k\|_{C[0,T]}
$$
\begin{equation*}
\times \frac{\Phi_1 \left(\|\varphi\|_{D\left(L^{\beta p}\right)}+  \|F\|_{\mathcal{D}_{2, \alpha q-s}}\right)}{2-\exp{\left\{ \frac{\Phi_1 \|k\|_{C[0,T]}T^{\alpha q +\alpha q'}}{ \alpha q +\alpha q'}\right\}}} \exp{\left\{ \frac{\Phi_1 \|k\|_{C[0,T]}T^{\alpha q +\alpha q'}}{ \alpha q +\alpha q'}\right\}}.
\end{equation*}
From this it follows that the inequality (\ref{theq30}) in Lemma \ref{lemmth5} holds.
\end{proof}

\begin{lemma}\label{lemmth7}
Let conditions A1), A3) and A4) be satisfied.
If $\varphi \in D\left(L^{\beta p}\right)$ and $F \in \mathcal{D}_{2, \alpha q-s}((0,T) \times \Omega)$, then problem (\ref{eqth11}) has a unique solution $u(t, x)$, and $u \in L^{\frac{1}{\alpha q'}-r}\left(0, T ; D\left(L^{\beta\left(p-p'\right)}\right)\right) \cap C^{\alpha q}\left((0, T] ; L_2(\Omega)\right) \cap C^s\left([0, T] ; D\left(L^{-\beta q}\right)\right),$ there exists a positive constant $C_3$ such that
$$
\|u\|_{L^{\frac{1}{a l'}-r}\left(0, T ; D\left(L^{\beta\left(p-p'\right)}\right)\right)}+\|u\|_{C^{\alpha q}\left(0, T ; L_2(\Omega)\right)}
$$
\begin{equation}\label{theq31}
+\|u\|_{C^s\left([0, T] ; D\left(L^{-\beta q'}\right)\right)} \leq C_3\|\varphi\|_{D\left(L^{\beta p}\right)}+C_3\|F\|_{\mathcal{D}_{2, \alpha q-s}}.
\end{equation}
\end{lemma}

The inequality (\ref{theq31}) in Lemma \ref{lemmth7} is obtained by combining the inequalities (\ref{theq28}) - (\ref{theq30}).

\section*{4. Investigating the time fractional derivative of the mild solution}

In the next theorem, we will investigate the time-space fractional derivative of the mild solution $u(t, x)$.
\begin{theorem}\label{theorth1}
Let $p, q, s, p', q', \widehat{p}, \widehat{q}, r, \widehat{r}$ be defined by A1), A3), A4), A5). If $\varphi \in D\left(L^{\beta(p+\widehat{q})}\right)$, and $F \in L^{\frac{1}{\alpha q-s}+\widehat{r}}\left(0, T ; L_2(\Omega)\right)$, then nonlocal initial and boundary problem (\ref{eqth11}) has a unique solution $u$ such that
$$
u \in L^{\frac{1}{\alpha q'}-r}\left(0,T;D\left(L^{\beta(p-p')}\right)\right) \cap C^{\alpha q}\left((0, T] ; L_2(\Omega)\right) \cap C^s\left([0, T]; D\left(L^{-\beta q'}\right) \right),
$$$$
\partial_t^\alpha u \in L^{\frac{1}{\alpha}-\widehat{r}}\left(0, T;D\left(L^{-\beta(q-\widehat{q})}\right) \right)
$$
and
$$
\left\|\partial_t^\alpha u(t, \cdot)\right\|_{D\left(L^{-\beta (q-\widehat{q})}\right)}\leq    M_3 \left\|F(t, \cdot)\right\|
+  T^\alpha\left\|k\right\|_{C[0,T]}
\frac{\Phi \left(\|\varphi\|_{D\left(L^{\beta p}\right)}+  \|F\|_{\mathcal{D}_{2, \alpha q}}\right)}{2-\exp{\left\{ \frac{\Phi \|k\|_{C[0,T]}T^{2\alpha q}}{2\alpha q}\right\}}}
$$
$$
\times \exp{\left\{ \frac{\Phi \|k\|_{C[0,T]}t^{2\alpha q}}{2\alpha q}\right\}}
 + M_2 T^\alpha t^{-\alpha} \|F\|_{\mathcal{D}_{2, \alpha (q-\widehat{q})}}+
\frac{C_{\kappa} M_2}{\lambda_1^\beta} t^{-\alpha}  \|\varphi\|_{D\left(L^{\beta (p+\widehat{q})}\right)}
$$
$$
+
\frac{C_{\kappa} M_2}{\lambda_1^\beta}   t^{-\alpha} \|F\|_{\mathcal{D}_{2, \alpha (q-\widehat{q})}}
+\Bigg(\frac{C_{\kappa} M_2}{\lambda_1^\beta} \|k\|_{C[0,T]}  t^{-\alpha}  + M_2    T^{\alpha q} \|k\|_{C[0,T]}t^{-\alpha \widehat{q}}
\Bigg)
$$
\begin{equation}\label{theq58}
\times \frac{\Phi \left(\|\varphi\|_{D\left(L^{\beta p}\right)}+  \|F\|_{\mathcal{D}_{2, \alpha q}}\right)}{2-\exp{\left\{ \frac{\Phi \|k\|_{C[0,T]}T^{2\alpha q}}{2\alpha q}\right\}}} \exp{\left\{ \frac{\Phi \|k\|_{C[0,T]}t^{2\alpha q}}{2\alpha q}\right\}} B(\alpha(q-\widehat{q}), 1-\alpha q).
\end{equation}
$$
\left\|\partial_t^\alpha u(t, \cdot)\right\|_{L^{\frac{1}{\alpha }-\widehat{r}}\left(0, T ; D\left(L^{-\beta (q-\widehat{q})}\right)\right)}  \leq  M_4  \left\|F\right\|_{L^{\frac{1}{\alpha q-s}+\widehat{r}}\left(0, T ; L_2(\Omega)\right)}+ M_4\|\varphi\|_{D\left(L^{\beta (p+\widehat{q})}\right)}
$$
$$
+ M_4 T^\alpha\left\|k\right\|_{C[0,T]}
\frac{\Phi \left(\|\varphi\|_{D\left(L^{\beta (p+\widehat{q})}\right)}+  \|F\|_{L^{\frac{1}{\alpha q-s}+\widehat{r}}\left(0, T ; L_2(\Omega)\right)} \right) }{2-\exp{\left\{ \frac{\Phi \|k\|_{C[0,T]}T^{2\alpha q}}{2\alpha q}\right\}}}
$$
$$
\times \exp{\left\{ \frac{\Phi \|k\|_{C[0,T]}t^{2\alpha q}}{2\alpha q}\right\}}
+\Bigg(\frac{C_{\kappa} M_2}{\lambda_1^\beta}   t^{-\alpha}  + M_2    T^{\alpha q}
\Bigg)\|k\|_{C[0,T]}B(\alpha(q-\widehat{q}), 1-\alpha q)
$$
\begin{equation}\label{theq59}
\times \frac{\Phi \left(\|\varphi\|_{D\left(L^{\beta (p+\widehat{q})}\right)}
+ 2 \|F\|_{L^{\frac{1}{\alpha q-s}+\widehat{r}}\left(0, T ; L_2(\Omega)\right)}\right)}{2-\exp{\left\{ \frac{\Phi \|k\|_{C[0,T]}T^{2\alpha q}}{2\alpha q}\right\}}} \exp{\left\{ \frac{\Phi \|k\|_{C[0,T]}t^{2\alpha q}}{2\alpha q}\right\}},
\end{equation}
where
$M_4=\max\left\{M_2 T^\alpha, M_3T^\alpha, C_0T^\alpha,\frac{C_{\kappa} M_2}{\lambda_1^\beta}, M_2 C_0T^\alpha B(\alpha(q-\widehat{q}),1-\alpha q) \right\}$.
\end{theorem}

\begin{proof} According to the condition  A5), we obtain $0<\alpha q-s<1$, and $1<\frac{1}{\alpha q-s}<\frac{1}{\alpha q-s}+\widehat{r}$. From the inclusion (\ref{theqt6}) embedding and (\ref{theqt7}), we get
\begin{equation}\label{eqth59}
L^{\frac{1}{\alpha q-s}+\widehat{r}}\left(0, T ; L_2(\Omega)\right) \subset \mathcal{D}_{2, \alpha q-s}((0,T) \times \Omega)
\end{equation}
Moreover, the Sobolev embedding $D\left(L^{\beta(p+\widehat{q})}\right) \hookrightarrow D\left(L^{\beta p}\right)$, holds. Consequently, the assumptions of this theorem also satisfy the conditions of lemmas \ref{lemmth3}-\ref{lemmth5} and Proposion \ref{lemmth7}.
 Therefore, problem (\ref{eqth11}) has a unique solution and
$$
u \in L^{\frac{1}{\alpha q'}-r}\left(0,T;D\left(L^{\beta(p-p')}\right)\right) \cap C^{\alpha q}\left((0, T] ; L_2(\Omega)\right) \cap C^s\left([0, T]; D\left(L^{-\beta q'}\right) \right).
$$

Now, we prove that $\partial_t^\alpha u$ exists and belongs to $L^{\frac{1}{\alpha}-\hat{r}}\left(0, T ; D\left(L^{-\beta(q-\widehat{q})}\right) \right) \cap C^\alpha\left((0, T] ; D\left(L^{-\beta q}\right) \right)$.
Using the identities (iv), (v) of Proposition \ref{proposition2} and equation (\ref{eqth9}), we obtain for $\partial_t^\alpha u_j(t)$:
$$
\partial_t^\alpha u_j(t) =\bigg[\varphi_j- F_j(T) \star \left\{T^{\alpha-1} E_{\alpha, \alpha}\left(-\lambda_j^\beta T^\alpha\right)\right\} \bigg]\Psi_j(T) \partial_t^\alpha E_{\alpha, 1}\left(-\lambda_j^\beta t^\alpha\right)
$$
$$
+\partial_t^\alpha \left[F_j(t) \star \left\{t^{\alpha-1} E_{\alpha, \alpha}\left(-\lambda_j^\beta t^\alpha\right)\right\}\right]-\partial_t^\alpha \left[\{k(t) u_j(t)\} \star \left\{t^{\alpha-1} E_{\alpha, \alpha}\left(-\lambda_j^\beta t^\alpha\right)\right\}\right]
$$
$$
=F_j(t) -k(t) u_j(t) - \lambda_j^\beta  F_j(t) \star \left\{t^{\alpha-1} E_{\alpha, \alpha}\left(-\lambda_j^\beta t^\alpha\right)\right\} - \lambda_j^\beta \varphi_j \Psi_j(T) E_{\alpha, 1}\left(-\lambda_j^\beta t^\alpha\right)
$$
$$
-\lambda_j^\beta F_j(T) \star \left\{T^{\alpha-1} E_{\alpha, \alpha}\left(-\lambda_j^\beta T^\alpha\right)\right\} \Psi_j(T)  E_{\alpha, 1}\left(-\lambda_j^\beta t^\alpha\right)
$$
$$
+\lambda_j^\beta \{k(T)u_j(T) \star \left\{T^{\alpha-1} E_{\alpha, \alpha}\left(-\lambda_j^\beta T^\alpha\right)\right\} \Psi_j(T)  E_{\alpha, 1}\left(-\lambda_j^\beta t^\alpha\right)
$$
$$
 + \lambda_j^\beta  \{k(t) u_j(t)\} \star \left\{t^{\alpha-1} E_{\alpha, \alpha}\left(-\lambda_j^\beta t^\alpha\right)\right\}
$$
$$
=F_j(t)-k(t) u_j(t)+J_{1j}(t)+J_{2j}(t)+J_{3j}(t)+J_{4j}+J_{5j}(t),\,\, j\in \mathbb{N}.
$$

First, we consider the sum $\sum\limits_{n_1\leq j \leq n_2}J_{1j}(t)e_j$, for $n_1,n_2\in \mathbb{N}$, $1\leq n_1 <  n_2.$

According to the definition of the dual space $D\left(L^{-\beta (q-\widehat{q})}\right)$
of $D\left(L^{\beta (q-\widehat{q})}\right),$ and the identity (\ref{eqth5555}) of their dual inner product, we have
$$
\left\|\sum\limits_{n_1\leq j \leq n_2}J_{1j}(t)e_j\right\|_{D\left(L^{-\beta (q-\widehat{q})}\right)}
= \left\|\sum\limits_{n_1\leq j \leq n_2}\lambda_j^\beta  F_j(t) \star \left\{t^{\alpha-1} E_{\alpha, \alpha}\left(-\lambda_j^\beta t^\alpha\right)\right\}e_j\right\|_{D\left(L^{-\beta (q-\widehat{q})}\right)}
$$
$$
\leq \int\limits_0^t\left\|\sum_{n_1 \leq j \leq n_2} \lambda_j^\beta  F_j(\tau)  (t-\tau)^{\alpha-1} E_{\alpha, \alpha}\left(-\lambda_j^\beta (t-\tau)^\alpha\right)  e_j\right\|_{D\left(L^{-\beta (q-\widehat{q})}\right)} d \tau
$$
$$
 \leq  \int\limits_0^t\Bigg\{\sum_{i=1}^{\infty} \lambda_i^{-2 \beta(p-\widehat{q})}\Bigg(\sum_{n_1 \leq j \leq n_2} \lambda_j^\beta  F_j(\tau)
 (t-\tau)^{\alpha-1}
$$$$
\times E_{\alpha, \alpha}\left(-\lambda_j^\beta (t-\tau)^\alpha\right)  e_j, e_i\Bigg)_{-\beta(q-\widehat{q}), \beta(q-\widehat{q})}^2\Bigg\}^{1 / 2} d \tau
 $$$$
\leq   \int\limits_0^t\left\{\sum_{n_1 \leq j \leq n_2} \lambda_j^{2 \beta(p+\widehat{q})}  F_j^2(\tau)  (t-\tau)^{2\alpha-2} \left(E_{\alpha, \alpha}\left(-\lambda_j^\beta (t-\tau)^\alpha\right)\right)^2\right\}^{1 / 2} d \tau .
$$
According to inequality $0<p+\widehat{q}<1$  and  inequalities in Proposition \ref{thproth3}, we get the following from the above inequality
$$
\left\|\sum\limits_{n_1\leq j \leq n_2}J_{1j}(t)e_j\right\|_{D\left(L^{-\beta (q-\widehat{q})}\right)}
$$
$$
\leq M_2  \int\limits_0^t\left\{\sum_{n_1 \leq j \leq n_2} \lambda_j^{2 \beta(p+\widehat{q})}  F_j^2(\tau)  (t-\tau)^{2\alpha-2}  \lambda_j^{-2 \beta(p+\widehat{q})}
(t-\tau)^{-2\alpha(p+\widehat{q})} \right\}^{1 / 2} d \tau
$$
\begin{equation}\label{eqth60}
 = M_2 T^\alpha t^{-\alpha} \int\limits_0^t (t-\tau)^{\alpha(q-\widehat{q})-1}
\left\{\sum_{n_1 \leq j \leq n_2}    F_j^2(\tau)
 \right\}^{1 / 2} d \tau.
\end{equation}

Second, we proceed to establish an estimate for the sum $\sum\limits_{n_1\leq j \leq n_2}J_{2j}(t)e_j$. Using the inequalities in Proposition \ref{thproth3}, we derive the following
$$
\left\|\sum\limits_{n_1\leq j \leq n_2}J_{2j}(t)e_j\right\|_{D\left(L^{-\beta (q-\widehat{q})}\right)} = \left\|\sum\limits_{n_1\leq j \leq n_2}\lambda_j^\beta \varphi_j \Psi_j(T) E_{\alpha, 1}\left(-\lambda_j^\beta t^\alpha\right)e_j\right\|_{D\left(L^{-\beta (q-\widehat{q})}\right)}
$$
$$
 = \left\{\sum_{i=1}^{\infty} \lambda_i^{-2 \beta(p-\widehat{q})}\Bigg(\sum\limits_{n_1\leq j \leq n_2}\lambda_j^\beta \varphi_j \Psi_j(T) E_{\alpha, 1}\left(-\lambda_j^\beta t^\alpha\right)e_j, e_i\Bigg)_{-\beta(q-\widehat{q}), \beta(q-\widehat{q})}^2\right\}^{\frac{1}{2}}
$$
$$
=\left\{\sum_{n_1 \leq j \leq n_2} \lambda_j^{2 \beta(p+\widehat{q})} \varphi_j^2 \Psi_j^2(T) \left(E_{\alpha, 1}\left(-\lambda_j^\beta t^\alpha\right)\right)^2\right\}^{1 / 2}
$$
\begin{equation}\label{eqth61}
 \leq \frac{C_{\kappa} M_2}{\lambda_1^\beta}   \left\{\sum_{n_1 \leq j \leq n_2} \lambda_j^{2 \beta(p+\widehat{q})} \varphi_j^2 t^{-2 \alpha}\right\}^{1 / 2}
\leq \frac{C_{\kappa} M_2}{\lambda_1^\beta} t^{-\alpha}  \left\{\sum_{n_1 \leq j \leq n_2} \lambda_j^{2 \beta(p+\widehat{q})} \varphi_j^2  \right\}^{1 / 2}.
\end{equation}

Third, we estimate the sum  $\sum\limits_{n_1\leq j \leq n_2}J_{3j}(t)e_j$. By methods similar to (\ref{eqth61}), we obtain:
$$
\left\|\sum\limits_{n_1\leq j \leq n_2}J_{3j}(t)e_j\right\|_{D\left(L^{-\beta (q-\widehat{q})}\right)}
$$
$$
= \left\|\sum\limits_{n_1\leq j \leq n_2}\lambda_j^\beta F_j(T) \star \left\{T^{\alpha-1} E_{\alpha, \alpha}\left(-\lambda_j^\beta T^\alpha\right)\right\} \Psi_j(T)  E_{\alpha, 1}\left(-\lambda_j^\beta t^\alpha\right)e_j\right\|_{D\left(L^{-\beta (q-\widehat{q})}\right)}
$$
$$
\leq \int\limits_0^T \Bigg\|\sum_{n_1 \leq j \leq n_2} \lambda_j^\beta  F_j(\tau)  (T-\tau)^{\alpha-1} E_{\alpha, \alpha}\left(-\lambda_j^\beta (T-\tau)^\alpha\right)
$$
$$
\times \Psi_j(T)  E_{\alpha, 1}\left(-\lambda_j^\beta t^\alpha\right)e_j\Bigg\|_{D\left(L^{-\beta (q-\widehat{q})}\right)} d \tau
$$
$$
 \leq  \int\limits_0^T \Bigg\{\sum_{i=1}^{\infty} \lambda_i^{-2 \beta(p-\widehat{q})}\Bigg(\sum_{n_1 \leq j \leq n_2} \lambda_j^\beta  F_j(\tau)  (T-\tau)^{\alpha-1} E_{\alpha, \alpha}\left(-\lambda_j^\beta (T-\tau)^\alpha\right)
$$$$
\times \Psi_j(T)  E_{\alpha, 1}\left(-\lambda_j^\beta t^\alpha\right)e_j, e_i\Bigg)_{-\beta(q-\widehat{q}), \beta(q-\widehat{q})}^2\Bigg\}^{1 / 2} d \tau
 $$
$$
 \leq  \frac{C_{\kappa} M_2}{\lambda_1^\beta}   t^{-\alpha} \int\limits_0^T \Bigg\{\sum_{n_1 \leq j \leq n_2} \lambda_j^{2 \beta(p+\widehat{q})}  F_j^2(\tau)  (T-\tau)^{2\alpha(q-\widehat{q})-2} \lambda_j^{-2 \beta(p+\widehat{q})} \Bigg\}^{1 / 2} d \tau
$$
\begin{equation}\label{eqth62}
\leq  \frac{C_{\kappa} M_2}{\lambda_1^\beta}   t^{-\alpha} \int\limits_0^T (T-\tau)^{\alpha(q-\widehat{q})-1}  \Bigg\{\sum_{n_1 \leq j \leq n_2} F_j^2(\tau)    \Bigg\}^{1 / 2} d \tau.
\end{equation}

Next, we estimate the sum  $\sum\limits_{n_1\leq j \leq n_2}J_{4j}(t)e_j$. By methods similar to (\ref{eqth62}), we obtain:
$$
\left\|\sum\limits_{n_1\leq j \leq n_2}J_{4j}(t)e_j\right\|_{D\left(L^{-\beta (q-\widehat{q})}\right)} = \Bigg\|\sum\limits_{n_1\leq j \leq n_2} \{k(T)u_j(T)\} \star \left\{T^{\alpha-1} E_{\alpha, \alpha}\left(-\lambda_j^\beta T^\alpha\right)\right\}
$$
$$
\times \lambda_j^\beta \Psi_j(T)  E_{\alpha, 1}\left(-\lambda_j^\beta t^\alpha\right)e_j\Bigg\|_{D\left(L^{-\beta (q-\widehat{q})}\right)}
$$
$$
 \leq  \int\limits_0^T \Bigg\{\sum_{i=1}^{\infty} \lambda_i^{-2 \beta(p-\widehat{q})}\Bigg(\sum_{n_1 \leq j \leq n_2} \lambda_j^\beta  k(T)u_j(T) (T-\tau)^{\alpha-1} E_{\alpha, \alpha}\left(-\lambda_j^\beta (T-\tau)^\alpha\right)
$$$$
\times \Psi_j(T)  E_{\alpha, 1}\left(-\lambda_j^\beta t^\alpha\right)e_j, e_i\Bigg)_{-\beta(q-\widehat{q}), \beta(q-\widehat{q})}^2\Bigg\}^{1 / 2} d \tau
 $$
$$
 \leq  \frac{C_{\kappa} M_2}{\lambda_1^\beta}   t^{-\alpha} \int\limits_0^T \Bigg\{\sum_{n_1 \leq j \leq n_2} \lambda_j^{2 \beta(p+\widehat{q})}  k^2(T)u_j^2(T)  (T-\tau)^{2\alpha(q-\widehat{q})-2} \lambda_j^{-2 \beta(p+\widehat{q})} \Bigg\}^{1 / 2} d \tau
$$
\begin{equation*}
\leq  \frac{C_{\kappa} M_2}{\lambda_1^\beta} \|k\|_{C[0,T]}  t^{-\alpha} \int\limits_0^T (T-\tau)^{\alpha(q-\widehat{q})-1}  \Bigg\{\sum_{n_1 \leq j \leq n_2} u_j^2(T)    \Bigg\}^{1 / 2} d \tau.
\end{equation*}
For $0<\tau<T$, we have $u(\tau,\cdot)$  belonging to $L_2(\Omega)$. This follows that the sequence $\left\{\sum_{j \geq n}    u_j^2(\tau)  \right\}^{1 / 2}$, converges pointwise to $0$ as $n$ goes to infinity.  According to the  Sobolev imbedding $D\left(L^{\beta(q+\widehat{q})}\right) \hookrightarrow D\left(L^{\beta p}\right)$ and (\ref{est25}), we have
$$
\left\|\sum\limits_{n_1\leq j \leq n_2}J_{4j}(t)e_j\right\|_{D\left(L^{-\beta (q-\widehat{q})}\right)} \leq  \frac{C_{\kappa} M_2}{\lambda_1^\beta} \|k\|_{C[0,T]}  t^{-\alpha}  \frac{\Phi \left(\|\varphi\|_{D\left(L^{\beta p}\right)}+  \|F\|_{\mathcal{D}_{2, \alpha q}}\right)}{2-\exp{\left\{ \frac{\Phi \|k\|_{C[0,T]}T^{2\alpha q}}{2\alpha q}\right\}}}
$$
\begin{equation}\label{eqth62th}
\times \exp{\left\{ \frac{\Phi \|k\|_{C[0,T]}t^{2\alpha q}}{2\alpha q}\right\}} B(\alpha(q-\widehat{q}), 1-\alpha q).
\end{equation}

Similar to the price of $J_{1j}(t)$ above, we estimate $J_{5j}(t)$
$$
\left\|\sum\limits_{n_1\leq j \leq n_2}J_{5j}(t)e_j\right\|_{D\left(L^{-\beta (q-\widehat{q})}\right)}
$$
$$
= \left\|\sum\limits_{n_1\leq j \leq n_2}\lambda_j^\beta  k(t)u_j(t) \star \left\{t^{\alpha-1} E_{\alpha, \alpha}\left(-\lambda_j^\beta t^\alpha\right)\right\}e_j\right\|_{D\left(L^{-\beta (q-\widehat{q})}\right)}
$$
$$
\leq \int\limits_0^t\left\|\sum_{n_1 \leq j \leq n_2} \lambda_j^\beta k(\tau) u_j(\tau)  (t-\tau)^{\alpha-1} E_{\alpha, \alpha}\left(-\lambda_j^\beta (t-\tau)^\alpha\right)  e_j\right\|_{D\left(L^{-\beta (q-\widehat{q})}\right)} d \tau
$$
$$
 \leq  \int\limits_0^t\Bigg\{\sum_{i=1}^{\infty} \lambda_i^{-2 \beta(p-\widehat{q})}\Bigg(\sum_{n_1 \leq j \leq n_2} \lambda_j^\beta k(\tau) u_j(\tau)
 (t-\tau)^{\alpha-1}
$$$$
\times E_{\alpha, \alpha}\left(-\lambda_j^\beta (t-\tau)^\alpha\right)  e_j, e_i\Bigg)_{-\beta(q-\widehat{q}), \beta(q-\widehat{q})}^2\Bigg\}^{1 / 2} d \tau
 $$$$
\leq   \int\limits_0^t\left\{\sum_{n_1 \leq j \leq n_2} \lambda_j^{2 \beta(p+\widehat{q})}  k^2(\tau) u_j^2(\tau)  (t-\tau)^{2\alpha-2} \left(E_{\alpha, \alpha}\left(-\lambda_j^\beta (t-\tau)^\alpha\right)\right)^2\right\}^{1 / 2} d \tau .
$$
By the inequality $0<p+\widehat{q}<1$  and  inequalities in Proposition \ref{thproth3}, we get
$$
\left\|\sum\limits_{n_1\leq j \leq n_2}J_{5j}(t)e_j\right\|_{D\left(L^{-\beta (q-\widehat{q})}\right)}
$$
$$
\leq M_2 \|k\|_{C[0,T]} \int\limits_0^t\left\{\sum_{n_1 \leq j \leq n_2} \lambda_j^{2 \beta(p+\widehat{q})}  u_j^2(\tau)  (t-\tau)^{2\alpha-2}  \lambda_j^{-2 \beta(p+\widehat{q})}
(t-\tau)^{-2\alpha(p+\widehat{q})} \right\}^{1 / 2} d \tau .
$$
$$
 = M_2  \|k\|_{C[0,T]}    \int\limits_0^t (t-\tau)^{\alpha(q-\widehat{q})-1}
\left\{\sum_{n_1 \leq j \leq n_2}    u_j^2(\tau)
 \right\}^{1 / 2} d \tau.
$$
For $0<\tau<T$, we have $u(\tau,\cdot)$  belonging to $L_2(\Omega)$. This follows that the sequence $\left\{\sum_{j \geq n}    u_j^2(\tau)  \right\}^{1 / 2}$, converges pointwise to $0$ as $n$ goes to infinity.  According to the  Sobolev imbedding $D\left(L^{\beta(q+\widehat{q})}\right) \hookrightarrow D\left(L^{\beta p}\right)$ and (\ref{est25}), we have
$$
\left\|\sum\limits_{n_1\leq j \leq n_2}J_{5j}(t)e_j\right\|_{D\left(L^{-\beta (q-\widehat{q})}\right)}  \leq M_2    T^{\alpha q} \|k\|_{C[0,T]}\frac{\Phi \left(\|\varphi\|_{D\left(L^{\beta p}\right)}+  \|F\|_{\mathcal{D}_{2, \alpha q}}\right)}{2-\exp{\left\{ \frac{\Phi \|k\|_{C[0,T]}T^{2\alpha q}}{2\alpha q}\right\}}}
$$
\begin{equation}\label{eqth63}
\times \exp{\left\{ \frac{\Phi \|k\|_{C[0,T]}T^{2\alpha q}}{2\alpha q}\right\}}  B(\alpha(q-\widehat{q}),1-\alpha q) t^{-\alpha \widehat{q}}.
\end{equation}

Therefore, the dominated convergence theorem yields that
$$
\lim _{n \rightarrow \infty}  \int\limits_0^t (t-\tau)^{\alpha(q-\widehat{q})-1}
\left\{\sum_{ j \leq n}    u_j^2(\tau)
 \right\}^{1 / 2} d \tau =0,
$$
$$
\lim _{n \rightarrow \infty}  \int\limits_0^T (T-\tau)^{\alpha(q-\widehat{q})-1}
\left\{\sum_{ j \leq n}    u_j^2(\tau)
 \right\}^{1 / 2} d \tau =0.
$$

We deduce $
\left\|\sum\limits_{n_1\leq j \leq n_2}J_{ij}(t)e_j\right\|_{D\left(L^{-\beta (q-\widehat{q})}\right)},\,\, i=1,2,3,4,5
$ exist finitely.

For almost every $\tau$ in the interval $(0, T)$, by (\ref{eqth59}), we have that $F(\tau,\cdot)$ belongs to $L_2(\Omega)$. This implies $\sum\limits_{1\leq j \leq n} F_j (\tau)e_j $ is a Cauchy sequence in $L_2(\Omega)$. From the above statements and Sobolev embedding  $L_2(\Omega) \hookrightarrow D\left(L^{-\beta(q-\widehat{q})}\right)$
implies that $\sum\limits_{1\leq j \leq n} F_j (\tau)e_j $  is also a Cauchy sequence in $D\left(L^{-\beta(q-\widehat{q})}\right).$

On the other hand, it follows from $\varphi\in D\left(L^{-\beta(q-\widehat{q})}\right)$
that
$$
\lim\limits_{n_1,n_2\rightarrow \infty} \sum_{n_1 \leq j \leq n_2} \varphi_j^2(\tau)\lambda_i^{2 \beta(p+\widehat{q})}=0.
$$
By A5), $0 \leq \widehat{q} \leq \frac{s}{\alpha}$, and we obtain the inclusion $\mathcal{D}_{2, \alpha q-s}((0,T) \times \Omega) \subset \mathcal{D}_{2, \alpha(q-\widehat{q})}((0,T) \times \Omega)$. We deduce that $F \in \mathcal{D}_{2, \alpha(q-\widehat{q})}((0,T) \times \Omega)$, and by the dominated convergence theorem
$$
\lim _{n_1, n_2 \rightarrow \infty} \int\limits_0^T (T-\tau)^{\alpha(q-\widehat{q})-1}\left\{\sum_{n_1 \leq j \leq n_2} F_j^2(\tau)\right\}^{1 / 2} d \tau=0.
$$

In addition, the Sobolev embedding $L_2(\Omega) \hookrightarrow D\left(L^{-\beta(q-\widehat{q})}\right)$ yields that there exists a positive constant $M_3$ such that
$$
\left\|k u\right\|_{D\left(L^{-\beta (q-\widehat{q})}\right)} \leq M_3 \left\|k u\right\|_{L_2(\Omega)}.
$$
Hence $\sum_{j=1}^n \partial_t^\alpha u_j(t) e_j$ is a Cauchy sequence and a convergent sequence in $D\left(L^{-\beta (q-\widehat{q})}\right)$. Taking all the above arguments together, we conclude that $\left\|\sum_{j=1}^n \partial_t^\alpha u_j(t) e_j\right\|_{D\left(L^{-\beta (q-\widehat{q})}\right)}
$ finitely exists.

Moreover, combining with estimates (\ref{eqth60})-(\ref{eqth63}), we obtain
$$
\left\|\partial_t^\alpha u(t, \cdot)\right\|_{D\left(L^{-\beta (q-\widehat{q})}\right)}\leq  \left\|F(t, \cdot)\right\|_{D\left(L^{-\beta (q-\widehat{q})}\right)}+
 \left\|k(t) u(t, \cdot)\right\|_{D\left(L^{-\beta (q-\widehat{q})}\right)}
$$
$$
+\sum\limits_{i=1}^{5}\left\| \sum\limits_{i=1}^{\infty} J_{ij}(t)e_j \right\|_{D\left(L^{-\beta (q-\widehat{q})}\right)}
\leq M_3 \left\|F(t, \cdot)\right\|
+  T^\alpha\left\|k\right\|_{C[0,T]}
\frac{\Phi \left(\|\varphi\|_{D\left(L^{\beta p}\right)}+  \|F\|_{\mathcal{D}_{2, \alpha q}}\right)}{2-\exp{\left\{ \frac{\Phi \|k\|_{C[0,T]}T^{2\alpha q}}{2\alpha q}\right\}}}
$$
$$
\times \exp{\left\{ \frac{\Phi \|k\|_{C[0,T]}t^{2\alpha q}}{2\alpha q}\right\}}
 + M_2 T^\alpha t^{-\alpha} \|F\|_{\mathcal{D}_{2, \alpha (q-\widehat{q})}}+
\frac{C_{\kappa} M_2}{\lambda_1^\beta} t^{-\alpha}  \|\varphi\|_{D\left(L^{\beta (p+\widehat{q})}\right)}
$$
$$
+
\frac{C_{\kappa} M_2}{\lambda_1^\beta}   t^{-\alpha} \|F\|_{\mathcal{D}_{2, \alpha (q-\widehat{q})}}
+\Bigg(\frac{C_{\kappa} M_2}{\lambda_1^\beta} \|k\|_{C[0,T]}  t^{-\alpha}  + M_2    T^{\alpha q} \|k\|_{C[0,T]}t^{-\alpha \widehat{q}}
\Bigg)
$$
\begin{equation*}
\times \frac{\Phi \left(\|\varphi\|_{D\left(L^{\beta p}\right)}+  \|F\|_{\mathcal{D}_{2, \alpha q}}\right)}{2-\exp{\left\{ \frac{\Phi \|k\|_{C[0,T]}T^{2\alpha q}}{2\alpha q}\right\}}} \exp{\left\{ \frac{\Phi \|k\|_{C[0,T]}t^{2\alpha q}}{2\alpha q}\right\}} B(\alpha(q-\widehat{q}), 1-\alpha q).
\end{equation*}

By performing simple operations in the above inequality, we obtain the inequality (\ref{theq58}) from Theorem \ref{theorth1}.

So,  from $0<\widehat{r} \leq \frac{1-\alpha}{\alpha}$ and $0<\alpha q-s<\alpha$ that $1 \leq \frac{1}{\alpha}-\widehat{r}<\frac{1}{\alpha q-s}+\widehat{r}$. This implies the following Sobolev embedding
$$
L^{\frac{1}{\alpha q-s}+\widehat{r}}\left(0, T ; D\left(L^{-\beta(q-\widehat{q})}\right) \right) \hookrightarrow L^{\frac{1}{\alpha}-\widehat{r}}\left(0, T ; D\left(L^{-\beta(q-\widehat{q})}\right)\right)
$$

Moreover, by the assumption A5), $\widehat{q}\alpha<s$, we have $\frac{1}{\alpha q-s}+\hat{r}>\frac{1}{\alpha(q-\widehat{q})}$. This implies that there exists a constant $M^*>0$ such that
\begin{equation}\label{eqth65}
\left\|F\right\|_{\mathcal{D}_{2, \alpha(q-\widehat{q})}} \leq M^{*}\left\|F \right\|_{L^{\frac{1}{\alpha q-s}+\widehat{r}}}\left(0, T ; L_2(\Omega)\right) .
\end{equation}

Hence, we deduce from (\ref{theq58}) that there exists a constant $M_{4}>0$ satisfying
$$
\left\|\partial_t^\alpha u(t, \cdot)\right\|_{L^{\frac{1}{\alpha }-\widehat{r}}\left(0, T ; D\left(L^{-\beta (q-\widehat{q})}\right)\right)}  \leq  M_4  \left\|F\right\|_{L^{\frac{1}{\alpha q-s}+\widehat{r}}\left(0, T ; L_2(\Omega)\right)}+ M_4\|\varphi\|_{D\left(L^{\beta (p+\widehat{q})}\right)}
$$
$$
+ M_4 T^\alpha\left\|k\right\|_{C[0,T]}
\frac{\Phi \left(\|\varphi\|_{D\left(L^{\beta (p+\widehat{q})}\right)}+  \|F\|_{L^{\frac{1}{\alpha q-s}+\widehat{r}}\left(0, T ; L_2(\Omega)\right)} \right) }{2-\exp{\left\{ \frac{\Phi \|k\|_{C[0,T]}T^{2\alpha q}}{2\alpha q}\right\}}}
$$
$$
\times \exp{\left\{ \frac{\Phi \|k\|_{C[0,T]}t^{2\alpha q}}{2\alpha q}\right\}}
 + M_2 T^\alpha  \|F\|_{\mathcal{D}_{2, \alpha (q-\widehat{q})}}+
\frac{C_{\kappa} M_2}{\lambda_1^\beta}   \|\varphi\|_{D\left(L^{\beta (p+\widehat{q})}\right)}
$$
$$
+
\frac{C_{\kappa} M_2}{\lambda_1^\beta}  \|F\|_{\mathcal{D}_{2, \alpha (q-\widehat{q})}}
+\Bigg(\frac{C_{\kappa} M_2}{\lambda_1^\beta}   t^{-\alpha}  + M_2    T^{\alpha q}
\Bigg)\|k\|_{C[0,T]}B(\alpha(q-\widehat{q}), 1-\alpha q)
$$
\begin{equation*}
\times \frac{\Phi \left(\|\varphi\|_{D\left(L^{\beta (p+\widehat{q})}\right)}
+ 2 \|F\|_{L^{\frac{1}{\alpha q-s}+\widehat{r}}\left(0, T ; L_2(\Omega)\right)}\right)}{2-\exp{\left\{ \frac{\Phi \|k\|_{C[0,T]}T^{2\alpha q}}{2\alpha q}\right\}}} \exp{\left\{ \frac{\Phi \|k\|_{C[0,T]}t^{2\alpha q}}{2\alpha q}\right\}},
\end{equation*}
where we note that $\left\|t^{-\alpha}\right\|_{L^{\frac{1}{\alpha}-\widehat{r}}(0, T ; \mathbb{R})}<\infty$, $\left\|t^{-\alpha \widehat{q}}\right\|_{L^{\frac{1}{\alpha}-\widehat{r}}(0, T ; \mathbb{R})}<\infty$.  This inequality is the inequality (\ref{theq59}) from Theorem \ref{theorth1}.

\end{proof}

\begin{theorem}\label{theorth2}
Let $p, q, s, p', q', \widehat{p}, \widehat{q}, r, \widehat{r}$ be defined by A1), A3), A4), A5). If $\varphi \in D\left(L^{\beta(p+\widehat{q})}\right)$, and $F \in L^{\frac{1}{\alpha q-s}+\widehat{r}}\left(0, T ; L_2(\Omega)\right)$, then nonlocal initial and boundary problem (\ref{eqth11}) has a unique solution $u$ such that
$$
\begin{gathered}
u \in L^{\frac{1}{\alpha q'}-r}\left(0, T ;
D\left(L^{\beta(p-p')}\right)\right) \cap C^{\alpha q}\left((0, T] ; L_2(\Omega)\right) \cap C^s\left([0, T];D\left(L^{-\beta q'}\right)\right), \\
\partial_t^\alpha u \in L^{\frac{1}{\alpha}-\widehat{r}}\left(0, T;D\left(L^{-\beta (q-\widehat{q})}\right)\right) \cap C^\alpha\left((0, T];D\left(L^{-\beta q}\right)\right .
\end{gathered}
$$
and
$$
 \left\|\partial_t^\alpha u \right\|_{C^{\alpha}\left((0, T];D\left(L^{-\beta q}\right)\right)} \leq M_4 \left\|F\right\|_{C^{\alpha}\left((0, T];D\left(L^{-\beta q}\right)\right)}
$$$$
+M_4\left(M^{*}\|\varphi\|_{D\left(L^{\beta (p+\widehat{q})}\right)}+ 2 M^{*}\left\|F\right\|_{L^{\frac{1}{\alpha q-s}+\widehat{r}}\left(0, T ; L_2(\Omega)\right)} \right)
$$
$$
+2M_4 \left\|k\right\|_{C[0,T]}
  \Bigg(M^{*}\|\varphi\|_{D\left(L^{\beta (p+\widehat{q})}\right)}
+ 2 M^{*} \left\|F\right\|_{L^{\frac{1}{\alpha q-s}+\widehat{r}}\left(0, T ; L_2(\Omega)\right)}\Bigg)
$$
\begin{equation}\label{theqth67}
\times Z_{\alpha q,1}\left(\bigg(\frac{M_2}{\lambda_1^{\beta p}} \Gamma(\alpha q) \|k\|_{C[0,T]}\bigg)^{\frac{1}{\alpha q}}T\right).
\end{equation}
Moreover, there exists a positive constant $M_5$ such that
$$
 \left\|\partial_t^\alpha u \right\|_{L^{\frac{1}{\alpha}-\widehat{r}}\left(0, T;D\left(L^{-\beta (q-\widehat{q})}\right)\right)} +\left\|\partial_t^\alpha u \right\|_{C^{\alpha}\left((0, T];D\left(L^{-\beta q}\right)\right)}
$$
$$
\leq M_5 \left\|F\right\|_{C^{\alpha}\left((0, T];D\left(L^{-\beta q}\right)\right)}
+M_5\left(M^{*}\|\varphi\|_{D\left(L^{\beta (p+\widehat{q})}\right)}+ 2 M^{*}\left\|F\right\|_{L^{\frac{1}{\alpha q-s}+\widehat{r}}\left(0, T ; L_2(\Omega)\right)} \right)
$$
\begin{equation}\label{theqth68}
M_5 \left\|k\right\|_{C[0,T]}
\frac{\Phi_1 \left(\|\varphi\|_{D\left(L^{\beta p}\right)}+  \|F\|_{\mathcal{D}_{2, \alpha q-s}}\right)}{2-\exp{\left\{ \frac{\Phi_1 \|k\|_{C[0,T]}T^{\alpha q +\alpha q'}}{ \alpha q +\alpha q'}\right\}}}
\exp{\left\{ \frac{\Phi_1 \|k\|_{C[0,T]}t^{\alpha q +\alpha q'}}{ \alpha q +\alpha q'}\right\}}.
\end{equation}
\end{theorem}

\begin{proof}
We prove $u(t, x) \in  C^{\alpha }\left((0, T];D\left(L^{-\beta q}\right)\right)$. Let us consider $0<t_1<t_2 \leq T$.   Therefore, also, according to Theorem \ref{theorth1}, we prove that for the difference $\partial_t^\alpha u\left(t_2,\cdot\right)- \partial_t^\alpha u\left(t_1, \cdot\right)$ the following holds
\begin{equation}\label{ththeq67}
\lim _{t_2-t_1 \rightarrow 0}\left\|\partial_t^\alpha u\left(t_2,\cdot\right)- \partial_t^\alpha u\left(t_1, \cdot\right)\right\|_{D\left(L^{-\beta q}\right)}=0.
\end{equation}

After making some simple calculations on the difference of the fractional differential, we find:
$$
\partial_t^\alpha u(t_2, x) - \partial_t^\alpha u(t_1, x)
=F(t, x) \Bigg|_{t=t_1}^{t=t_2} -\sum_{j=1}^{\infty} \lambda_j^{\beta} e_j(x)  F_j(t) \star \left\{t^{\alpha-1} E_{\alpha, \alpha}\left(-\lambda_j^\beta t^\alpha\right)\right\} \Bigg|_{t=t_1}^{t=t_2}
$$
$$
-\sum_{j=1}^{\infty} \lambda_j^{\beta} \Psi_j(T) e_j(x) \varphi_j  E_{\alpha, 1}\left(-\lambda_j^\beta t^\alpha\right)\Bigg|_{t=t_1}^{t=t_2}
$$
$$
+\sum_{j=1}^{\infty} \lambda_j^{\beta} \Psi_j(T) e_j(x)  F_j(T) \star \left\{T^{\alpha-1} E_{\alpha, \alpha}\left(-\lambda_j^\beta T^\alpha\right)\right\}  E_{\alpha, 1}\left(-\lambda_j^\beta t^\alpha\right)\Bigg|_{t=t_1}^{t=t_2}
$$
$$
-\sum_{j=1}^{\infty} \lambda_j^{\beta} \Psi_j(T) e_j(x)  \{k(T) u_j(T)\} \star \left\{T^{\alpha-1} E_{\alpha, \alpha}\left(-\lambda_j^\beta T^\alpha\right)\right\}  E_{\alpha, 1}\left(-\lambda_j^\beta t^\alpha\right)\Bigg|_{t=t_1}^{t=t_2}
$$
$$
- \{k(t) u(t, x)\}\Bigg|_{t=t_1}^{t=t_2} +\sum_{j=1}^{\infty} \lambda_j^\beta e_j(x) \{k(t) u_j(t)\} \star \left\{t^{\alpha-1} E_{\alpha, \alpha}\left(-\lambda_j^\beta t^\alpha\right)\right\}\Bigg|_{t=t_1}^{t=t_2}.
$$
$$
=F(t_2, x) - F(t_1, x)-(k(t_2)u(t_2, x) - k(t_1)u(t_1, x))
$$
$$
-\sum_{j=1}^{\infty} \lambda_j^\beta \int\limits_0^{t_1} \int\limits_{t_1-\tau}^{t_2-\tau} F_j(\tau) \eta^{\alpha-2} E_{\alpha, \alpha-1}\left(-\lambda_j^\beta \eta^\alpha\right) d \eta d \tau e_j(x)
$$
$$
-\sum_{j=1}^{\infty} \lambda_j^\beta \int\limits_{t_1}^{t_2} F_j(\tau) (t_2-\tau)^{\alpha-1} E_{\alpha, \alpha}\left(-\lambda_j^\beta (t_2-\tau)^{\alpha}\right) d \tau  e_j(x)
$$
$$
+\sum_{j=1}^{\infty} \lambda_j^{2\beta} \Psi_j(T)  \varphi_j  \int\limits_{t_1}^{t_2} \tau^{\alpha-1} E_{\alpha, \alpha}\left(-\lambda_j^\beta \tau^\alpha\right) d \tau e_j(x)
$$
$$
-\sum_{j=1}^{\infty} \lambda_j^{2\beta} \Psi_j(T)   F_j(T) \star \left\{T^{\alpha-1} E_{\alpha, \alpha}\left(-\lambda_j^\beta T^\alpha\right)\right\}  \int\limits_{t_1}^{t_2} \tau^{\alpha-1} E_{\alpha, \alpha}\left(-\lambda_j^\beta \tau^\alpha\right) d \tau e_j(x)
$$
$$
+\sum_{j=1}^{\infty} \lambda_j^{2\beta} \Psi_j(T) \{k(T) u_j(T)\} \star \left\{T^{\alpha-1} E_{\alpha, \alpha}\left(-\lambda_j^\beta T^\alpha\right)\right\}  \int\limits_{t_1}^{t_2} \tau^{\alpha-1} E_{\alpha, \alpha}\left(-\lambda_j^\beta \tau^\alpha\right) d \tau e_j(x)
$$
$$
+\sum_{j=1}^{\infty} \lambda_j^\beta \int\limits_0^{t_1} \int\limits_{t_1-\tau}^{t_2-\tau} k(\tau) u_j(\tau) \eta^{\alpha-2} E_{\alpha, \alpha-1}\left(-\lambda_j^\beta \eta^\alpha\right) d \eta d \tau e_j(x)
$$
$$
+\sum_{j=1}^{\infty} \lambda_j^\beta \int\limits_{t_1}^{t_2} k(\tau) u_j(\tau) (t_2-\tau)^{\alpha-1} E_{\alpha, \alpha}\left(-\lambda_j^\beta (t_2-\tau)^{\alpha}\right) d \tau  e_j(x)
$$
\begin{equation}\label{eqth67}
:
=F(t_2, x) - F(t_1, x)-(k(t_2)u(t_2, x) - k(t_1)u(t_1, x))+\sum\limits_{j=1}^{7}\mathcal{J}_j.
\end{equation}
Here we note that $\mathcal{J}_j=L^\beta \mathcal{I}_j, \,\, j=\overline{1,7}$.

Since $F(t, x)\in  C^{\alpha }\left((0, T];D\left(L^{-\beta q}\right)\right)$, we have just to prove $\left\|\mathcal{J}_j\right\|_{D\left(L^{-\beta q}\right)}$ approaches 0 as $t_2-t_1$ approaches 0 .

 Let us first consider $\left\|\mathcal{J}_1\right\|_{D\left(L^{-\beta q}\right)}$. Using inequalities in Proposition \ref{thproth3} we obtain:
$$
\left\|\mathcal{J}_1\right\|_{D\left(L^{-\beta q}\right)}  \leq \int_0^{t_1}\left\|-\sum_{j=1}^{\infty} \lambda_j^\beta  \int\limits_{t_1-\tau}^{t_2-\tau} F_j(\tau) \eta^{\alpha-2} E_{\alpha, \alpha-1}\left(-\lambda_j^\beta \eta^\alpha\right) d \eta   e_j\right\|_{D\left(L^{-\beta q}\right)} d \tau
$$
$$
 \leq \int_0^{t_1}\left\{\sum_{j=1}^{\infty} \lambda_j^{2 \beta} \lambda_j^{-2 \beta q} F_j^2(\tau)\left|\int\limits_{t_1-\tau}^{t_2-\tau} \eta^{\alpha-2} E_{\alpha, \alpha-1}\left(-\lambda_j^\beta \eta^\alpha\right)|d \eta\right|^2\right\}^{1 / 2} d \tau
$$$$
  \leq M_2 \int_0^{t_1}\left\{\sum_{j=1}^{\infty} \lambda_j^{2 \beta} \lambda_j^{-2 \beta q} F_j^2(\tau)\left|\int\limits_{t_1-\tau}^{t_2-\tau} \eta^{\alpha-2} \lambda_j^{-\beta p} \eta^{-\alpha p} d \eta\right|^2\right\}^{1 / 2} d \tau
$$

$$
\leq  M_2 \int\limits_0^{t_1} \int\limits_{t_1-\tau}^{t_2-\tau} \eta^{\alpha q-2} d \eta \|F(\tau, \cdot)\| d \tau.
$$
Using (\ref{eqth38}), (\ref{esth39}) and (\ref{eqth59}) in the last integral of the above inequality, we get
$$
\left\|\mathcal{J}_1\right\|_{D\left(L^{-\beta q}\right)}   \leq    \frac{M_2 \left(t_2-t_1\right)^s}{ 1-\alpha q}   \int_0^{t_1}\|F(\tau, \cdot)\|\left(t_1-\tau\right)^{\alpha q -s-1} d \tau
$$
\begin{equation}\label{est67}
  \leq  \frac{M_2 }{ 1-\alpha q} \|F\|_{\mathcal{D}_{2, \alpha q-s}}\left(t_2-t_1\right)^s.
\end{equation}

Let us secondly consider $\left\|\mathcal{J}_2\right\|_{D\left(L^{-\beta q}\right)}$. We have

$$
\left\|\mathcal{J}_2\right\|_{D\left(L^{-\beta q}\right)}  \leq
$$$$
\int\limits_{t_1}^{t_2}\left\|\sum_{j=1}^{\infty} \lambda_j^\beta  F_j(\tau) (t_2-\tau)^{\alpha-1} E_{\alpha, \alpha}\left(-\lambda_j^\beta (t_2-\tau)^{\alpha}\right)  e_j\right\|_{D\left(L^{-\beta q}\right)} d \tau
$$$$
\leq M_2 \int\limits_{t_1}^{t_2}\left\{\sum_{j=1}^{\infty} \lambda_j^{2 \beta} \lambda_j^{-2 \beta q} F_j^2(\tau) \lambda_j^{-2 \beta p}\left(t_2-\tau\right)^{-2 \alpha p}\right\}^{1 / 2}\left(t_2-\tau\right)^{\alpha-1} d \tau
$$
\begin{equation}\label{est68}
\leq M_2 \int\limits_{t_1}^{t_2}\|F(\tau, \cdot)\|\left(t_2-\tau\right)^{\alpha q-s-1}\left(t_2-\tau\right)^s d \tau \leq M_2 \|F\|_{\mathcal{D}_{2, \alpha q-s}}\left(t_2-t_1\right)^s.
\end{equation}

Thirdly, we consider the norm $\left\|\mathcal{J}_3\right\|_{D\left(L^{-\beta q}\right)} $. Hence, we have
$$
\left\|\mathcal{J}_3\right\|_{D\left(L^{-\beta q}\right)}=\left\| \sum_{j=1}^{\infty} \lambda_j^{2\beta} \Psi_j(T)  \varphi_j  \int\limits_{t_1}^{t_2} \tau^{\alpha-1} E_{\alpha, \alpha}\left(-\lambda_j^\beta \tau^\alpha\right) d \tau e_j(x)\right\|_{D\left(L^{-\beta q}\right)}
$$
$$
\leq \left\{
\sum_{j=1}^{\infty}  \varphi_j^2 \Psi_j^2(T) \lambda_j^{4 \beta} \lambda_j^{-2 \beta q} \left[\int\limits_{t_1}^{t_2} \tau^{\alpha-1} E_{\alpha, \alpha}\left(-\lambda_j^\beta \tau^\alpha\right) d \tau\right]^2 \right\}^{1 / 2}
$$
$$
\leq C_{\kappa} M_2  \left\{ \sum_{j=1}^{\infty}  \varphi_j^2  \lambda_j^{4 \beta} \lambda_j^{-2 \beta q} \left[\int\limits_{t_1}^{t_2} \tau^{\alpha-1} \lambda_j^{-\beta } \tau^{-\alpha} d \tau\right]^2 \right\}^{1 / 2}
$$
$$
\leq C_{\kappa} M_2 t_1^{-\alpha} \left\{ \sum_{j=1}^{\infty}  \varphi_j^2  \lambda_j^{2 \beta p} \left[\int\limits_{t_1}^{t_2} \tau^{\alpha -1}  d \tau\right]^2 \right\}^{1 / 2}
$$
$$
\leq \frac{C_{\kappa} M_2 }{\alpha}   t_1^{-\alpha} \left\{ \sum_{j=1}^{\infty}  \varphi_j^2  \lambda_j^{2 \beta p} \left[ t_2^{\alpha }-t_1^{\alpha}  \right]^2 \right\}^{1 / 2}
$$
\begin{equation}\label{est69}
\leq \frac{C_{\kappa} M_2 }{\alpha}   t_1^{-\alpha} \left( t_2^{\alpha }-t_1^{\alpha} \right)
\|\varphi\|_{D\left(L^{\beta p}\right)}.
\end{equation}

Similar to the estimate of $\mathcal{J}_3$, we get the estimate of $\mathcal{J}_4$ and $\mathcal{J}_5$. Using the estimation method for $\mathcal{J}_3$ and inequalities in Proposition \ref{thproth3}, we get
$$
\left\|\mathcal{J}_4\right\|_{D\left(L^{-\beta q}\right)}
\leq \int\limits_0^T \Bigg\{
\sum_{j=1}^{\infty} \lambda_j^{4 \beta} \lambda_j^{-2 \beta q}  F_j^2(\tau) \Psi_j^2(T)(T-\tau)^{2\alpha-2} \left(E_{\alpha, \alpha}\left(-\lambda_j^\beta (T-\tau)^\alpha\right)\right)^2
$$
$$
\left[\int\limits_{t_1}^{t_2} \eta^{\alpha-1} E_{\alpha, \alpha}\left(-\lambda_j^\beta \eta^\alpha\right) d \eta\right]^2 \Bigg\}^{1 / 2}  d \tau
\leq  C_{\kappa} M_2 \int\limits_0^T \Bigg\{
\sum_{j=1}^{\infty} \lambda_j^{4 \beta} \lambda_j^{-2 \beta q}
F_j^2(\tau)
$$
$$
\lambda_j^{-\beta p} (T-\tau)^{2\alpha-2}
(T-\tau)^{-2\alpha p}\left[\int\limits_{t_1}^{t_2} \eta^{\alpha-1} \lambda_j^{-\beta } \eta^{-\alpha} d \eta\right]^2 \Bigg\}^{1 / 2}  d \tau
$$
$$
\leq  \frac{C_{\kappa} M_2}{\alpha  }   t_1^{-\alpha} \int\limits_0^T \Bigg\{
\sum_{j=1}^{\infty} F_j^2(\tau) (T-\tau)^{2\alpha q-2}
 \left[t_2^{\alpha }-t_1^{\alpha } \right]^2 \Bigg\}^{1 / 2}  d \tau
$$
$$
\leq  \frac{C_{\kappa} M_2}{\alpha  }   t_1^{-\alpha}  \int\limits_0^T   \|F(\tau, \cdot)\| (T-\tau)^{\alpha q-1}
\left(t_2^{\alpha }-t_1^{\alpha }\right)    d \tau
$$
\begin{equation}\label{est70}
\leq  \frac{C_{\kappa} M_2}{\alpha  }   t_1^{-\alpha}\left(t_2^{\alpha }-t_1^{\alpha }\right)\|F\|_{\mathcal{D}_{2, \alpha q-s}},
\end{equation}
$$
\left\|\mathcal{J}_5\right\|_{D\left(L^{-\beta q}\right)}
\leq \int\limits_0^T \Bigg\{
\sum_{j=1}^{\infty} \lambda_j^{4 \beta} \lambda_j^{-2 \beta q} k^2(\tau) u_j^2(\tau) \Psi_j^2(T)(T-\tau)^{2\alpha-2}
$$
$$
\times \left(E_{\alpha, \alpha}\left(-\lambda_j^\beta (T-\tau)^\alpha\right)\right)^2\left[\int\limits_{t_1}^{t_2} \eta^{\alpha-1} E_{\alpha, \alpha}\left(-\lambda_j^\beta \eta^\alpha\right) d \eta\right]^2 \Bigg\}^{1 / 2}  d \tau
$$
$$
\leq  C_{\kappa} M_2 \int\limits_0^T \Bigg\{
\sum_{j=1}^{\infty} \lambda_j^{4 \beta} \lambda_j^{-2 \beta q}
k^2(\tau) u_j^2(\tau) \lambda_j^{-\beta p} (T-\tau)^{2\alpha -2}
$$$$
(T-\tau)^{-2\alpha p}\left[\int\limits_{t_1}^{t_2} \eta^{\alpha-1} \lambda_j^{-\beta } \eta^{-\alpha} d \eta\right]^2 \Bigg\}^{1 / 2}  d \tau
$$
$$
\leq  \frac{C_{\kappa} M_2}{\alpha  }   t_1^{-\alpha} \int\limits_0^T \Bigg\{
\sum_{j=1}^{\infty}k^2(\tau) u_j^2(\tau) (T-\tau)^{2\alpha q-2}
 \left[t_2^{\alpha }-t_1^{\alpha } \right]^2 \Bigg\}^{1 / 2}  d \tau
$$
$$
\leq  \frac{C_{\kappa} M_2 \|k\|_{C[0,T]}}{\alpha }   t_1^{-\alpha}  \int\limits_0^T   \|u(\tau, \cdot)\| (T-\tau)^{\alpha q-1}
\left(t_2^{\alpha }-t_1^{\alpha }\right)    d \tau
$$
$$
\leq  \frac{C_{\kappa} M_2  \|k\|_{C[0,T]}T^{\alpha q}}{\alpha^2 q  }   t_1^{-\alpha}\left(t_2^{\alpha }-t_1^{\alpha }\right)\frac{\Phi_1 \left(\|\varphi\|_{D\left(L^{\beta p}\right)}+  \|F\|_{\mathcal{D}_{2, \alpha q-s}}\right)}{2-\exp{\left\{ \frac{\Phi_1 \|k\|_{C[0,T]}T^{\alpha q +\alpha q'}}{ \alpha q +\alpha q'}\right\}}}
$$
\begin{equation}\label{est84est}
\times \exp{\left\{ \frac{\Phi_1 \|k\|_{C[0,T]}t^{\alpha q +\alpha q'}}{ \alpha q +\alpha q'}\right\}}B(1 - \alpha q', \alpha q).
\end{equation}

Now let's estimate $\mathcal{J}_6$ in the same way
$$
\left\|\mathcal{J}_6\right\|_{D\left(L^{-\beta q}\right)}
$$
$$
\leq \int_0^{t_1}\left\|\sum_{j=1}^{\infty} \lambda_j^\beta  \int\limits_{t_1-\tau}^{t_2-\tau} k(\tau) u_j(\tau) \eta^{\alpha-2} E_{\alpha, \alpha-1}\left(-\lambda_j^\beta \eta^\alpha\right) d \eta d \tau e_j(x)\right\|_{D\left(L^{-\beta q}\right)} d \tau
$$
$$
 \leq \int_0^{t_1}\left\{\sum_{j=1}^{\infty} \lambda_j^{2 \beta} \lambda_j^{-2 \beta q}  k^2(\tau)   u_j^2(\tau)\left|\int\limits_{t_1-\tau}^{t_2-\tau} \eta^{\alpha-2} E_{\alpha, \alpha-1}\left(-\lambda_j^\beta \eta^\alpha\right)|d \eta\right|^2\right\}^{1 / 2} d \tau
$$$$
\leq M_2 \|k\|_{C[[0,T]} \int_0^{t_1}\left\{\sum_{j=1}^{\infty}   u_j^2(\tau) \left|\int\limits_{t_1-\tau}^{t_2-\tau} \omega^{\alpha-2}   \omega^{-\alpha p} d \omega\right|^2\right\}^{1 / 2} d \tau
$$

$$
\leq M_2 \|k\|_{C[[0,T]} \int\limits_0^{t_1} \|u(\tau, \cdot)\| \int\limits_{t_1-\tau}^{t_2-\tau} \eta^{\alpha q-2} d \eta  d \tau.
$$
Using (\ref{thineq35}) and (\ref{eqth38}), (\ref{esth39})  in the last integral of the above inequality, we get
$$
\left\|\mathcal{J}_6\right\|_{D\left(L^{-\beta q}\right)}   \leq    \frac{ M_2 \|k\|_{C[[0,T]} \left(t_2-t_1\right)^s}{ (1-\alpha q)} t_1^{\alpha(q - q') - s}  B(1 - \alpha q', \alpha q- s)
$$
\begin{equation}\label{est71}
\times \frac{\Phi_1 \left(\|\varphi\|_{D\left(L^{\beta p}\right)}+  \|F\|_{\mathcal{D}_{2, \alpha q-s}}\right)}{2-\exp{\left\{ \frac{\Phi_1 \|k\|_{C[0,T]}T^{\alpha q +\alpha q'}}{ \alpha q +\alpha q'}\right\}}} \exp{\left\{ \frac{\Phi_1 \|k\|_{C[0,T]}t^{\alpha q +\alpha q'}}{ \alpha q +\alpha q'}\right\}}.
\end{equation}

Let's estimate $\mathcal{J}_7:$
 $$
\left\|\mathcal{J}_7\right\|_{D\left(L^{-\beta q}\right)}  \leq
$$$$
\int\limits_{t_1}^{t_2}\left\|\sum_{j=1}^{\infty} \lambda_j^\beta k(\tau) u_j(\tau) (t_2-\tau)^{\alpha-1} E_{\alpha, \alpha}\left(-\lambda_j^\beta (t_2-\tau)^{\alpha}\right)  e_j\right\|_{D\left(L^{-\beta q}\right)} d \tau
$$$$
\leq M_2 \|k\|_{C[[0,T]} \int\limits_{t_1}^{t_2}\left\{\sum_{j=1}^{\infty} \lambda_j^{2 \beta} \lambda_j^{-2 \beta q} u_j^2(\tau) \lambda_j^{-2 \beta p}\left(t_2-\tau\right)^{-2 \alpha p}\right\}^{1 / 2}\left(t_2-\tau\right)^{\alpha-1} d \tau
$$
\begin{equation*}
\leq M_2  \|k\|_{C[[0,T]}  \int\limits_{t_1}^{t_2}\|u(\tau, \cdot)\|\left(t_2-\tau\right)^{\alpha q-s-1}\left(t_2-\tau\right)^s d \tau
\end{equation*}
\begin{equation*}
\leq M_2  \|k\|_{C[[0,T]} \left(t_2-t_1\right)^s \int\limits_{t_1}^{t_2}\|u(\tau, \cdot)\|\left(t_2-\tau\right)^{\alpha q-s-1}  d \tau
\end{equation*}

$$
\leq  M_2 \|k\|_{C[[0,T]}  t_2^{\alpha(q- q') - s} B(1 - \alpha q', \alpha q- s)\left(t_2-t_1\right)^{s}
$$
\begin{equation}\label{est73}
\times \frac{\Phi_1 \left(\|\varphi\|_{D\left(L^{\beta p}\right)}+  \|F\|_{\mathcal{D}_{2, \alpha q-s}}\right)}{2-\exp{\left\{ \frac{\Phi_1 \|k\|_{C[0,T]}T^{\alpha q +\alpha q'}}{ \alpha q +\alpha q'}\right\}}} \exp{\left\{ \frac{\Phi_1 \|k\|_{C[0,T]}t^{\alpha q +\alpha q'}}{ \alpha q +\alpha q'}\right\}}.
\end{equation}
Now let's estimate the difference norm $k(t_2)u(t_2, x) - k(t_1)u(t_1, x)$. Using the obvious inequality
$$
\|k(t_2)u(t_2, x) - k(t_1)u(t_1, x)\|_{D\left(L^{-\beta q}\right)} \leq \left\|k(t_2)-k(t_1)\right\| \|u(t_2, x)\|_{D\left(L^{-\beta q}\right)}
$$
\begin{equation}\label{est74}
+\|k(t_1)\| \left\|u(t_2, x)-u(t_1, x)\right\|_{D\left(L^{-\beta q}\right)}.
\end{equation}

If we consider (\ref{eqth67}),  and  (\ref{est68}) - (\ref{est74}) together, then it follows that (\ref{ththeq67}). Therefore, $\partial_t^\alpha u(t, x)$ is in the class $ C\left((0, T];D\left(L^{-\beta q}\right)\right)$.

In general, it is easy to see that the estimates (\ref{eqth60})-(\ref{eqth63})  also hold for $\widehat{q}=0$. Hence, we deduce from (\ref{theq58}) and (\ref{eqth65}) that
$$
t^{\alpha} \left\|\partial_t^\alpha u(t, \cdot)\right\|_{D\left(L^{-\beta q}\right)} \leq M_4 \left\|F(t, \cdot)\right\|_{D\left(L^{-\beta q}\right)}
$$$$
+M_4\left(\|\varphi\|_{D\left(L^{\beta p}\right)}+ 2 M^{*}\left\|F\right\|_{L^{\frac{1}{\alpha q-s}+\widehat{r}}\left(0, T ; L_2(\Omega)\right)} \right)
$$
$$
+2M_4 \left\|k\right\|_{C[0,T]}
  \Bigg(\|\varphi\|_{D\left(L^{\beta p}\right)}
+ 2 M^{*} \left\|F\right\|_{L^{\frac{1}{\alpha q-s}+\widehat{r}}\left(0, T ; L_2(\Omega)\right)}\Bigg)
$$
\begin{equation}\label{theq76}
\times Z_{\alpha q,1}\left(\bigg(\frac{M_2}{\lambda_1^{\beta p}} \Gamma(\alpha q) \|k\|_{C[0,T]}\bigg)^{\frac{1}{\alpha q}}T\right).
\end{equation}
From this we get that $\partial_t^\alpha u \in  C^{\alpha}\left((0, T];D\left(L^{-\beta q}\right)\right)$

 In addition, there exists a positive constant $M^{*}$ such that
$$
 \left\|\partial_t^\alpha u \right\|_{C^{\alpha}\left((0, T];D\left(L^{-\beta q}\right)\right)} \leq M_4 \left\|F\right\|_{C^{\alpha}\left((0, T];D\left(L^{-\beta q}\right)\right)}
$$$$
+M_4\left(M^{*}\|\varphi\|_{D\left(L^{\beta (p+\widehat{q})}\right)}+ 2 M^{*}\left\|F\right\|_{L^{\frac{1}{\alpha q-s}+\widehat{r}}\left(0, T ; L_2(\Omega)\right)} \right)
$$
$$
+2M_4 \left\|k\right\|_{C[0,T]}
\frac{\Phi_1 \left(\|\varphi\|_{D\left(L^{\beta p}\right)}+  \|F\|_{\mathcal{D}_{2, \alpha q-s}}\right)}{2-\exp{\left\{ \frac{\Phi_1 \|k\|_{C[0,T]}T^{\alpha q +\alpha q'}}{ \alpha q +\alpha q'}\right\}}}
\exp{\left\{ \frac{\Phi_1 \|k\|_{C[0,T]}t^{\alpha q +\alpha q'}}{ \alpha q +\alpha q'}\right\}}.
$$

From this, we obtain the result (\ref{theqth67}) from Theorem \ref{theorth2}. According to (\ref{theq59}) and (\ref{theqth67}), we obtain the estimate (\ref{theqth68})  in Theorem \ref{theorth2}.
 Theorem \ref{theorth2} is proven.
\end{proof}

\section*{4 Nonlocal initial problem with a Nonlinear source}

In this section, we investigate the existence, uniqueness, and regularity of the mild solutions of problem (\ref{eqth1})--(\ref{eqth3}) corresponding to the nonlinear source function $F(t, x, u (t, x))$. It is appropriate to consider the assumptions that $u(t,\cdot)$ and $F (t,\cdot, u (t,\cdot))$ belong to the same spatial space $L_2(\Omega)$.

We introduce the following assumptions on the numbers $p, q, p', q', \widehat{p}, \widehat{q}, r, \widehat{r}$.

B1) $0<q<p<1$ such that $p+q=1$;

B2) $0<p'<p, \quad q'=1-p', \quad 0<r \leq \frac{1-\alpha q'}{\alpha q'}$;

B3) $0<p' \leq p-q, \quad q'=1-p', \quad 0<r \leq \frac{1-\alpha q'}{\alpha q'}$;

B4) $0 \leq \widehat{q}<q, \quad \widehat{p}=1-\widehat{q}, \quad 0<\widehat{r} \leq \frac{1-\alpha}{\alpha}$.

In our work, we will assume on $F (t,\cdot, u (t,\cdot))$ the following assumptions

C1) $F(t, \cdot, \mathbf{0})=\mathbf{0}$, and there exists a constant $\Upsilon>0$ such that, for all $v_1, v_2 \in L_2(\Omega)$ and $t \in (0,T)$,
$$
\left\|F\left(t, \cdot, v_1\right)-F\left(t, \cdot, v_2\right)\right\| \leq \Upsilon\left\|v_1-v_2\right\| .
$$

C2) $F(t, \cdot, \mathbf{0})=\mathbf{0}$, and there exists a constant $\Upsilon^{*}>0$ such that, for all $v_1, v_2 \in L_2(\Omega)$ and $t_1, t_2 \in J$,
$$
\left\|F\left(t_1, ., v_1\right)-F\left(t_2, ., v_2\right)\right\| \leq \Upsilon^*\left(\left|t_1-t_2\right|+\left\|v_1-v_2\right\|\right).
$$

Note that the assumption C1), C2) imply that, for $v \in L_2(\Omega)$,
\begin{equation}\label{eqth79}
\|F(t, \cdot, v)\| \leq \Upsilon\|v\|.
\end{equation}

We try to develop the ideas of the nonlocal value probleminitial problem with a linear source (\ref{eqth11}) to deal with the nonlocal value problem initial problem with a non-linear source (\ref{eqth1})--(\ref{eqth3}). In Section 3, for the linear function $F(t, x)$ we assumed that
\begin{equation}\label{eqth80}
F \in \mathcal{D}_{2, \alpha q}((0,T) \times \Omega) \text {, or } F \in \mathcal{D}_{2, \alpha q-s}((0,T) \times \Omega) \text {, or } F \in L^{\frac{1}{\alpha q-s}+\widehat{r}}\left(0, T ; L_2(\Omega)\right),
\end{equation}
where $p, q, s, \widehat{r}$ are defined by A1), A3), A5). However, we cannot suppose that the nonlinear source function $F(t, x, u(t, x))$ satisfies the same assumptions as in (\ref{eqth80}), and then find the solution $u$. A natural idea might be to combine the idea of Lemma \ref{lemmth1} with the inequality (\ref{eqth79}), i.e., we predict the solution $u$ may be contained in the set
$$
\mathbf{D}_{\gamma, \eta}^\theta((0,T) \times \Omega):=\left\{w \in \mathcal{D}_{2, \eta}((0,T) \times \Omega): \quad\|w(t, \cdot)\| \leq \theta t^{-\gamma}, \text { for } 0<t \leq T\right\},
$$
for $\theta>0,\,\, 0<\gamma \leq \eta<1$.
The prediction will be proved in the next lemma. However, it is necessary to give some useful notes on $\mathbf{D}_{\gamma, \eta}^\theta((0,T) \times \Omega)$ as follows. For $\mathbf{D}_{\gamma, \eta}^\theta((0,T) \times \Omega)$, we see
$$
\underset{0 \leq t \leq T}{\operatorname{ess} \sup } \int_0^t\|w(\tau, \cdot)\|(t-\tau)^{\eta-1} d \tau \leq \theta \underset{0 \leq t \leq T}{\operatorname{ess} \sup } \int_0^t \tau^{-\gamma}(t-\tau)^{\eta-1} d \tau .
$$
 Hence, we have
\begin{equation}\label{eqth81}
\left\|w\|_{\mathcal{D}_{2, \eta}} \leq \theta T^{\eta-\gamma} B(\eta, 1-\gamma) .\right.
\end{equation}

Moreover, if $\gamma<\eta$, then there always exists a real number $p$ such that $1<\frac{1}{\eta}<p<\frac{1}{\gamma}$. This implies that the function $t^{-\gamma}$ belongs to $L_p\left(0, T ; L_2(\Omega)\right)$. Therefore, we can obtain the following inclusions
$$
\mathbf{D}_{\gamma, \eta}^\theta((0,T) \times \Omega) \subset L_p\left(0, T ; L_2(\Omega)\right) \subset \mathcal{D}_{2, \eta}((0,T) \times \Omega)
$$

We will consider the case $\gamma=\eta$, which we will denote by $\mathbf{D}_{\gamma}^\theta((0,T) \times \Omega):= \mathbf{D}_{\gamma, \gamma}^\theta((0,T) \times \Omega)$.

Now, the Sobolev embedding $D\left(L^{\beta p}\right) \rightarrow L_2(\Omega)$ shows that there exists a positive constant $C_\Omega$ depending on $\Omega, \beta, q$ such that $\|v\| \leq C_\Omega\|v\|_{D\left(L^{\beta p}\right)}$ for all $v \in D\left(L^{\beta p}\right)$.

\section*{4.1 Existence, uniqueness, and regularity of the mild solution of the nonlocal problem with a nonlinear source}

The following lemma holds.

\begin{lemma}\label{lemmth6}
Let $p, q$ be defined by A1). Let $\{w_{(n)}\}$ be defined by $w_{(0)}=\phi$,
$$
w_{(n)}(t, x)=\mathcal{G}_1 (t, x) F(t, \cdot, w_{(n-1)}(t, \cdot,))+\mathcal{G}_2 (t, x) \phi+
\mathcal{G}_3 (t, x) F(t, \cdot, w_{(n-1)}(t, \cdot,))
$$
\begin{equation}\label{eqth82}
-\mathcal{G}_3 (t, x) \{k w_{(n-1)}(t, \cdot,)\}-\mathcal{G}_1 (t, x) \{k(t) w_{(n-1)}(t, \cdot,))\}, \quad n \in \mathbb{N}.
\end{equation}
 If $\phi$ belongs to $D\left(L^{\beta p}\right)$, $F$ satisfies C1), and $\Theta(T)<1$, then
\begin{equation}\label{eqth83}
 \{w_{(n)}\}_{n \geq 0} \subset \mathbf{D}_{\alpha q}^{\chi_0}((0,T) \times \Omega),
\end{equation}
where $\chi_0:=\frac{M_0}{1-\Theta(T)}\|\phi\|_{D\left(L^{\beta p}\right)}$ and
$$
M_0=C_\Omega T^{\alpha q}+ M_2    \Bigg(\frac{T^{\alpha q}}{\lambda_1^{\beta p}} \Upsilon     + C_{\kappa}+ \frac{T^{\alpha q}}{\lambda_1^{\beta p}}C_{\kappa} M_2 \Upsilon M_0  +  T^{\alpha q} \|k\|_{C[0,T]}  \Bigg),
$$
$$
\Theta(T)= \frac{M_2}{\lambda_1^{\beta p}}   \Bigg(T^{\alpha q} \Upsilon     + C_{\kappa}+ C_{\kappa} M_2 \Upsilon  T^{\alpha q}  +  T^{\alpha q} \|k\|_{C[0,T]} + M_2 C_{\kappa} T^{\alpha q} \|k\|_{C[0,T]} \Bigg) B(\alpha q, 1-\alpha q).
$$
\end{lemma}

\begin{proof}
 First, we have
$$
\left\|w_{(0)}\right\|=\|\phi\| \leq C_\Omega\|\phi\|_{D\left(L^{\beta p}\right)}
$$$$
\leq C_\Omega T^{\alpha q}\|\phi\|_{D\left(L^{\beta p}\right)} t^{-\alpha q} \leq M_0 \|\phi\|_{D\left(L^{\beta p}\right)} t^{-\alpha q} \leq \chi_0 t^{-\alpha q}.
$$

So, inequality (\ref{eqth81}) and $\alpha q<1$, imply $w_{(0)} \in \mathbf{D}_{\alpha q}^{\chi_0}((0,T) \times \Omega)$.

Now, we obtain an estimate for $w_{(1)}$
$$
\left\|w_{(1)}\right\|  \leq  \|\mathcal{G}_1 (t, x) F(t, \cdot, w_{(0)}(t, \cdot))\|_{D\left(L^{\beta p}\right)} +
 \|\mathcal{G}_2 (t, x)  w_{(0)}\|_{D\left(L^{\beta p}\right)}
$$$$
 +  \|\mathcal{G}_3 (t, x) F(t, \cdot, w_{(0)}(t, \cdot))\|_{D\left(L^{\beta p}\right)}+  \|\mathcal{G}_3 (t, x) \{k(t) w_{(0)}(t, \cdot)\}\|_{D\left(L^{\beta p}\right)}
$$$$
+
  \|\mathcal{G}_1 (t, x) \{k(t) w_{(0)}(t, \cdot))\}\|_{D\left(L^{\beta p}\right)}.
$$

First, let's get estimates of $\mathcal{G}_i (t, x)  w_{(0)}(t, \cdot),\,\, i=1,2,3$   in $D\left(L^{\beta p}\right)$
$$
\|\mathcal{G}_1 (t, x) F(t, \cdot, w_{(0)}(t, \cdot))\|_{D\left(L^{\beta p}\right)}
$$$$
 \leq M_2  \int\limits_0^t\left\{\sum_{j=1}^{\infty} \frac{1}{\lambda_j^{2\beta p}(t-\tau)^{2\alpha p}}F_j^2(t, \cdot, w_{(0)}(t, \cdot)) (t-\tau)^{2 \alpha-2} \lambda_j^{2\beta p} \right\}^{\frac{1}{2}} d \tau
$$
$$
\leq  \frac{M_2}{\lambda_1^{\beta p}}  \Upsilon \int\limits_0^t (t-\tau)^{ \alpha-1-\alpha p} \|w_{(0)}\|_{D\left(L^{\beta p}\right)}d \tau
$$
$$
\leq  \frac{M_2}{\lambda_1^{\beta p}} \Upsilon   M_0 \|\phi\|_{D\left(L^{\beta p}\right)}\int\limits_0^t (t-\tau)^{ \alpha q-1} \tau^{-\alpha q} d \tau
$$
\begin{equation}\label{eqth84}
\leq \frac{M_2}{\lambda_1^{\beta p}} \Upsilon   M_0 T^{\alpha q}\|\phi\|_{D\left(L^{\beta p}\right)} B(\alpha q, 1-\alpha q) t^{-\alpha q},
\end{equation}
\begin{equation}\label{eqth85}
\|\mathcal{G}_2 (t, x)  \phi\|_{D\left(L^{\beta p}\right)}
 \leq \frac{M_2}{\lambda_1^{\beta p}} C_{\kappa}   M_0 \|\phi\|_{D\left(L^{\beta p}\right)} t^{-\alpha q},
\end{equation}
$$
\|\mathcal{G}_3 (t, x) F(t, \cdot, w_{(0)}(t, \cdot))\|_{D\left(L^{\beta p}\right)}
$$$$
\leq \frac{M_2^2}{\lambda_1^{\beta p}} C_{\kappa}   \int\limits_0^T \left\{\sum_{j=1}^{\infty} F_j^2(t, \cdot, w_{(0)}(t, \cdot))    (T-\tau)^{2\alpha-2-2\alpha p} \lambda_j^{2\beta p}
\right\}^{\frac{1}{2}} d \tau
$$
$$
\leq \frac{M_2^2}{\lambda_1^{\beta p}} C_{\kappa}  \Upsilon M_0\|\phi\|_{D\left(L^{\beta p}\right)}    \int\limits_0^T (T-\tau)^{\alpha q-1} \tau^{-\alpha q}  d \tau
$$
\begin{equation}\label{eqth86}
\leq  \frac{M_2^2}{\lambda_1^{\beta p}} C_{\kappa}   \Upsilon M_0  T^{\alpha q}  \|\phi\|_{D\left(L^{\beta p}\right)}B(\alpha q, 1-\alpha q)  t^{-\alpha q}.
\end{equation}
Similar to $\mathcal{G}_1(t, x) F$, $\mathcal{G}_3(t, x) F$ we estimate $\mathcal{G}_1(t, x)\{k,  u\}$ and $\mathcal{G}_3(t, x)\{k,  u\}$
$$
\|\mathcal{G}_3 (t, x) \{k(t)  w_{(0)}(t, \cdot) \}\|_{D\left(L^{\beta p}\right)}
$$
$$
\leq \int\limits_0^t\left\{\sum_{j=1}^{\infty} \bigg( k(\tau) w_{(0)j} (\tau) E_{\alpha, \alpha} \left(-\lambda_j^\beta(t-\tau)^\alpha\right)(t-\tau)^{ \alpha-1}\bigg)^2 \lambda_j^{2\beta p}\right\}^{\frac{1}{2}}  d \tau
$$
$$
 \leq \frac{M_2}{\lambda_1^{\beta p}} t^{-\alpha q} \|k\|_{C[0,T]} \int\limits_0^t\left\{\sum_{j=1}^{\infty}  w_{(0)j}^2  (t-\tau)^{2 \alpha q -2}\right\}^{\frac{1}{2}} d \tau
$$
$$
\leq   \frac{M_2}{\lambda_1^{\beta p}} t^{-\alpha q}  \|k\|_{C[0,T]} \int\limits_0^t (t-\tau)^{ \alpha q-1} \|w_{(0)}(\tau, \cdot)\| d \tau
$$
\begin{equation}\label{eqth87th}
\leq  \frac{M_2}{\lambda_1^{\beta p}}  \|k\|_{C[0,T]}  M_0 T^{\alpha q}  \|\phi\|_{D\left(L^{\beta p}\right)}B(\alpha q, 1-\alpha q)  t^{-\alpha q}.
\end{equation}
$$
\|\mathcal{G}_3 (t, x) \{k(t)  w_{(0)}(t, \cdot)\}\|_{D\left(L^{\beta p}\right)}
$$$$
\leq \frac{M_2^2}{\lambda_1^{\beta p}} C_{\kappa}   \int\limits_0^T \left\{\sum_{j=1}^{\infty} k^2(t) w_{(0)}^2(t, \cdot)   (T-\tau)^{2\alpha-2-2\alpha p} \lambda_j^{2\beta p}
\right\}^{\frac{1}{2}} d \tau
$$
$$
\leq \frac{M_2^2}{\lambda_1^{\beta p}} C_{\kappa}  \|k\|_{C[0,T]}   M_0\|\phi\|_{D\left(L^{\beta p}\right)}    \int\limits_0^T (T-\tau)^{\alpha q-1} \tau^{-\alpha q}  d \tau
$$
\begin{equation}\label{eqth87}
\leq  \frac{M_2^2}{\lambda_1^{\beta p}} C_{\kappa}    \|k\|_{C[0,T]}  M_0  T^{\alpha q}  \|\phi\|_{D\left(L^{\beta p}\right)}B(\alpha q, 1-\alpha q)  t^{-\alpha q}.
\end{equation}

Finally, from (\ref{eqth84}) - (\ref{eqth87}) we obtain
$$
\left\|w_{(1)}\right\|  \leq \Theta(T)   M_0 \|\phi\|_{D\left(L^{\beta p}\right)}   t^{-\alpha q}.
$$
Then, by using (\ref{eqth81}), we get
\begin{equation}\label{eqth9191}
\left\|w_{(1)}\right\|_{\mathcal{D}_{2, \alpha q}} \leq \Theta(T) M_0\|\phi\|_{D\left(L^{\beta p}\right)}  B(\alpha q, 1-\alpha q)\leq \chi_0  B(\alpha q, 1-\alpha q).
\end{equation}
Hence, inequality (\ref{eqth81}) and $\alpha q<1$, imply $w_{(1)} \in \mathbf{D}_{a q}^{\chi_0}(J \times D)$.

Now, we assume that $w_{(n-1)}$
belongs to $\mathbf{D}_{a q}^{\chi_0}(J \times D)$ for some $n \geq 1$. Then, by using (\ref{eqth81}), we have

\begin{equation}\label{eqth89}
\left\|w_{(n-1)}\right\|_{\mathcal{D}_{2, \alpha q}}  \leq \chi_0  B(\alpha q, 1-\alpha q).
\end{equation}

By induction, for all $\phi \in D\left(L^{\beta p}\right)$, inclusion (\ref{eqth83}) will be proved by showing that $w_{(n)}$ belongs to $\mathbf{D}_{a q}^{\chi_0}(J \times D).$

Now, applying the same reasoning as in the proof of Lemma \ref{lemmth1}, we obtain following estimates.  Using the same line of reasoning as in (\ref{est20}), we deduce
$$
\left\|\mathcal{G}_1(t, \cdot) F(w_{(n-1)})\right\|
$$
$$
\leq \int\limits_0^t\left\{\sum_{j=1}^{\infty} F_j^2(\tau,\cdot,w_{(n-1)}(\tau,\cdot)) E_{\alpha, \alpha}^2\left(-\lambda_j^\beta(t-\tau)^\alpha\right)(t-\tau)^{2 \alpha-2}\right\}^{\frac{1}{2}} d \tau
$$$$
 \leq M_2 \int\limits_0^t\left\{\sum_{j=1}^{\infty} \frac{1}{\lambda_j^{2\beta p}(t-\tau)^{2\alpha p}}F_j^2(\tau,\cdot,w_{(n-1)}(\tau,\cdot))(\tau) (t-\tau)^{2 \alpha-2}\right\}^{\frac{1}{2}} d \tau
$$
$$
\leq \frac{1}{\lambda_1^{\beta p}} M_2 \int\limits_0^t (t-\tau)^{ \alpha q-1} \|F_j(\tau,\cdot,w_{(n-1)}(\tau,\cdot))\|d \tau.
$$
Using (\ref{eqth79}) and (\ref{eqth89}), we obtain the following estimate
\begin{equation}\label{eqth90}
\left\|\mathcal{G}_1(t, \cdot) F(w_{(n-1)})\right\| \leq   \frac{M_2 }{\lambda_1^{\beta p}} \Upsilon \left\|w_{(n-1)}\right\|_{\mathcal{D}_{2, \alpha q}}  \leq \frac{M_2 T^{\alpha q}}{\lambda_1^{\beta p}} \Upsilon  \chi_0  B(\alpha q, 1-\alpha q)t^{-\alpha q}.
\end{equation}

On the other hand,
the norm $\mathcal{G}_2(t, \cdot) \phi$ is estimated by (\ref{est22}), i.e.,
$$
\left\|\mathcal{G}_2(t, \cdot) \phi \right\|
\leq \left\{
\sum_{j=1}^{\infty} \left(\phi_j \Psi_j(T)E_{\alpha, 1}\left(-\lambda_j^\beta t^\alpha\right) \right)^2\right\}^{1 / 2}
$$$$
\leq C_{\kappa} M_2
\left\{\sum_{j=1}^{\infty} \phi_j^2 \frac{1}{\left(1+\lambda_j^\beta t^\alpha\right)^2}\right\}^{1 / 2} \leq C_{\kappa} M_2
\left\{\sum_{j=1}^{\infty} \phi_j^2 \frac{1}{\left(1+\lambda_j^\beta t^\alpha\right)^{2q}}\right\}^{1 / 2}
$$
\begin{equation}\label{eqth91}
\leq C_{\kappa}\frac{M_2 T^{\alpha q}}{\lambda_1^{\beta q}} t^{-\alpha q}
\left\{\sum_{j=1}^{\infty} \phi_j^2 \lambda_j^{2\beta q}\right\}^{1 / 2} \leq M_0  \|\phi\|_{D\left(L^{\beta q}\right)} t^{-\alpha q}.
\end{equation}

The norm $\left\|\mathcal{G}_3(t, x)F\right\|$ can be estimated in the same way as in the proof of (\ref{est23}). That is,
$$
\left\|\mathcal{G}_3(t, \cdot)F(t,\cdot,w_{(n-1)}(t,\cdot))\right\|
$$$$
\leq \frac{C_{\kappa} M_2^2}{\lambda_1^{\beta p}}    \int\limits_0^T \left\{\sum_{j=1}^{\infty} F_j^2(\tau,\cdot,w_{(n-1)}(\tau,\cdot) (T-\tau)^{2\alpha-2-2\alpha p}
\right\}^{\frac{1}{2}} d \tau
$$
$$
\leq \frac{C_{\kappa} M_2^2}{\lambda_1^{\beta p}}    \int\limits_0^T  \| F_j(\tau,\cdot,w_{(n-1)})\|  (T-\tau)^{\alpha q-1}   d \tau.
$$
Once again, we use (\ref{eqth79}) and (\ref{eqth89}) to get
$$
\left\|\mathcal{G}_3(t, \cdot)F(t,\cdot,w_{(n-1)}(t,\cdot))\right\| \leq \frac{C_{\kappa} M_2^2}{\lambda_1^{\beta p}}  \Upsilon   \int\limits_0^T  \| w_{(n-1)}\|  (T-\tau)^{\alpha q-1}   d \tau
$$
\begin{equation}\label{eqth92}
\leq  \frac{C_{\kappa} M_2^2}{\lambda_1^{\beta p}} \Upsilon \left\|w_{(n-1)}\right\|_{\mathcal{D}_{2, \alpha q}}
\leq \frac{C_{\kappa} M_2^2 T^{\alpha q}}{\lambda_1^{\beta p}}  \Upsilon   \chi_0  B(\alpha q, 1-\alpha q)t^{-\alpha q}.
\end{equation}

Similar to $\mathcal{G}_1(t, x) F$, we estimate $\mathcal{G}_1(t, x)\{k, w_{(n-1)}\}$
$$
\left\|\mathcal{G}_1(t, x)\{k(t), w_{(n-1)}(t,\cdot)\}\right\| \leq \frac{M_2}{\lambda_1^{\beta p}}  \|k\|_{C[0,T]} \int\limits_0^t (t-\tau)^{ \alpha q-1} \|w_{(n-1)}(\tau, \cdot)\| d \tau
$$
\begin{equation}\label{eqth93th}
\leq \frac{M_2 }{\lambda_1^{\beta p}}  \|k\|_{C[0,T]} \left\|w_{(n-1)}\right\|_{\mathcal{D}_{2, \alpha q}}
\leq  \frac{M_2 T^{\alpha q}}{\lambda_1^{\beta p}}  \|k\|_{C[0,T]}
 \chi_0  B(\alpha q, 1-\alpha q)t^{-\alpha q}.
\end{equation}
Now, let's estimate $\mathcal{G}_3(t, x)\{k, w_{(n-1)}\}$  the same way we estimate $\mathcal{G}_3(t, \cdot)F(t,\cdot,w_{(n-1)}(t,\cdot))$
$$
\left\|\mathcal{G}_3(t, x)\{k, w_{(n-1)}\}(t,\cdot)\right\| \leq \frac{C_{\kappa} M_2^2}{\lambda_1^{\beta p}}  \|k\|_{C[0,T]}   \int\limits_0^T  \| w_{(n-1)}(\tau, \cdot)\|  (T-\tau)^{\alpha q-1}   d \tau
$$
\begin{equation}\label{eqth93}
\leq  \frac{C_{\kappa} M_2^2}{\lambda_1^{\beta p}} \|k\|_{C[0,T]}  \left\|w_{(n-1)}\right\|_{\mathcal{D}_{2, \alpha q}}
\leq \frac{C_{\kappa} M_2^2 T^{\alpha q}}{\lambda_1^{\beta p}}  \|k\|_{C[0,T]}  \chi_0  B(\alpha q, 1-\alpha q)t^{-\alpha q}.
\end{equation}

We deduce from (\ref{eqth82}) and (\ref{eqth90}) - (\ref{eqth93}) that
\begin{equation}\label{eqthth97}
\left\| w_{(n)}(t, \cdot)\right\|\leq  \Theta(T)   \chi_0   t^{-\alpha q} +M_0  \|\phi\|_{D\left(L^{\beta q}\right)} t^{-\alpha q}.
\end{equation}
Using $\chi_0:=\frac{M_0}{1-\Theta(T)}\|\phi\|_{D\left(L^{\beta p}\right)}$ in inequality (\ref{eqthth97}), we obtain
$$
\left\| w_{(n)}(t, \cdot)\right\|\leq  \Theta(T)   \chi_0   t^{-\alpha q} + \chi_0  (1-\Theta(T)) t^{-\alpha q}  = \chi_0 t^{-\alpha q}.
$$
Consequently, from $\alpha q<1$, we obtain the inclusion $\{w_{(n)}\}_{n \geq 0} \subset \mathbf{D}_{\alpha q}^{\chi_0}((0,T) \times \Omega)$.
\end{proof}

Next, it is necessary to give a definition of mild solutions of nonlocal initial problem (\ref{eqth1})--(\ref{eqth3}).

\begin{definition}\label{defth2}
If a function $u(t, x)$ belongs to $L_p\left(0,T;L_q(\Omega)\right)$, for some $p,\,\, q \geq 1$, and satisfies equation (\ref{eqth10}), then $u(t, x)$ is said to be a mild solution of nonlocal initial problem (\ref{eqth1})--(\ref{eqth3}).
\end{definition}

The following theorem presents existence, uniqueness, and regularity of a mild solution of nonlocal initial problem (\ref{eqth1})--(\ref{eqth3}).

\begin{theorem}\label{theorth3}
Let $p, q, r, p^{\prime}, q^{\prime}$ be defined by B1), B2). If $\varphi \in D\left(L^{\beta q}\right)$, $k\in C[0,T]$, $F$ satisfies C1), and $\Theta(T)<1$, then nonlocal initial problem (\ref{eqth1})--(\ref{eqth3}) has a unique solution $u(t, x)$, such that
$
u \in L^{\frac{1}{\alpha q^{\prime}}-r}\left(0, T ; D\left(L^{\beta (p-p')}\right)\right)
$
and
$$
\|u\|_{L^{\frac{1}{a q^{\prime}}-r}\left(0, T ; D\left(L^{\beta\left(p-p'\right)}\right)\right)} \leq \frac{M_2 }{\lambda_1^{\beta p'}}\Bigg(\Upsilon \chi_0 B(\alpha q,1-\alpha q) T^{\alpha q'}
$$
$$
+C_{\kappa} T^{\alpha q'}+\frac{C_{\kappa}  }{\lambda_1^{\beta q'}} \Upsilon \chi_0  B(\alpha q,1-\alpha q)
$$
\begin{equation}\label{theqth95}
+C_{\kappa} \|k\|_{C[0,T]} \chi_0 B(\alpha q,1-\alpha q)+\|k\|_{C[0,T]}   \chi_0 B(\alpha q,1-\alpha q)  T^{\alpha q'}\Bigg)\|\varphi\|_{D\left(L^{\beta p}\right)}t^{-\alpha q'},
\end{equation}
where $B(\cdot,\cdot)$ is the Euler's Beta function.
\end{theorem}

\begin{proof}
We divide the proof of this theorem into the following two steps.\\
\textit{Step 1.}   We prove the existence and uniqueness of a mild solution. In order to prove the existence of a mild solution of nonlocal initial problem (\ref{eqth1})--(\ref{eqth3}), we will construct a convergent sequence in $L^{\frac{1}{\alpha q}-r}\left(0, T ; L_2(\Omega)\right)$ whose limit will be a mild solution of the problem. Here, $r$ is defined by A2). Let $\left\{w_{(n)}\right\}_{n \geq 0}$ be a sequence defined by Lemma \ref{lemmth6} with respect to $\phi=\varphi \in D\left(L^{\beta q}\right)$, then $\left\{w_{(n)}\right\}_{n \geq 0} \subset \mathbf{D}_{\alpha q}^{\chi_0}((0,T) \times \Omega)$ where  $\chi_0:=\frac{M_0}{1-\Theta(T)}\|\varphi\|_{D\left(L^{\beta p}\right)}$. Therefore, $
\left\|w_{(n)}(t, .)\right\| \leq \chi_0 t^{-\alpha q}, \quad 0<t \leq T
$ for all $n \geq 1$. This together with $t^{-\alpha q}$ belonging to $L^{\frac{1}{\alpha q}-r}(0, T ; \mathbb{R})$ implies that $\left\{w_{(n)}\right\}_{n \geq 0}$ is a bounded sequence in $L^{\frac{1}{\alpha q}-r}\left(0, T ; L_2(\Omega)\right)$.

Now, we will show that $\left\{w_{(n)}\right\}_{n \geq 0}$ is convergent by proving that it is also a Cauchy sequence.  For fixed $n\geq 1$ and $k\geq 1$, the definition (\ref{eqth82}) of $\left\{w_{(n)}\right\}_{n \geq 0}$ yields that
$$
w_{(n+k)}(t, x)-w_{(n)}(t, x)  =\mathcal{G}_1(t, x)\left[F\left(w_{(n-1+k)}\right)-F\left(w_{(n-1)}\right)\right]
$$$$
+\mathcal{G}_3(t, x)\left[F\left(w_{(n-1+k)}\right)-F\left(w_{(n-1)}\right)\right]
$$
$$
-\mathcal{G}_3(t, x)\left[k(t) w_{(n-1+k)}(t, \cdot,))-k(t) w_{(n-1)}(t, \cdot,))\right]
$$
$$
-\mathcal{G}_1(t, x)\left[k(t) w_{(n-1+k)}(t, \cdot,))-k(t) w_{(n-1)}(t, \cdot,))\right] .
$$

Since $F$ satisfies Lemma \ref{lemmth6}, the latter equation shows that we can apply the same arguments as in Lemma \ref{lemmth6} with $\phi=\mathbf{0}$. Hence, by using estimates (\ref{eqth90}) - (\ref{eqth93}), we obtain an estimate for the difference $w_{(1+k)}(t, x)-w_{(0)}(t, x):$
$$
\left\|w_{(1+k)}(t, .)-w_{(1)}(t, .)\right\|
$$
$$
\leq   \left\|\mathcal{G}_1(t, .)\left[F\left(w_{(0+k)}\right)-F\left(w_{(0)}\right)\right]\right\|+\left\|\mathcal{G}_3(t, x)\left[F\left(w_{(0+k)}\right)-F\left(w_{(0)}\right)\right]\right\|
$$
$$
+\left\|\mathcal{G}_3(t, x)\left[k(t) w_{(0+k)}(t, \cdot,))-k(t) w_{(0)}(t, \cdot,))\right]\right\|
$$
$$
+\left\|\mathcal{G}_1(t, x)\left[k(t) w_{(0+k)}(t, \cdot,))-k(t) w_{(0)}(t, \cdot,))\right]\right\|
$$
$$
\leq \frac{M_2}{\lambda_1^{\beta p}}  \int_0^t\left\|F\left(\tau, ., w_{(0+k)}(\tau, \cdot)\right)-F\left(\tau, ., w_{(0)}(\tau, \cdot)\right)\right\|(t-\tau)^{\alpha q-1} d \tau
$$
$$
\frac{C_{\kappa} M_2^2}{\lambda_1^{\beta p}} t^{-\alpha q} \int_0^T\left\|F\left(\tau, ., w_{(0+k)}(\tau, \cdot)\right)-F\left(\tau, ., w_{(0)}(\tau, \cdot)\right)\right\|(T-\tau)^{\alpha q-1} d \tau
$$
$$
\frac{C_{\kappa} M_2^2}{\lambda_1^{\beta p}} \|k\|_{C[0,T]} t^{-\alpha q} \int_0^T \left\|w_{(0+k)}(\tau, \cdot)-w_{(0)}(\tau, \cdot)\right\| (T-\tau)^{\alpha q-1} d \tau
$$
$$
+ \frac{M_2}{\lambda_1^{\beta p}} \|k\|_{C[0,T]} \int_0^t\left\|w_{(0+k)}(\tau, \cdot)-w_{(0)}(\tau, \cdot)\right\|(t-\tau)^{\alpha q-1} d \tau.
$$
 From $\left\{w_{(n)}\right\}_{n \geq 0} \subset \mathbf{D}_{\alpha q}^{\chi_0}((0,T) \times \Omega)$, we have
$$
\left\|w_{(0+k)}(\tau, \cdot)-w_{(0)}(\tau, \cdot)\right\| \leq 2 \chi_0 t^{-\alpha q}.
$$
From this, we find that
$$
\left\|w_{(1+k)}(t, .)-w_{(1)}(t, .)\right\|
$$
$$
\leq \frac{M_2}{\lambda_1^{\beta p}} \Upsilon \int_0^t\left\|u_{(0+k)}(\tau, \cdot)-u_{(0)}(\tau, \cdot)\right\|(t-\tau)^{\alpha q-1} d \tau
$$
$$
+\frac{C_{\kappa} M_2^2}{\lambda_1^{\beta p}} \Upsilon t^{-\alpha q} \int_0^T\left\|u_{(0+k)}(\tau, \cdot)-u_{(0)}(\tau, \cdot)\right\|(T-\tau)^{\alpha q-1} d \tau
$$
$$
+\frac{C_{\kappa} M_2^2}{\lambda_1^{\beta p}} \|k\|_{C[0,T]}  t^{-\alpha q} \int_0^T\left\|u_{(0+k)}(\tau, \cdot)-u_{(0)}(\tau, \cdot)\right\|(T-\tau)^{\alpha q-1} d \tau
$$
$$
+ \frac{M_2}{\lambda_1^{\beta p}} \|k\|_{C[0,T]} \int_0^t\left\|u_{(0+k)}(\tau, \cdot)-u_{(0)}(\tau, \cdot)\right\|(t-\tau)^{\alpha q-1} d \tau \leq  \frac{2 \chi_0 M_2}{\lambda_1^{\beta p}}\Bigg(\Upsilon T^{\alpha q}+ C_{\kappa} M_2 \Upsilon
$$
$$
+ C_{\kappa} M_2\|k\|_{C[0,T]} +\|k\|_{C[0,T]} T^{\alpha q} \Bigg) B(\alpha q,1-\alpha q) t^{-\alpha q}
\leq  2 \chi_0 \Theta(T)t^{-\alpha q},
$$
that is
$$
\left\|w_{(1+k)}(t, .)-w_{(1)}(t, .)\right\| \leq  2 \chi_0 \Theta(T)t^{-\alpha q}.
$$

Iterating this method $n$-times shows

$$
\left\|w_{(n+k)}(t, .)-w_{(n)}(t, .)\right\|
\leq  2 \chi_0 \Theta^n(T)t^{-\alpha q}.
$$
Taking the $L^{\frac{1}{a^q}-r}(0, T ; \mathbb{R})$-norm of both sides of the above inequality directly implies
\begin{equation}\label{eqth95}
\left\|w_{(n+k)}-w_{(n)}\right\|_{L^{\frac{1}{\alpha q}-r}\left(0, T ; L_2(\Omega)\right)} \leq 2 \chi_0 \Theta^n(T) \left\|t^{-\alpha q}\right\|_{L^{\frac{1}{a q}-r}(0, T ; \mathbb{R})}.
\end{equation}

Note that the constants in (\ref{eqth95}) also do not depend on $(n, k)$. Therefore, substituting $n$ into infinity, we get:
$$
\lim\limits_{n, k \rightarrow \infty}\left\|w_{(n+k)}-w_{(n)}\right\|_{L^{\frac{1}{\alpha q}-r}\left(0, T ; L_2(\Omega)\right)}=0
$$
i.e., $\left\{w_{(n)}\right\}_{n \geq 0}$ is a bounded Cauchy sequence in $L^{\frac{1}{\alpha q}-r}\left(0, T ; L_2(\Omega)\right)$. Hence, there exists a function $u(t, x)$ in $L^{\frac{1}{\alpha q}-\mathrm{r}}\left(0, T ; L_2(\Omega)\right)$ such that
$
u=\lim\limits_{n \rightarrow \infty} w_{(n)}, \quad \text { in } L^{\frac{1}{\alpha q}-r}\left(0, T ; L_2(\Omega)\right),
$
and $u(t, x)$ satisfies equation (\ref{eqth10}), i.e., $u(t, x)$ is a mild solution of nonlocal initial problem (\ref{eqth1})--(\ref{eqth3}). At the same time, the boundedness  (\ref{eqthth97}) of $\left\{w_{(n)}\right\}_{n \geq 0}$  gives:
\begin{equation}\label{eqth96}
\|u(t,\cdot)\| \leq  \chi_0  \|\varphi\|_{D\left(L^{\beta p}\right)} t^{-\alpha q}
\end{equation}
and so that
\begin{equation*}
\|u(t,\cdot)\|_{L^{\frac{1}{\alpha q}-r}\left(0, T ; L_2(\Omega)\right)} \leq  \chi_0  \|\varphi\|_{D\left(L^{\beta p}\right)} \|t^{-\alpha q} \|_{L^{\frac{1}{\alpha q}-r}\left(0, T; \mathbb{R}\right)}.
\end{equation*}

Now, we show the uniqueness of the solution $u(t, x)$. Assume that $v(t, x)$ is another solution of nonlocal initial problem (\ref{eqth1})--(\ref{eqth3}).
Then, by applying the same argument as in  (\ref{eqth95}), we get
$$
\left\|u-v\right\|_{L^{\frac{1}{\alpha q}-r}\left(0, T ; L_2(\Omega)\right)} \leq 2 \chi_0 \Theta^n(T) \left\|t^{-\alpha q}\right\|_{L^{\frac{1}{a q}-r}(0, T ; \mathbb{R})},
$$
for all $n \in \mathbb{N}$. Thus $\left\|u-v\right\|_{L^{\frac{1}{\alpha q}-r}\left(0, T ; L_2(\Omega)\right)}=0$ by letting $n$ go to infinity. Hence $u=v$ in $L^{\frac{1}{\alpha q}-r}\left(0, T ; L_2(\Omega)\right).$

\textit{Step 2.} Using inequality (\ref{eqth96}), we prove that $u(t, x)\in L^{\frac{1}{\alpha q}-r}\left(0, T; D\left(L^{\beta (p-p')}\right)\right).$  We now apply the same arguments as in the proofs of (\ref{est2828})-(\ref{est31}) to estimate
$\|u(t,\cdot)\|_{D\left(L^{\beta (p-p')}\right)}$
as follows. First, we get an estimate  for $\mathcal{G}_1(t, \cdot) F(u)$ in $D\left(L^{\beta (p-p')}\right)$
$$
\left\|\mathcal{G}_1(t, \cdot) F\right\|_{D\left(L^{\beta\left(p-p'\right)}\right)}
$$$$
\leq \int\limits_0^t\left\|L^{\beta\left(p-p'\right)}\sum_{j=1}^{\infty} F_j(\tau,\cdot,u(\tau,\cdot)) (t-\tau)^{\alpha-1} E_{\alpha, \alpha}\left(-\lambda_j^\beta (t-\tau)^\alpha\right) e_j\right\| d \tau
$$
$$
 \leq M_2 \int\limits_0^t\left\{\sum_{j=1}^{\infty} \frac{1}{\lambda_j^{2\beta p}(t-\tau)^{2\alpha p}}F_j^2(\tau,\cdot,u(\tau,\cdot)) (t-\tau)^{2 \alpha-2} \lambda_j^{2\beta (p-p')} \right\}^{\frac{1}{2}} d \tau
$$
$$
\leq \frac{1}{\lambda_1^{\beta p'}} M_2 \int\limits_0^t (t-\tau)^{ \alpha q-1} \|F(\tau,\cdot,u(\tau,\cdot))\| d \tau
$$
$$
\leq \frac{M_2}{\lambda_1^{\beta p'}} \Upsilon \chi_0 \|\varphi\|_{D\left(L^{\beta p}\right)} \int\limits_0^t (t-\tau)^{ \alpha q-1} \tau^{-\alpha q} d \tau
$$
$$
\leq \frac{M_2 T^{\alpha q'}}{\lambda_1^{\beta p'}} \Upsilon \chi_0 \|\varphi\|_{D\left(L^{\beta p}\right)} B(\alpha q,1-\alpha q) t^{-\alpha q'}
$$
\begin{equation}\label{eqth97}
\leq \frac{M_2 T^{\alpha q'}}{\lambda_1^{\beta p'}} \Upsilon \chi_0 B(\alpha q,1-\alpha q)
\|\varphi\|_{D\left(L^{\beta p}\right)} t^{-\alpha q'}.
\end{equation}
Next, we have an estimate  for $\mathcal{G}_2(t, \cdot) \varphi$
\begin{equation}\label{eqth98}
\left\|\mathcal{G}_2(t, \cdot) \varphi\right\|_{D\left(L^{\beta\left(p-p'\right)}\right)}
\leq \frac{C_{\kappa} M_2 T^{\alpha q'}}{\lambda_1^{\beta p'}}t^{-\alpha q'}
\|\varphi\|_{D\left(L^{\beta p}\right)}.
\end{equation}
Using the above methods, we estimate $\mathcal{G}_3(t, x)F$ and $\mathcal{G}_3(t, x)\{k, u\}$
$$
\left\|\mathcal{G}_3(t, x)F\right\|_{D\left(L^{\beta\left(p-p'\right)}\right)}=\left\|L^{\beta\left(p-p'\right)}\mathcal{G}_2(t, x) \mathcal{G}_1(T, x)F(\tau,\cdot,u(\tau,\cdot))\right\|
$$
$$
=\Bigg\|L^{\beta\left(p-p'\right)}\sum_{j=1}^{\infty} \Psi_j(T) e_j(x) F_j(T, \cdot,u(T,\cdot)) \star \left\{T^{\alpha-1} E_{\alpha, \alpha}\left(-\lambda_j^\beta T^\alpha\right)\right\}
$$$$
\times E_{\alpha, 1}\left(-\lambda_j^\beta t^\alpha\right)\Bigg\|
\leq \int\limits_0^T\Bigg\|L^{\beta\left(p-p'\right)}\sum_{j=1}^{\infty} F_j(\tau,\cdot,u(\tau,\cdot)) (T-\tau)^{\alpha-1}
$$
$$
\times E_{\alpha, \alpha}\left(-\lambda_j^\beta (T-\tau)^\alpha\right) \Psi_j(T) e_j(x) E_{\alpha, 1}\left(-\lambda_j^\beta t^\alpha\right)\Bigg\| d \tau
$$
$$
\leq \int\limits_0^T\Bigg\{\sum_{j=1}^{\infty} \bigg( F_j(\tau,\cdot,u(\tau,\cdot)) (T-\tau)^{\alpha-1} E_{\alpha, \alpha}\left(-\lambda_j^\beta (T-\tau)^\alpha\right)
$$$$
\times \Psi_j(T)  E_{\alpha, 1}\left(-\lambda_j^\beta t^\alpha\right)\lambda_j^{\beta(p-p')}\bigg)^2\Bigg\}^{\frac{1}{2}} d \tau
$$
$$
\leq \frac{C_{\kappa} M_2}{\lambda_1^{\beta p'+\beta q'}}    \int\limits_0^T \left\{\sum_{j=1}^{\infty} F_j^2(\tau,\cdot,u(\tau,\cdot)) (T-\tau)^{2\alpha-2-2\alpha p} t^{-2\alpha q'}
\right\}^{\frac{1}{2}} d \tau
$$
$$
\leq \frac{C_{\kappa} M_2 t^{-\alpha q'}}{\lambda_1^{\beta}}    \int\limits_0^T  \| F(\tau,\cdot,u(\tau,\cdot))\|_{L_2(\Omega)} (T-\tau)^{\alpha q-1}   d \tau
$$
\begin{equation}\label{eqth99}
\leq \frac{C_{\kappa} M_2 }{\lambda_1^{\beta}}\Upsilon \chi_0 B(\alpha q,1-\alpha q)
\|\varphi\|_{D\left(L^{\beta p}\right)} t^{-\alpha q'},
\end{equation}
$$
\left\|\mathcal{G}_3(t, x)\{k, u\}\right\|_{D\left(L^{\beta\left(p-p'\right)}\right)}=\left\|L^{\beta\left(p-p'\right)}\mathcal{G}_2(t, x) \mathcal{G}_1(T, x)k(\tau)u(\tau,\cdot)\right\|
$$
$$
=\Bigg\|L^{\beta\left(p-p'\right)}\sum_{j=1}^{\infty} \Psi_j(T) e_j(x) k(T) u_j(T, \cdot) \star \left\{T^{\alpha-1} E_{\alpha, \alpha}\left(-\lambda_j^\beta T^\alpha\right)\right\}
$$$$
\times E_{\alpha, 1}\left(-\lambda_j^\beta t^\alpha\right)\Bigg\|
\leq \int\limits_0^T\Bigg\|L^{\beta\left(p-p'\right)}\sum_{j=1}^{\infty} k(T) u_j(T, \cdot) (T-\tau)^{\alpha-1}
$$
$$
\times E_{\alpha, \alpha}\left(-\lambda_j^\beta (T-\tau)^\alpha\right) \Psi_j(T) e_j(x) E_{\alpha, 1}\left(-\lambda_j^\beta t^\alpha\right)\Bigg\| d \tau
$$
$$
\leq \int\limits_0^T\Bigg\{\sum_{j=1}^{\infty} \bigg( k(T) u_j(T, \cdot) (T-\tau)^{\alpha-1} E_{\alpha, \alpha}\left(-\lambda_j^\beta (T-\tau)^\alpha\right)
$$$$
\times \Psi_j(T)  E_{\alpha, 1}\left(-\lambda_j^\beta t^\alpha\right)\lambda_j^{\beta(p-p')}\bigg)^2\Bigg\}^{\frac{1}{2}} d \tau
$$
$$
\leq \frac{C_{\kappa} M_2}{\lambda_1^{\beta p'+\beta q'}}    \int\limits_0^T \left\{\sum_{j=1}^{\infty} k^2(T) u_j^2(T, \cdot) (T-\tau)^{2\alpha-2-2\alpha p} t^{-2\alpha q'}
\right\}^{\frac{1}{2}} d \tau
$$
$$
\leq \frac{C_{\kappa} M_2 t^{-\alpha q'}}{\lambda_1^{\beta}}   \|k\|_{C[0,T]}  \int\limits_0^T  \| u(\tau,\cdot,u(\tau,\cdot))\|_{L_2(\Omega)} (T-\tau)^{\alpha q-1}   d \tau
$$
\begin{equation}\label{eqth99th}
\leq \frac{C_{\kappa} M_2 }{\lambda_1^{\beta}} \|k\|_{C[0,T]} \chi_0 B(\alpha q,1-\alpha q)
\|\varphi\|_{D\left(L^{\beta p}\right)} t^{-\alpha q'}.
\end{equation}

Using the same methods for estimating $\mathcal{G}_1 (t, x) F$, we estimate $\mathcal{G}_1 (t, x) \{k, u\}$
$$
\left\|\mathcal{G}_1(t, x)\{k,  u\}\right\|_{D\left(L^{\beta\left(p-p'\right)}\right)}
 \leq \frac{M_2 \|k\|_{C[0,T]}}{\lambda_1^{\beta p'}}  \int\limits_0^t\left\{\sum_{j=1}^{\infty} \frac{1}{(t-\tau)^{2\alpha p}}u_j^2(\tau) (t-\tau)^{2 \alpha-2}\right\}^{\frac{1}{2}} d \tau
$$
\begin{equation}\label{eqth100}
\leq \frac{M_2 T^{\alpha q'}}{\lambda_1^{\beta p'}}  \|k\|_{C[0,T]}   \chi_0 B(\alpha q,1-\alpha q)
\|\varphi\|_{D\left(L^{\beta p}\right)}t^{-\alpha q'}.
\end{equation}

Based on the above arguments (\ref{eqth97})-(\ref{eqth100}), we arrive at an estimate
\begin{equation*}
\|u(t, x)\|_{D\left(L^{\beta\left(p-p'\right)}\right)}\leq \|\mathcal{G}_1(t, x) F\|_{D\left(L^{\beta\left(p-p'\right)}\right)}
\|\mathcal{G}_2(t, x) \varphi\|_{D\left(L^{\beta\left(p-p'\right)}\right)}
\end{equation*}
$$
+\|\mathcal{G}_3(t, x) F\|_{D\left(L^{\beta\left(p-p'\right)}\right)}+\|\mathcal{G}_3(t, x) \{k,  u\}\|_{D\left(L^{\beta\left(p-p'\right)}\right)}+\|\mathcal{G}_1(t, x)\{k,  u\}\|_{D\left(L^{\beta\left(p-p'\right)}\right)}
$$
$$
\leq \frac{M_2 }{\lambda_1^{\beta p'}}\Bigg(\Upsilon \chi_0 B(\alpha q,1-\alpha q) T^{\alpha q'} +C_{\kappa} T^{\alpha q'}+\frac{C_{\kappa}  }{\lambda_1^{\beta q'}} \Upsilon \chi_0  B(\alpha q,1-\alpha q)
$$
$$
+C_{\kappa} \|k\|_{C[0,T]} \chi_0 B(\alpha q,1-\alpha q)+\|k\|_{C[0,T]}   \chi_0 B(\alpha q,1-\alpha q)  T^{\alpha q'}\Bigg)\|\varphi\|_{D\left(L^{\beta p}\right)}t^{-\alpha q'}.
$$
By taking the $L^{\frac{1}{\alpha q'}-r}(0,T;\mathbb{R})$ norm, then from the last inequalities we get the estimate (\ref{theqth95}).
\end{proof}

\begin{theorem}\label{theorth4}
Let $p, q, r, p^{\prime}, q^{\prime}$ be defined by B1), B2). If $\varphi \in D\left(L^{\beta q}\right)$, $k\in C[0,T]$, $F$ satisfies C1), and $\Theta(T)<1$, then nonlocal initial problem (\ref{eqth1})--(\ref{eqth3}) has a unique solution
$
u \in  C^{\alpha q}\left((0, T] ; L_2(\Omega)\right)
$
and
\begin{equation}\label{thth102}
\|u\|_{C^{\alpha q}\left((0,T];L_2(\Omega)\right)} \leq  \chi_0  \|\varphi\|_{D\left(L^{\beta p}\right)}.
\end{equation}
\end{theorem}

\begin{proof}
We prove that $u \in C^{\alpha q}\left((0, T] ; L_2(\Omega)\right)$. Let us consider $0<t_1<t_2 \leq T$. By the same arguments as in (\ref{eqth37}),  the difference $u(t_2, x) -u(t_1, x)$ can be calculated as
$$
u\left(t_2, x\right)-u\left(t_1, x\right)=\sum_{j=1}^{\infty} \int\limits_0^{t_1} \int\limits_{t_1-\tau}^{t_2-\tau} F_j(\tau,u(\tau)) \eta^{\alpha-2} E_{\alpha, \alpha-1}\left(-\lambda_j^\beta \eta^\alpha\right) d \eta d \tau e_j(x)
$$
$$
+\sum_{j=1}^{\infty} \int\limits_{t_1}^{t_2} F_j(\tau,u(\tau)) (t_2-\tau)^{\alpha-1} E_{\alpha, \alpha}\left(-\lambda_j^\beta (t_2-\tau)^{\alpha}\right) d \tau  e_j(x)
$$
$$
-L^\beta \sum_{j=1}^{\infty} \Psi_j(T)  \varphi_j  \int\limits_{t_1}^{t_2} \tau^{\alpha-1} E_{\alpha, \alpha}\left(-\lambda_j^\beta \tau^\alpha\right) d \tau e_j(x)
$$
$$
-L^\beta\sum_{j=1}^{\infty} \Psi_j(T)   F_j(T,u(T)) \star \left\{T^{\alpha-1} E_{\alpha, \alpha}\left(-\lambda_j^\beta T^\alpha\right)\right\}  \int\limits_{t_1}^{t_2} \tau^{\alpha-1} E_{\alpha, \alpha}\left(-\lambda_j^\beta \tau^\alpha\right) d \tau e_j(x)
$$
$$
+L^\beta\sum_{j=1}^{\infty} \Psi_j(T) \{k(T)  u_j(T)\} \star \left\{T^{\alpha-1} E_{\alpha, \alpha}\left(-\lambda_j^\beta T^\alpha\right)\right\}  \int\limits_{t_1}^{t_2} \tau^{\alpha-1} E_{\alpha, \alpha}\left(-\lambda_j^\beta \tau^\alpha\right) d \tau e_j(x)
$$
$$
+\sum_{j=1}^{\infty} \int\limits_0^{t_1} \int\limits_{t_1-\tau}^{t_2-\tau} k(\tau) u_j(\tau) \eta^{\alpha-2} E_{\alpha, \alpha-1}\left(-\lambda_j^\beta \eta^\alpha\right) d \eta d \tau e_j(x)
$$
\begin{equation}\label{eqth102}
+\sum_{j=1}^{\infty} \int\limits_{t_1}^{t_2} k(\tau) u_j(\tau) (t_2-\tau)^{\alpha-1} E_{\alpha, \alpha}\left(-\lambda_j^\beta (t_2-\tau)^{\alpha}\right) d \tau  e_j(x)
:=\sum\limits_{j=1}^7 \mathcal{R}_j.
\end{equation}

In what follows, we will establish the convergence for $\left\|\mathcal{R}_j\right\|, \,\, j=\overline{1,7}$ which can be treated similarly as in (\ref{est40})-(\ref{est45}) based on the Lipschitzian assumptions C1), C2). We first see that
$$
\left\|\mathcal{R}_1\right\|   \leq \int_0^{t_1}\left\|\sum_{j=1}^{\infty}   \int\limits_{t_1-\tau}^{t_2-\tau} F_j(\tau,u(\tau)) \eta^{\alpha-2} E_{\alpha, \alpha-1}\left(-\lambda_j^\beta \eta^\alpha\right) d \eta d \tau e_j(x)\right\| d \tau
$$
$$
\leq  \frac{M_2}{ \lambda_1^{\beta p}} \int_0^{t_1}\left\{ \sum_{j=1}^{\infty}   F_j^2(\tau,u(\tau)) \left| \int\limits_{t_1-\tau}^{t_2-\tau}  \eta^{\alpha-2}\eta^{-\alpha p} d \eta \right|^2 d \tau  \right\}^{\frac{1}{2}} d \tau
$$
$$
\leq  \frac{M_2}{ \lambda_1^{\beta p}}  \int\limits_0^{t_1} \|F(\tau, u(\tau))\| \int\limits_{t_1-\tau}^{t_2-\tau} \eta^{\alpha q-2} d \eta  d \tau.
$$
$$
\leq   \frac{M_2 }{ (1-\alpha q)\lambda_1^{\beta p}}  \Upsilon \int\limits_0^{t_1} \|u(\tau, \cdot)\| \Big[ (t_1-\tau)^{\alpha q-1} - (t_2-\tau)^{\alpha q-1}\Big]  d \tau.
$$
$$
\leq   \frac{M_2 }{ (1-\alpha q)\lambda_1^{\beta p}}  \Upsilon \int\limits_0^{t_1} \|u(\tau, \cdot)\| \Big[ (t_1-\tau)^{\alpha q-1} - (t_2-\tau)^{\alpha q-1}\Big]  d \tau.
$$
By using (\ref{eqth96}) in the above inequality, we obtain
$$
\left\|\mathcal{R}_1\right\|   \leq   \frac{M_2 }{ (1-\alpha q)\lambda_1^{\beta p}} \chi_0  \Upsilon \|\varphi\|_{D\left(L^{\beta p}\right)}\int\limits_0^{t_1} \tau^{-\alpha q} \Big[ (t_1-\tau)^{\alpha q-1} - (t_2-\tau)^{\alpha q-1}\Big]  d \tau.
$$

Let's study the integral term in the last inequality
$$
\int\limits_0^{t_1} \tau^{-\alpha q}\left[\left(t_1-\tau\right)^{\alpha q-1}-\left(t_2-\tau\right)^{\alpha q-1}\right] d \tau
$$$$
=
\int\limits_0^{t_1} \tau^{-\alpha q} \left(t_1-\tau\right)^{\alpha q-1} d \tau -
\int\limits_0^{t_1} \tau^{-\alpha q} \left(t_2-\tau\right)^{\alpha q-1} d \tau.
$$
The first part of this integral
$$
\int\limits_0^{t_1} \tau^{-\alpha q}\left(t_1-\tau\right)^{\alpha q-1} d \tau=B(\alpha q, 1-\alpha q),
$$
and in the second part of the integral, due to the substitution $\tau= t_2\mu $, we have
$$
\int\limits_0^{t_1} \tau^{-\alpha q} \left(t_2-\tau\right)^{\alpha q-1} d \tau=\int_0^{t_1 / t_2} \mu^{-\alpha q}(1-\mu)^{\alpha q-1} d \mu
$$
$$
=B(\alpha q, 1-\alpha q)-\int_{t_1 / t_2}^1 \mu^{-\alpha q}(1-\mu)^{\alpha q-1} d \mu
$$

Based on the above
$$
\int\limits_0^{t_1} \tau^{-\alpha q}\left[\left(t_1-\tau\right)^{\alpha q-1}-\left(t_2-\tau\right)^{\alpha q-1}\right] d \tau=\int_{t_1 / t_2}^1 \mu^{-\alpha q}(1-\mu)^{\alpha q-1} d \mu
$$
From this, we get the following estimate
$$
\int\limits_0^{t_1} \tau^{-\alpha q}\left[\left(t_1-\tau\right)^{\alpha q-1}-\left(t_2-\tau\right)^{\alpha q-1}\right] d \tau
$$
\begin{equation}\label{thesth104}
\leq\left(\frac{t_2}{t_1}\right)^{\alpha q} \int_{t_1 / t_2}^1(1-\mu)^{\alpha q-1} d \mu=\frac{1}{\alpha q} \frac{(t_2-t_1)^{\alpha q}}{t_1^{\alpha q}}.
\end{equation}
Therefore, for the estimate of $\mathcal{R}_1$, we arrive at
\begin{equation}\label{est103}
\left\|\mathcal{R}_1\right\|   \leq   \frac{M_2 t_1^{-\alpha q}}{\alpha q (1-\alpha q)\lambda_1^{\beta p}} \chi_0  \Upsilon \|\varphi\|_{D\left(L^{\beta p}\right)} (t_2-t_1)^{\alpha q}.
\end{equation}

Secondly, an estimate for the term $\mathcal{R}_2$ can be shown by using inequalities in Proposition \ref{thproth3} as follows
$$
\left\|\mathcal{R}_2\right\|  \leq
\int\limits_{t_1}^{t_2}\left\|\sum_{j=1}^{\infty}   F_j(\tau,u(\tau)) (t_2-\tau)^{\alpha-1} E_{\alpha, \alpha}\left(-\lambda_j^\beta (t_2-\tau)^{\alpha}\right) e_j(x)\right\| d \tau
$$
$$
\leq \frac{M_2}{ \lambda_1^{\beta p}} \int\limits_{t_1}^{t_2}\|F(\tau, u(\tau))\|\left(t_2-\tau\right)^{\alpha q-1} d \tau
\leq \frac{M_2}{ \lambda_1^{\beta p}} \Upsilon \int\limits_{t_1}^{t_2}\|u(\tau, \cdot)\|\left(t_2-\tau\right)^{\alpha q-1} d \tau
$$
$$
\leq \frac{M_2  }{ \lambda_1^{\beta p}} \chi_0  \Upsilon \|\varphi\|_{D\left(L^{\beta p}\right)} \int\limits_{t_1}^{t_2} \tau^{-\alpha q}\left(t_2-\tau\right)^{\alpha q-1} d \tau
$$
\begin{equation}\label{est104}
\leq \frac{M_2 }{\alpha q \lambda_1^{\beta (p-p_0)}} t_1^{-\alpha q} \chi_0  \Upsilon \|\varphi\|_{D\left(L^{\beta p}\right)} (t_2-t_1)^{\alpha q}.
\end{equation}

 Third, we will estimate the term $\mathcal{R}_3$. The conditions of Theorem 2 hold for the function $\varphi$, therefore, the estimate in (\ref{est42}) holds
\begin{equation}\label{est105}
\left\|\mathcal{R}_3\right\|  \leq
\frac{C_{\kappa} M_2 T^{\alpha q}}{\alpha q}t^{-\alpha q} \left( t_2-t_1 \right)^{\alpha q}
\|\varphi\|_{D\left(L^{\beta q}\right)}.
\end{equation}

Fourth, we proceed to estimate
$\mathcal{R}_4$. Using the estimation method for $\mathcal{I}_3$ and inequalities in Proposition \ref{thproth3},  the same argument as in (\ref{tfor43}) gives
$$
\left\|\mathcal{R}_4\right\|
\leq \int\limits_0^T \Bigg\{
\sum_{j=1}^{\infty} \lambda_j^{2 \beta} F_j^2(\tau, u(\tau)) \Psi_j^2(T)(T-\tau)^{2\alpha-2} \left(E_{\alpha, \alpha}\left(-\lambda_j^\beta (T-\tau)^\alpha\right)\right)^2
$$
$$
\left[\int\limits_{t_1}^{t_2} \eta^{\alpha-1} E_{\alpha, \alpha}\left(-\lambda_j^\beta \eta^\alpha\right) d \eta\right]^2 \Bigg\}^{1 / 2}  d \tau
\leq  C_{\kappa} M_2 \int\limits_0^T \Bigg\{
\sum_{j=1}^{\infty} \lambda_j^{2 \beta} F_j^2(\tau, u(\tau))
$$
$$
 (T-\tau)^{2\alpha-2} \frac{1}{\lambda_j^{2\beta p}}
(T-\tau)^{-2\alpha p}\left[\int\limits_{t_1}^{t_2} \eta^{\alpha-1} \lambda_j^{-\beta q} \eta^{-\alpha q} d \eta\right]^2 \Bigg\}^{1 / 2}  d \tau
$$
$$
\leq  C_{\kappa} M_2 \int\limits_0^T \Bigg\{
\sum_{j=1}^{\infty} \lambda_j^{2 \beta} F_j^2(\tau, u(\tau)) (T-\tau)^{2\alpha q-2}
 \frac{T^\alpha}{(\alpha q)^2 \lambda_j^{2\beta}}  \left[t_1^{-\alpha q}-t_2^{-\alpha q} \right]^2 \Bigg\}^{1 / 2}  d \tau
$$
$$
\leq  \frac{ C_{\kappa} M_2 T^\alpha}{\alpha q }  \int\limits_0^T   \|F(\tau, u(\tau))\| (T-\tau)^{\alpha q-1}
\frac{t_2^{\alpha q}-t_1^{\alpha q}}{t_1^{\alpha q}t_2^{\alpha q} }    d \tau
$$
$$
\leq  \frac{ C_{\kappa} M_2 T^\alpha}{\alpha q } \frac{t_2^{\alpha q}-t_1^{\alpha q}}{t_1^{\alpha q}t_2^{\alpha q} } \chi_0  \Upsilon \|\varphi\|_{D\left(L^{\beta p}\right)}\int\limits_0^T   \tau^{-\alpha q} (T-\tau)^{\alpha q-1}   d \tau
$$
\begin{equation}\label{est106}
\leq  \frac{ C_{\kappa} M_2 T^\alpha}{\alpha q } \frac{t_2^{\alpha q}-t_1^{\alpha q}}{t_1^{\alpha q}t_2^{\alpha q} } \chi_0  \Upsilon \|\varphi\|_{D\left(L^{\beta p}\right)}B(\alpha q,1-\alpha q).
\end{equation}
Next, we estimate $\mathcal{R}_5$
$$
\left\|\mathcal{R}_5\right\|
\leq \int\limits_0^T \Bigg\{
\sum_{j=1}^{\infty} \lambda_j^{2 \beta} k^2(\tau) u_j^2(\tau) \Psi_j^2(T)(T-\tau)^{2\alpha-2} \left(E_{\alpha, \alpha}\left(-\lambda_j^\beta (T-\tau)^\alpha\right)\right)^2
$$
$$
\left[\int\limits_{t_1}^{t_2} \eta^{\alpha-1} E_{\alpha, \alpha}\left(-\lambda_j^\beta \eta^\alpha\right) d \eta\right]^2 \Bigg\}^{1 / 2}  d \tau
$$
$$
\leq  C_{\kappa} M_2 \int\limits_0^T \Bigg\{
\sum_{j=1}^{\infty} \lambda_j^{2 \beta} k^2(\tau) u_j^2(\tau) (T-\tau)^{2\alpha q-2}
 \frac{T^\alpha}{(\alpha q)^2 \lambda_j^{2\beta}}  \left[t_1^{-\alpha q}-t_2^{-\alpha q} \right]^2 \Bigg\}^{1 / 2}  d \tau
$$
$$
\leq  \frac{ C_{\kappa} M_2 T^\alpha}{\alpha q }  \|k\|_{C[[0,T]} \int\limits_0^T   \|u(\tau, \cdot)\| (T-\tau)^{\alpha q-1}
\frac{t_2^{\alpha q}-t_1^{\alpha q}}{t_1^{\alpha q}t_2^{\alpha q} }    d \tau
$$
$$
\leq  \frac{ C_{\kappa} M_2 T^\alpha}{\alpha q } \frac{t_2^{\alpha q}-t_1^{\alpha q}}{t_1^{\alpha q}t_2^{\alpha q} } \chi_0 \|k\|_{C[[0,T]} \|\varphi\|_{D\left(L^{\beta p}\right)}\int\limits_0^T   \tau^{-\alpha q} (T-\tau)^{\alpha q-1}   d \tau
$$
\begin{equation}\label{est106thh}
\leq  \frac{ C_{\kappa} M_2 T^\alpha}{\alpha q } \frac{t_2^{\alpha q}-t_1^{\alpha q}}{t_1^{\alpha q}t_2^{\alpha q} } \chi_0 \|k\|_{C[[0,T]} \|\varphi\|_{D\left(L^{\beta p}\right)}B(\alpha q,1-\alpha q).
\end{equation}

Now let's consider $\mathcal{R}_6$. According to the arguments in $\left\|\mathcal{I}_6\right\|$, we obtain

$$
\left\|\mathcal{R}_6\right\|   \leq \int_0^{t_1}\left\|\sum_{j=1}^{\infty}   \int\limits_{t_1-\tau}^{t_2-\tau} k(\tau) u_j(\tau, \cdot) \eta^{\alpha-2} E_{\alpha, \alpha-1}\left(-\lambda_j^\beta \eta^\alpha\right) d \eta d \tau e_j(x)\right\| d \tau
$$
$$
\leq  \frac{M_2 \|k\|_{C[[0,T]}}{ \lambda_1^{\beta p}}  \int_0^{t_1}\|u(\tau, \cdot)\| d \tau \int\limits_{t_1-\tau}^{t_2-\tau} \eta^{\alpha q-2} d \eta
$$
$$
\leq  \frac{M_2 \|k\|_{C[[0,T]}}{(1-\alpha q) \lambda_1^{\beta p}} \chi_0  \|\varphi\|_{D\left(L^{\beta p}\right)} \int_0^{t_1} \tau^{-\alpha q} \Big[(t_1-\tau)^{\alpha q-1}-(t_2-\tau)^{\alpha q-1}\Big] d \tau
$$
According to inequality (\ref{thesth104}), we have
\begin{equation}\label{est107}
\left\|\mathcal{R}_6\right\|   \leq   \frac{M_2 \|k\|_{C[[0,T]}}{\alpha q (1-\alpha q) \lambda_1^{\beta p}} \chi_0     \frac{(t_2-t_1)^{\alpha q}}{t_1^{\alpha q}} \|\varphi\|_{D\left(L^{\beta p}\right)}.
\end{equation}

Let's estimate $\mathcal{I}_7$
$$
\left\|\mathcal{R}_7\right\|  \leq
\int\limits_{t_1}^{t_2}\left\|\sum_{j=1}^{\infty}  k(\tau) u_j(\tau) (t_2-\tau)^{\alpha-1} E_{\alpha, \alpha}\left(-\lambda_j^\beta (t_2-\tau)^{\alpha}\right) e_j(x)\right\| d \tau
$$
$$
\leq \frac{M_2 \|k\|_{C[[0,T]}}{ \lambda_1^{\beta p}} \int\limits_{t_1}^{t_2}\|u(\tau, \cdot)\|\left(t_2-\tau\right)^{\alpha q-1} d \tau
$$
$$
\leq \frac{M_2 \|k\|_{C[[0,T]}}{ \lambda_1^{\beta p}} \chi_0    \|\varphi\|_{D\left(L^{\beta p}\right)} \int\limits_{t_1}^{t_2} \tau^{-\alpha q} \left(t_2-\tau\right)^{\alpha q-1} d \tau
$$
\begin{equation}\label{est108}
\leq \frac{M_2 \|k\|_{C[[0,T]}}{ \lambda_1^{\beta p}} t_1^{-\alpha q} \chi_0  \|\varphi\|_{D\left(L^{\beta p}\right)} (t_2-t_1)^{\alpha q}.
\end{equation}

Using (\ref{eqth102}) and  (\ref{est103}) - (\ref{est108}), we obtain an estimate for $u(t_2, x) -u(t_1, x)$
$$
\left\|u(t_2, x) -u(t_1, x)\right\|\leq  \frac{M_2 t_1^{-\alpha q}}{\alpha q (1-\alpha q)\lambda_1^{\beta p}} \chi_0  \Upsilon \|\varphi\|_{D\left(L^{\beta p}\right)} (t_2-t_1)^{\alpha q}
$$
$$
+\frac{M_2 }{\alpha q \lambda_1^{\beta (p-p_0)}} t_1^{-\alpha q} \chi_0  \Upsilon \|\varphi\|_{D\left(L^{\beta p}\right)} (t_2-t_1)^{\alpha q}
+\frac{C_{\kappa} M_2 T^{\alpha q}}{\alpha q}t^{-\alpha q} \left( t_2-t_1 \right)^{\alpha q}
\|\varphi\|_{D\left(L^{\beta q}\right)}
$$
$$
+\frac{ C_{\kappa} M_2 T^\alpha}{\alpha q } \frac{t_2^{\alpha q}-t_1^{\alpha q}}{t_1^{\alpha q}t_2^{\alpha q} } \chi_0  \Upsilon \|\varphi\|_{D\left(L^{\beta p}\right)}B(\alpha q,1-\alpha q)
$$
$$
+\frac{ C_{\kappa} M_2 T^\alpha}{\alpha q } \frac{t_2^{\alpha q}-t_1^{\alpha q}}{t_1^{\alpha q}t_2^{\alpha q} } \chi_0 \|k\|_{C[[0,T]} \|\varphi\|_{D\left(L^{\beta p}\right)}B(\alpha q,1-\alpha q)
$$
$$
+\frac{M_2 \|k\|_{C[[0,T]}}{\alpha q (1-\alpha q) \lambda_1^{\beta p}} \chi_0     \frac{(t_2-t_1)^{\alpha q}}{t_1^{\alpha q}} \|\varphi\|_{D\left(L^{\beta p}\right)}
$$
$$
+\frac{M_2 \|k\|_{C[[0,T]}}{ \lambda_1^{\beta p}} t_1^{-\alpha q} \chi_0    \|\varphi\|_{D\left(L^{\beta p}\right)} (t_2-t_1)^{\alpha q}.
$$

We derive from the above that $\|u\left(t_2, x\right)-u\left(t_1, x\right)\|$ tends to $0$ as $t_2-t_1$ tends $0$ for $0<t_1<t_2\leq T$. The above statements prove $u \in C\left((0,T];L_2(\Omega)\right)$. This combines with (\ref{eqth96}) so that  $C^{\alpha q}\left((0,T];L_2(\Omega)\right)$, and from this, the estimate  (\ref{thth102}) in Theorem \ref{theorth4} holds.
\end{proof}

\begin{lemma}\label{lemmath8}
Let $p, q, r, p^{\prime}, q^{\prime}$ be defined by B1), B2). If $\varphi \in D\left(L^{\beta q}\right)$, $F$ satisfies C1), and $\Theta(T)<1$, then nonlocal initial problem (\ref{eqth1})--(\ref{eqth3}) has a unique solution
$$
u \in L^{\frac{1}{\alpha q^{\prime}}-r}\left(0, T ; D\left(L^{\beta (p-p')}\right)\right) \cap C^{\alpha q}\left((0, T] ; L_2(\Omega)\right),
$$
there exists a positive constant $C_5$ such that
\begin{equation}\label{theq110}
\|u\|_{L^{\frac{1}{a q^{\prime}}-r}\left(0, T ; D\left(L^{\beta\left(p-p'\right)}\right)\right)}+\|u\|_{C^{\alpha q}\left(0, T ; L_2(\Omega)\right)}\leq C_5\|\varphi\|_{D\left(L^{\beta p}\right)}
\end{equation}
\end{lemma}
The inequality (\ref{theq110}) in Lemma \ref{lemmath8} is obtained by combining the inequalities (\ref{theqth95}) and (\ref{thth102}).

\begin{theorem}\label{theorth5}
Let $p, q, r, p', q'$ be defined by B1), B3). If $\varphi$ belongs to $ D\left(L^{\beta q}\right)$, $F$ satisfies C1), and $\Theta(T)<1$, then nonlocal initial problem (\ref{eqth1})--(\ref{eqth3}) has a unique solution
$
u \in C^{\alpha q}\left([0, T] ; D\left(L^{-\beta q'}\right)\right)
$
and
$$
\|u\left(t, x\right)\|_{C^{\alpha q}\left([0,T];D\left(L^{-\beta q'}\right)\right)}
\leq  M_2 \Bigg[ \frac{T^{\alpha(p-q-p')} \chi_0  \Upsilon}{\alpha q(1-\alpha q')} +  \frac{ M_6  t_1^{-\alpha q}}{\alpha q \lambda_1^{\beta p'}}  \chi_0  \Upsilon
$$
$$
+ \frac{C_{\kappa}  }{\alpha q' \lambda_1^{2 \beta p}} T^{\alpha (p-p')} +\frac{ C_{\kappa} M_2 }{\alpha q  \lambda_1^{\beta q}} T^{\alpha (p-p')}   \chi_0  \Upsilon B(\alpha q',1-\alpha q)
+
\frac{\|k\|_{C[[0,T]}t_1^{-\alpha q}}{ (1-\alpha q \lambda_1^{\beta (p+q')}}
\chi_0
$$
\begin{equation}\label{ththeq112}
+\frac{ C_{\kappa} M_2 }{\alpha q  \lambda_1^{\beta q}} T^{\alpha (p-p')}   \chi_0  \|k\|_{C[0,T]}   B(\alpha q',1-\alpha q)+\frac{\|k\|_{C[[0,T]} }{ \lambda_1^{\beta (p+q')}} t_1^{-\alpha q} \chi_0  \Bigg]\|\varphi\|_{D\left(L^{\beta p}\right)}.
\end{equation}
Moreover, there exists a positive constant $C_6$ such that
$$
\|u\|_{L^{\frac{1}{\alpha q^{\prime}}-r}\left(0, T ; D\left(L^{\beta (p-p')}\right)\right)}+\|u\|_{C^{\alpha q}\left((0, T] ; L_2(\Omega)\right)}
$$
\begin{equation}\label{ththeq113}
+ \| u\|_{C^{\alpha q}\left([0, T] ; D\left(L^{-\beta q'}\right)\right)} \leq C_6 \| \varphi \|_{D\left(L^{\beta p}\right)}.
\end{equation}
\end{theorem}
\begin{proof}
According to Theorem \ref{theorth4}, we have just to prove that $u \in C^{\alpha q}\left([0, T] ; D\left(L^{-\beta q'}\right)\right)
$. In this part, we
consider $0 \leq  t_1 < t_2 \leq T$.

 Let us first consider $\left\|\mathcal{R}_1\right\|_{D\left(L^{-\beta q'}\right)}$. Using inequalities in Proposition \ref{thproth3} we obtain:
$$
\left\|\mathcal{R}_1\right\|_{D\left(L^{-\beta q'}\right)}
$$
$$
\leq \int_0^{t_1}\left\|-\sum_{j=1}^{\infty} \lambda_j^\beta  \int\limits_{t_1-\tau}^{t_2-\tau} F_j(\tau,u(\tau)) \eta^{\alpha-2} E_{\alpha, \alpha-1}\left(-\lambda_j^\beta \eta^\alpha\right) d \eta   e_j\right\|_{D\left(L^{-\beta q}\right)} d \tau
$$
$$
 \leq \int_0^{t_1}\left\{\sum_{j=1}^{\infty} \lambda_j^{2 \beta} \lambda_j^{-2 \beta q'} F_j^2(\tau,u(\tau))\left|\int\limits_{t_1-\tau}^{t_2-\tau} \eta^{\alpha-2} E_{\alpha, \alpha-1}\left(-\lambda_j^\beta \eta^\alpha\right)d \eta\right|^2\right\}^{1 / 2} d \tau
$$$$
  \leq M_2 \int_0^{t_1}\left\{\sum_{j=1}^{\infty} \lambda_j^{2 \beta} \lambda_j^{-2 \beta q'} F_j^2(\tau,u(\tau))\left|\int\limits_{t_1-\tau}^{t_2-\tau} \eta^{\alpha-2} \lambda_j^{-\beta p'} \eta^{-\alpha p'} d \eta\right|^2\right\}^{1 / 2} d \tau
$$
$$
\leq  M_2 \int\limits_0^{t_1} \|F(\tau, u(\tau))\| \int\limits_{t_1-\tau}^{t_2-\tau} \eta^{\alpha q'-2} d \eta  d \tau
$$
$$
\leq \frac{M_2}{1-\alpha q'}  \int\limits_0^{t_1} \|F(\tau, u(\tau))\| \Big[(t_1-\tau)^{\alpha q'-1} - (t_2-\tau)^{\alpha q'-1}\Big]  d \tau
$$
$$
\leq \frac{M_2}{1-\alpha q'} \Upsilon \int\limits_0^{t_1} \|u(\tau, \cdot)\| (t_1-\tau)^{\alpha(p-p')} \Big[(t_1-\tau)^{\alpha q-1} - (t_2-\tau)^{\alpha q-1}\Big]  d \tau
$$
$$
\leq \frac{M_2 \chi_0  \Upsilon}{1-\alpha q'} t_1^{\alpha(p-p')} \|\varphi\|_{D\left(L^{\beta p}\right)} \int\limits_0^{t_1} \tau^{-\alpha q}  \Big[(t_1-\tau)^{\alpha q-1} - (t_2-\tau)^{\alpha q-1}\Big]  d \tau.
$$
By using inequality (\ref{thesth104}), we have
$$
\left\|\mathcal{R}_1\right\|_{D\left(L^{-\beta q'}\right)} \leq  \frac{M_2 \chi_0  \Upsilon}{\alpha q(1-\alpha q')} t_1^{\alpha(p-p')} \|\varphi\|_{D\left(L^{\beta p}\right)} \frac{(t_2-t_1)^{\alpha q}}{t_1^{\alpha q}}
$$
\begin{equation}\label{esth112}
\leq  \frac{M_2 \chi_0  \Upsilon}{\alpha q(1-\alpha q')} T^{\alpha(p-q-p')} \|\varphi\|_{D\left(L^{\beta p}\right)} (t_2-t_1)^{\alpha q}.
\end{equation}

Secondly, we are going to consider $\mathcal{R}_2$. The Sobolev embedding $L_2(\Omega) \hookrightarrow D\left(L^{-\beta q'}\right)$
implies that yields that there exists a positive
constant $M_6$ such that
$$
\left\|\mathcal{R}_2\right\|_{D\left(L^{-\beta q'}\right)} \leq M_6 \left\|\mathcal{R}_2\right\|
\leq   \frac{M_2 M_6}{\alpha q \lambda_1^{\beta p'}} t_1^{-\alpha q} \chi_0  \Upsilon \|\varphi\|_{D\left(L^{\beta p}\right)} (t_2-t_1)^{\alpha q}
$$
\begin{equation}\label{est113}
\leq \frac{M_2 M_6}{\alpha q \lambda_1^{\beta p'}} t_1^{-\alpha q} \chi_0  \Upsilon \|\varphi\|_{D\left(L^{\beta p}\right)} (t_2-t_1)^{\alpha q}.
\end{equation}

Thirdly, we consider the norm $\left\|\mathcal{R}_3\right\|_{D\left(L^{-\beta q'}\right)} $. By applying the same arguments as in (\ref{est50}), one can get
$$
\left\|\mathcal{R}_3\right\|_{D\left(L^{-\beta q'}\right)}=\left\|L^\beta \sum_{j=1}^{\infty} \Psi_j(T)  \varphi_j  \int\limits_{t_1}^{t_2} \tau^{\alpha-1} E_{\alpha, \alpha}\left(-\lambda_j^\beta \tau^\alpha\right) d \tau e_j(x)\right\|_{D\left(L^{-\beta q'}\right)}
$$
$$
\leq \left\{
\sum_{j=1}^{\infty}  \varphi_j^2 \Psi_j^2(T)\left[\int\limits_{t_1}^{t_2} \tau^{\alpha-1} E_{\alpha, \alpha}\left(-\lambda_j^\beta \tau^\alpha\right) d \tau\right]^2 \lambda_j^{2 \beta} \lambda_j^{-2 \beta q'}\right\}^{1 / 2}
$$
$$
\leq C_{\kappa} M_2  \left\{ \sum_{j=1}^{\infty}  \varphi_j^2    \left[\int\limits_{t_1}^{t_2} \tau^{\alpha-1}   \tau^{\alpha p'} d \tau\right]^2 \right\}^{1 / 2}
$$
$$
\leq C_{\kappa} M_2  \left\{ \sum_{j=1}^{\infty}  \varphi_j^2   \left[\int\limits_{t_1}^{t_2} \tau^{\alpha q+\alpha (p-p')-1}  d \tau\right]^2 \right\}^{1 / 2}
$$
$$
\leq \frac{C_{\kappa} M_2}{\alpha q' \lambda_1^{2 \beta p}} \|\varphi\|_{D\left(L^{\beta p}\right)} \left[t_2^{\alpha q+\alpha (p-p')}-t_1^{\alpha q+\alpha (p-p')}\right].
$$
From inequality (\ref{esth39}), we have
\begin{equation}\label{est119}
\left\|\mathcal{R}_3\right\|_{D\left(L^{-\beta q'}\right)} \leq \frac{C_{\kappa} M_2}{\alpha q' \lambda_1^{2 \beta p}} T^{\alpha (p-p')}\|\varphi\|_{D\left(L^{\beta p}\right)} \left(t_2-t_1\right)^{\alpha q}.
\end{equation}

Fourthly, we proceed to estimate
$\mathcal{R}_4$. Using the estimation method for $\mathcal{I}_4$, we get the following estimate for $\mathcal{R}_4$
$$
\left\|\mathcal{R}_4\right\|_{D\left(L^{-\beta q'}\right)}
$$$$
\leq \int\limits_0^T \Bigg\{
\sum_{j=1}^{\infty} \lambda_j^{2 \beta} \lambda_j^{-2 \beta q'} F_j^2(\tau,u(\tau)) \Psi_j^2(T)(T-\tau)^{2\alpha-2} \left(E_{\alpha, \alpha}\left(-\lambda_j^\beta (T-\tau)^\alpha\right)\right)^2
$$
$$
\left[\int\limits_{t_1}^{t_2} \eta^{\alpha-1} E_{\alpha, \alpha}\left(-\lambda_j^\beta \eta^\alpha\right) d \eta\right]^2 \Bigg\}^{1 / 2}  d \tau
\leq C_{\kappa} M_2^2   \int\limits_0^T \Bigg\{
\sum_{j=1}^{\infty} \lambda_j^{2 \beta} \lambda_j^{-2\beta q'} F_j^2(\tau,u(\tau))
$$
$$
 (T-\tau)^{2\alpha-2} \frac{1}{\lambda_j^{2\beta p'}}
(T-\tau)^{-2\alpha p'}\left[\int\limits_{t_1}^{t_2} \eta^{\alpha-1} \lambda_j^{-\beta p} \eta^{-\alpha p} d \eta\right]^2 \Bigg\}^{1 / 2}  d \tau
$$
$$
\leq  \frac{C_{\kappa} M_2^2}{ \lambda_1^{\beta p}} \int\limits_0^T \Bigg\{
\sum_{j=1}^{\infty} \lambda_j^{2 \beta} F_j^2(\tau) (T-\tau)^{2\alpha q' -2}
\left[\int\limits_{t_1}^{t_2} \eta^{\alpha q-1}   d \eta\right]^2 \Bigg\}^{1 / 2}  d \tau
$$
$$
\leq  \frac{ C_{\kappa} M_2^2}{\alpha q  \lambda_1^{\beta q}} \left(t_2^{\alpha q}-t_1^{\alpha q}\right) \int\limits_0^T   \|F(\tau, \cdot)\| (T-\tau)^{\alpha q'-1} d \tau
$$
$$
\leq  \frac{ C_{\kappa} M_2^2 }{\alpha q  \lambda_1^{\beta q}} \left(t_2^{\alpha q}-t_1^{\alpha q}\right) \chi_0  \Upsilon \|\varphi\|_{D\left(L^{\beta p}\right)} \int\limits_0^T  \tau^{-\alpha q}  (T-\tau)^{\alpha q' -1} d \tau
$$
\begin{equation}\label{est115}
\leq  \frac{ C_{\kappa} M_2^2 }{\alpha q  \lambda_1^{\beta q}} T^{\alpha (p-p')}   \chi_0  \Upsilon \|\varphi\|_{D\left(L^{\beta p}\right)}  B(\alpha q',1-\alpha q)\left(t_2-t_1\right)^{\alpha q}.
\end{equation}
Next, we estimate  $\mathcal{R}_5$
$$
\left\|\mathcal{R}_5\right\|_{D\left(L^{-\beta q'}\right)} \leq \int\limits_0^T \Bigg\{
\sum_{j=1}^{\infty} \lambda_j^{2 \beta} \lambda_j^{-2 \beta q'} k^2(\tau) u_j^2(\tau) \Psi_j^2(T)(T-\tau)^{2\alpha-2}
$$
$$
\times \left(E_{\alpha, \alpha}\left(-\lambda_j^\beta (T-\tau)^\alpha\right)\right)^2
\left[\int\limits_{t_1}^{t_2} \eta^{\alpha-1} E_{\alpha, \alpha}\left(-\lambda_j^\beta \eta^\alpha\right) d \eta\right]^2 \Bigg\}^{1 / 2}  d \tau
$$
$$
\leq  \frac{C_{\kappa} M_2^2}{ \lambda_1^{\beta p}} \int\limits_0^T \Bigg\{
\sum_{j=1}^{\infty} \lambda_j^{2 \beta} k^2(\tau) u_j^2(\tau) (T-\tau)^{2\alpha q' -2}
\left[\int\limits_{t_1}^{t_2} \eta^{\alpha q-1}   d \eta\right]^2 \Bigg\}^{1 / 2}  d \tau
$$
$$
\leq  \frac{ C_{\kappa} M_2^2}{\alpha q  \lambda_1^{\beta q}} \left(t_2^{\alpha q}-t_1^{\alpha q}\right) \|k\|_{C[0,T]} \int\limits_0^T   \|u(\tau, \cdot)\| (T-\tau)^{\alpha q'-1} d \tau
$$
$$
\leq  \frac{ C_{\kappa} M_2^2 }{\alpha q  \lambda_1^{\beta q}} \left(t_2^{\alpha q}-t_1^{\alpha q}\right) \chi_0  \|k\|_{C[0,T]}  \|\varphi\|_{D\left(L^{\beta p}\right)} \int\limits_0^T  \tau^{-\alpha q}  (T-\tau)^{\alpha q' -1} d \tau
$$
\begin{equation}\label{est115th}
\leq  \frac{ C_{\kappa} M_2^2 }{\alpha q  \lambda_1^{\beta q}} T^{\alpha (p-p')}   \chi_0  \|k\|_{C[0,T]}  \|\varphi\|_{D\left(L^{\beta p}\right)}  B(\alpha q',1-\alpha q)\left(t_2-t_1\right)^{\alpha q}.
\end{equation}

Now let's estimate $\mathcal{R}_6$  in the same way:
$$
\left\|\mathcal{R}_6\right\|_{D\left(L^{-\beta q'}\right)}
\leq \int_0^{t_1}\left\|\sum_{j=1}^{\infty}   \int\limits_{t_1-\tau}^{t_2-\tau} k(\tau) u_j(\tau) \eta^{\alpha-2} E_{\alpha, \alpha-1}\left(-\lambda_j^\beta \eta^\alpha\right) d \eta d \tau e_j(x)\right\|_{D\left(L^{-\beta q'}\right)}  d \tau
$$
$$
\leq  \frac{M_2 \|k\|_{C[[0,T]}}{ \lambda_1^{\beta (p+q')}}  \int_0^{t_1}\|u(\tau, \cdot)\| d \tau \int\limits_{t_1-\tau}^{t_2-\tau} \eta^{\alpha q-2} d \eta
$$
$$
\leq  \frac{M_2 \|k\|_{C[[0,T]}}{ (1-\alpha q \lambda_1^{\beta (p+q')}}
\chi_0  \|\varphi\|_{D\left(L^{\beta p}\right)}\int\limits_0^{t_1} \tau^{-\alpha q} \Big[ (t_1-\tau)^{\alpha q-1} - (t_2-\tau)^{\alpha q-1}\Big]  d \tau.
$$
By using (\ref{thesth104}) in the above inequality, we obtain
\begin{equation}\label{est116}
\left\|\mathcal{R}_6\right\| \leq  \frac{M_2 \|k\|_{C[[0,T]}t_1^{-\alpha q}}{ (1-\alpha q \lambda_1^{\beta (p+q')}}
\chi_0  \|\varphi\|_{D\left(L^{\beta p}\right)}(t_2-t_1)^{\alpha q}.
\end{equation}

Similarly, we estimate $\mathcal{R}_7$ using the method used above
$$
\left\|\mathcal{R}_7\right\|_{D\left(L^{-\beta q'}\right)}
\leq
\int\limits_{t_1}^{t_2}\left\|\sum_{j=1}^{\infty}  k(\tau) u_j(\tau) (t_2-\tau)^{\alpha-1} E_{\alpha, \alpha}\left(-\lambda_j^\beta (t_2-\tau)^{\alpha}\right) e_j(x)\right\|_{D\left(L^{-\beta q'}\right)} d \tau
$$
$$
\leq \frac{M_2 \|k\|_{C[[0,T]}}{ \lambda_1^{\beta (p+q')}} \int\limits_{t_1}^{t_2}\|u(\tau, \cdot)\|\left(t_2-\tau\right)^{\alpha q-1} d \tau
$$
$$
\leq \frac{M_2 \|k\|_{C[[0,T]} }{ \lambda_1^{\beta (p+q')}} \chi_0  \|\varphi\|_{D\left(L^{\beta p}\right)} \int\limits_{t_1}^{t_2} \tau^{-\alpha q}\left(t_2-\tau\right)^{\alpha q-1} d \tau
$$
\begin{equation}\label{est117}
\leq \frac{M_2 \|k\|_{C[[0,T]} }{ \lambda_1^{\beta (p+q')}} t_1^{-\alpha q} \chi_0  \|\varphi\|_{D\left(L^{\beta p}\right)} \left(t_2-t_1\right)^{\alpha q}.
\end{equation}
Using (\ref{esth112}) - (\ref{est117}),  we estimate the difference $u\left(t_2, x\right)-u\left(t_1, x\right)$ in dual Space $D\left(L^{-\beta q'}\right)$ and obtain
$$
\|u\left(t_2, x\right)-u\left(t_1, x\right)\|_{D\left(L^{-\beta q'}\right)}   \leq M_2 \Bigg[ \frac{T^{\alpha(p-q-p')} \chi_0  \Upsilon}{\alpha q(1-\alpha q')}   +  \frac{ M_6  t_1^{-\alpha q}}{\alpha q \lambda_1^{\beta p'}}  \chi_0  \Upsilon
$$
$$
+ \frac{C_{\kappa}  }{\alpha q' \lambda_1^{2 \beta p}} T^{\alpha (p-p')} +\frac{ C_{\kappa} M_2 }{\alpha q  \lambda_1^{\beta q}} T^{\alpha (p-p')}   \chi_0  \Upsilon B(\alpha q',1-\alpha q)
+
\frac{\|k\|_{C[[0,T]}t_1^{-\alpha q}}{ (1-\alpha q \lambda_1^{\beta (p+q')}}
\chi_0
$$
\begin{equation}\label{est118}
+\frac{ C_{\kappa} M_2 }{\alpha q  \lambda_1^{\beta q}} T^{\alpha (p-p')}   \chi_0  \|k\|_{C[0,T]}   B(\alpha q',1-\alpha q)+\frac{\|k\|_{C[[0,T]} }{ \lambda_1^{\beta (p+q')}} t_1^{-\alpha q} \chi_0  \Bigg]\|\varphi\|_{D\left(L^{\beta p}\right)} \left(t_2-t_1\right)^{\alpha q}.
\end{equation}

We derive from the above that $\|u\left(t_2, x\right)-u\left(t_1, x\right)\|_{D\left(L^{-\beta q'}\right)}$ tends to $0$ as $t_2-t_1$ tends $0$ for $0<t_1<t_2\leq T$.  Thus, $u(t, x)$ belongs to the set $C^{\alpha q}\left([0,T];D\left(L^{-\beta q'}\right)\right)$.

Now, we divide both sides of inequality (\ref{est118}) by $\left(t_2-t_1\right)^{\alpha q}$
$$
\frac{\|u\left(t_2, x\right)-u\left(t_1, x\right)\|_{D\left(L^{-\beta q'}\right)}}{\left(t_2-t_1\right)^{\alpha q}}   \leq M_2 \Bigg[ \frac{T^{\alpha(p-q-p')} \chi_0  \Upsilon}{\alpha q(1-\alpha q')} +  \frac{ M_6  t_1^{-\alpha q}}{\alpha q \lambda_1^{\beta p'}}  \chi_0  \Upsilon
$$
$$
+ \frac{C_{\kappa}  }{\alpha q' \lambda_1^{2 \beta p}} T^{\alpha (p-p')} +\frac{ C_{\kappa} M_2 }{\alpha q  \lambda_1^{\beta q}} T^{\alpha (p-p')}   \chi_0  \Upsilon B(\alpha q',1-\alpha q)
+
\frac{\|k\|_{C[[0,T]}t_1^{-\alpha q}}{ (1-\alpha q \lambda_1^{\beta (p+q')}}
\chi_0
$$
$$
+\frac{ C_{\kappa} M_2 }{\alpha q  \lambda_1^{\beta q}} T^{\alpha (p-p')}   \chi_0  \|k\|_{C[0,T]}   B(\alpha q',1-\alpha q)+\frac{\|k\|_{C[[0,T]} }{ \lambda_1^{\beta (p+q')}} t_1^{-\alpha q} \chi_0  \Bigg]\|\varphi\|_{D\left(L^{\beta p}\right)}.
$$

On the other hand, according to the definition of the norm in the class of H\"{o}lder-continuous functions, we have  $u(t, x)\in C^{\alpha q}\left([0,T];D\left(L^{-\beta q'}\right)\right)$, such that
$$
\|u\left(t, x\right)\|_{C^{\alpha q}\left([0,T];D\left(L^{-\beta q'}\right)\right)}
\leq  M_2 \Bigg[ \frac{T^{\alpha(p-q-p')} \chi_0  \Upsilon}{\alpha q(1-\alpha q')} +  \frac{ M_6  t_1^{-\alpha q}}{\alpha q \lambda_1^{\beta p'}}  \chi_0  \Upsilon
$$
$$
+ \frac{C_{\kappa}  }{\alpha q' \lambda_1^{2 \beta p}} T^{\alpha (p-p')} +\frac{ C_{\kappa} M_2 }{\alpha q  \lambda_1^{\beta q}} T^{\alpha (p-p')}   \chi_0  \Upsilon B(\alpha q',1-\alpha q)
+
\frac{\|k\|_{C[[0,T]}t_1^{-\alpha q}}{ (1-\alpha q \lambda_1^{\beta (p+q')}}
\chi_0
$$
$$
+\frac{ C_{\kappa} M_2 }{\alpha q  \lambda_1^{\beta q}} T^{\alpha (p-p')}   \chi_0  \|k\|_{C[0,T]}   B(\alpha q',1-\alpha q)+\frac{\|k\|_{C[[0,T]} }{ \lambda_1^{\beta (p+q')}} t_1^{-\alpha q} \chi_0  \Bigg]\|\varphi\|_{D\left(L^{\beta p}\right)}.
$$
From this follows the inequality (\ref{ththeq112}) in Theorem \ref{theorth5}. The inequality (\ref{ththeq113}) in Theorem \ref{theorth5} is obtained by combining the inequalities (\ref{theqth95}), (\ref{thth102}) and (\ref{ththeq112}).
 \end{proof}

\section*{4.2 Analysis of the time-space fractional derivative of the mild solution of the nonlocal problem with a nonlinear source}

\begin{theorem}\label{theorth6}
Let $p, q, p', q', \widehat{p}, \widehat{q}, r, \widehat{r}$ be defined by A1), B2), B4). If $\varphi \in D\left(L^{\beta(p+\widehat{q})}\right)$, $k\in C[0,T]$, $F$ satisfies the assumptions C2), and $\Theta(T)<1$, then nonlocal initial problem (\ref{eqth1})--(\ref{eqth3}) has a unique solution $u(t, x)$ satisfying that
$
\partial_t^\alpha u \in L^{\frac{1}{\alpha }-\widehat{r}}\left(0, T ;
D\left(L^{-\beta(q-\widehat{q})}\right)\right).
$
Moreover, there exists a constant $C_7>0$ such that
\begin{equation}\label{ththeq121}
 \left\|\partial_t^\alpha u\right\|_{L^{\frac{1}{\alpha }-\widehat{r}}\left(0, T ;
D\left(L^{-\beta(q-\widehat{q})}\right)\right)} \leq  C_7 \left\| t^{-\alpha}\right\|_{L^{\frac{1}{\alpha}-\widehat{r}}\left(0, T\right)}  \|\varphi\|_{D\left(L^{\beta(p+\widehat{q})}\right)}.
\end{equation}
\end{theorem}

\begin{proof}
Since $F$ satisfies C2), $F$ also satisfies C1) with respect to the Lipschitz constant $\Upsilon^{*}$. In addition, the Sobolev imbedding $D\left(L^{\beta(p+\widehat{q})}\right)
\hookrightarrow D\left(L^{\beta p}\right)$ shows that $\varphi$ belongs to $D\left(L^{\beta p}\right)$. Hence, by Theorems \ref{theorth3}-\ref{theorth5}, nonlocal initial problem (\ref{eqth1})--(\ref{eqth3}) has a unique solution

$$
u \in L^{\frac{1}{\alpha q'}-r}\left(0, T;
D\left(L^{\beta(p-p')}\right) \right) \cap C^{\alpha q}\left((0, T] ; L_2(\Omega)\right).
$$

Moreover, the inequality (\ref{eqth96}) also holds. We conclude that, for $0<t \leq T$,
\begin{equation}\label{eqth121}
\|F(t, \cdot, u(t, \cdot))\| \leq \Upsilon^{*} \chi_0  \|\varphi\|_{D\left(L^{\beta p}\right)} t^{-\alpha q} \leq C_7 \Upsilon^{*} \chi_0  \|\varphi\|_{D\left(L^{\beta(p+\widehat{q})}\right)} t^{-\alpha q}.
\end{equation}

We prove $\partial_t^\alpha u$ finitely exists and belongs to $L^{\frac{1}{\alpha }-\widehat{r}}\left(0, T ;
D\left(L^{-\beta(p+\widehat{q})}\right)\right)$. By the same way
as in  Theorem \ref{theorth1}, using the above identities and equation (\ref{eqth9}), we obtain for $\partial_t^\alpha u_j(t)$:
$$
\partial_t^\alpha u_j(t) =\bigg[\varphi_j- F_j(T,u(T)) \star \left\{T^{\alpha-1} E_{\alpha, \alpha}\left(-\lambda_j^\beta T^\alpha\right)\right\} \bigg]\Psi_j(T) \partial_t^\alpha E_{\alpha, 1}\left(-\lambda_j^\beta t^\alpha\right)
$$
$$
+\partial_t^\alpha \left[F_j(t,u(t)) \star \left\{t^{\alpha-1} E_{\alpha, \alpha}\left(-\lambda_j^\beta t^\alpha\right)\right\}\right]-\partial_t^\alpha \left[\{k(t) u_j(t)\} \star \left\{t^{\alpha-1} E_{\alpha, \alpha}\left(-\lambda_j^\beta t^\alpha\right)\right\}\right]
$$
$$
=F_j(t,u(t)) -k(t) u_j(t) - \lambda_j^\beta  F_j(t,u(t)) \star \left\{t^{\alpha-1} E_{\alpha, \alpha}\left(-\lambda_j^\beta t^\alpha\right)\right\} - \lambda_j^\beta \varphi_j \Psi_j(T) E_{\alpha, 1}\left(-\lambda_j^\beta t^\alpha\right)
$$
$$
-\lambda_j^\beta F_j(T,u(T)) \star \left\{T^{\alpha-1} E_{\alpha, \alpha}\left(-\lambda_j^\beta T^\alpha\right)\right\}    \Psi_j(T)  E_{\alpha, 1}\left(-\lambda_j^\beta t^\alpha\right)
$$
$$
+\lambda_j^\beta \{k(T) u_j(T)\} \star \left\{T^{\alpha-1} E_{\alpha, \alpha}\left(-\lambda_j^\beta T^\alpha\right)\right\}    \Psi_j(T)  E_{\alpha, 1}\left(-\lambda_j^\beta t^\alpha\right)
$$
$$
 + \lambda_j^\beta  \{k(t) u_j(t)\} \star \left\{t^{\alpha-1} E_{\alpha, \alpha}\left(-\lambda_j^\beta t^\alpha\right)\right\}
$$
$$
=F_j(t,u(t))-k(t) u_j(t)+\sum\limits_{i=1}^{5} R_{ij}(t),\,\, j\in \mathbb{N}.
$$
In view of (\ref{eqth121}), $F(t,\cdot, u(t,\cdot))$ is contained in $L_2(\Omega)$ for $0<t \leq T$. This associates with the Sobolev embedding $L_2(\Omega) \hookrightarrow
D\left(L^{-\beta(q-\widehat{q})}\right)$ that $F(t,\cdot, u(t,\cdot))$ is contained in $D\left(L^{-\beta(q-\widehat{q})}\right)$, namely $\sum\limits_{j=1}^{\infty} F_j(t, u(t)) e_j$ is contained in $D\left(L^{-\beta(q-\widehat{q})}\right)$.

Now, we consider $\left\|\sum_{j=1}^{\infty} R_{ij}(t) e_j\right\|_{D\left(L^{-\beta(q-\widehat{q})}\right)}, i=\overline{1,5}$.
First, let's consider the sum $\sum\limits_{n_1\leq j \leq n_2} R_{1j}(t)e_j$, for $n_1,n_2\in \mathbb{N}$, $1\leq n_1 <  n_2.$
According to the definition of the dual space $D\left(L^{-\beta (q-\widehat{q})}\right)$
of $D\left(L^{\beta (q-\widehat{q})}\right),$ and the identity (\ref{eqth5555}) of their dual inner product, we have
$$
\left\|\sum\limits_{n_1\leq j \leq n_2}R_{1j}(t)e_j\right\|_{D\left(L^{-\beta (q-\widehat{q})}\right)}
$$
$$
= \left\|\sum\limits_{n_1\leq j \leq n_2}\lambda_j^\beta  F_j(t, u(t)) \star \left\{t^{\alpha-1} E_{\alpha, \alpha}\left(-\lambda_j^\beta t^\alpha\right)\right\}e_j\right\|_{D\left(L^{-\beta (q-\widehat{q})}\right)}
$$
$$
\leq \int\limits_0^t\left\|\sum_{n_1 \leq j \leq n_2} \lambda_j^\beta  F_j(\tau, u(\tau))  (t-\tau)^{\alpha-1} E_{\alpha, \alpha}\left(-\lambda_j^\beta (t-\tau)^\alpha\right)  e_j\right\|_{D\left(L^{-\beta (q-\widehat{q})}\right)} d \tau
$$
$$
 \leq  \int\limits_0^t\Bigg\{\sum_{i=1}^{\infty} \lambda_i^{-2 \beta(q-\widehat{q})}\Bigg(\sum_{n_1 \leq j \leq n_2} \lambda_j^\beta  F_j(\tau, u(\tau))
 (t-\tau)^{\alpha-1}
$$$$
\times E_{\alpha, \alpha}\left(-\lambda_j^\beta (t-\tau)^\alpha\right)  e_j, e_i\Bigg)_{-\beta(q-\widehat{q}), \beta(q-\widehat{q})}^2\Bigg\}^{1 / 2} d \tau
 $$$$
\leq   \int\limits_0^t\left\{\sum_{n_1 \leq j \leq n_2} \lambda_j^{2 \beta(p+\widehat{q})}  F_j^2(\tau, u(\tau))  (t-\tau)^{2\alpha-2} \left(E_{\alpha, \alpha}\left(-\lambda_j^\beta (t-\tau)^\alpha\right)\right)^2\right\}^{1 / 2} d \tau
$$
$$
\leq M_2  \int\limits_0^t\left\{\sum_{n_1 \leq j \leq n_2} \lambda_j^{2 \beta(p+\widehat{q})}  F_j^2(\tau, u(\tau))  (t-\tau)^{2\alpha-2}  \lambda_j^{-2 \beta(p+\widehat{q})}
(t-\tau)^{-2\alpha(p+\widehat{q})} \right\}^{1 / 2} d \tau
$$
\begin{equation} \label{esth122}
= M_2  \int\limits_0^t (t-\tau)^{\alpha(q-\widehat{q})-1}
\left\{\sum_{n_1 \leq j \leq n_2}   F_j^2 (\tau, u(\tau))    \right\}^{1 / 2} d \tau.
\end{equation}
For $0<\tau<T$, we have $F(\tau, ., u(\tau,))$ belonging to $L_2(\Omega)$. This follows that the sequence  $\left\{\sum_{ j \geq n}   F_j^2 (\tau, u(\tau))    \right\}^{1 / 2}$, converges pointwise to 0 as $n$ goes to infinity.  Moreover, by (\ref{eqth121}), we have
$$
\int\limits_0^t (t-\tau)^{\alpha(q-\widehat{q})-1}
\left\{\sum_{ j \geq n}   F_j^2 (\tau, u(\tau))    \right\}^{1 / 2} d \tau
\leq  C_7 \Upsilon^{*} \chi_0  \|\varphi\|_{D\left(L^{\beta(p+\widehat{q})}\right)}
$$
$$
\times\int\limits_0^t (t-\tau)^{\alpha(q-\widehat{q})-1}  \tau^{-\alpha q} d \tau
= M_2 C_7 \Upsilon^{*} \chi_0  \|\varphi\|_{D\left(L^{\beta(p+\widehat{q})}\right)} T^{\alpha-\alpha \widehat{q}} B(1-\alpha q,\alpha(q-\widehat{q}))t^{-\alpha}.
$$
Therefore, the dominated convergence theorem yields that
$$
\lim\limits_{n \rightarrow \infty} \int\limits_0^t (t-\tau)^{\alpha(q-\widehat{q})-1}
\left\{\sum_{ j \geq n}   F_j^2 (\tau, u(\tau))    \right\}^{1 / 2} d \tau=0.
$$
This together with $\left\{\sum_{n_1 \leq j \leq n_2}   F_j^2 (\tau, u(\tau))    \right\}^{1 / 2} \leq \left\{\sum_{ j \geq n}   F_j^2 (\tau, u(\tau))    \right\}^{1 / 2}$ gives
$$
\lim\limits_{n \rightarrow \infty} \int\limits_0^t (t-\tau)^{\alpha(q-\widehat{q})-1}
\left\{\sum_{n_1 \leq j \leq n_2}   F_j^2 (\tau, u(\tau))    \right\}^{1 / 2} d \tau=0.
$$

  Secondy, we proceed to establish an estimate for the sum $\sum\limits_{n_1\leq j \leq n_2} R_{2j}(t)e_j$.  On the other hand, $R_{2j}(t)=J_{2j}(t)$, and the norm
$\left\|\sum\limits_{n_1\leq j \leq n_2} R_{2j}(t)e_j\right\|_{D\left(L^{-\beta (q-\widehat{q})}\right)}$ exists finitely by (\ref{eqth61}), it follows from $\varphi\in D\left(L^{-\beta(q-\widehat{q})}\right)$
that
$$
\lim\limits_{n_1,n_2\rightarrow \infty} \sum_{n_1 \leq j \leq n_2} \varphi_j^2(\tau)\lambda_i^{2 \beta(p+\widehat{q})}=0.
$$

Third, we estimate the sum  $\sum\limits_{n_1\leq j \leq n_2} R_{3j}(t)e_j$. By methods similar to (\ref{esth122}), we obtain:
$$
\left\|\sum\limits_{n_1\leq j \leq n_2} R_{3j}(t)e_j\right\|_{D\left(L^{-\beta (q-\widehat{q})}\right)}
\leq  \int\limits_0^T \Bigg\{\sum_{i=1}^{\infty} \lambda_i^{-2 \beta(p-\widehat{q})}\Bigg(\sum_{n_1 \leq j \leq n_2} \lambda_j^\beta  F_j(\tau, u(\tau))
$$$$
\times (T-\tau)^{\alpha-1} E_{\alpha, \alpha}\left(-\lambda_j^\beta (T-\tau)^\alpha\right)  \Psi_j(T)  E_{\alpha, 1}\left(-\lambda_j^\beta t^\alpha\right)e_j, e_i\Bigg)_{-\beta(q-\widehat{q}), \beta(q-\widehat{q})}^2\Bigg\}^{1 / 2} d \tau
 $$
$$
 \leq  \frac{C_{\kappa} M_2}{\lambda_1^\beta}   t^{-\alpha} \int\limits_0^T \Bigg\{\sum_{n_1 \leq j \leq n_2} \lambda_j^{2 \beta(p+\widehat{q})}  F_j^2(\tau, u(\tau))  (T-\tau)^{2\alpha(q-\widehat{q})-2} \lambda_j^{-2 \beta(p+\widehat{q})} \Bigg\}^{1 / 2} d \tau
$$
\begin{equation}\label{eqth123}
\leq  \frac{C_{\kappa} M_2}{\lambda_1^\beta}   t^{-\alpha} \int\limits_0^T (T-\tau)^{\alpha(q-\widehat{q})-1}  \Bigg\{\sum_{n_1 \leq j \leq n_2} F_j^2(\tau, u(\tau))    \Bigg\}^{1 / 2} d \tau.
\end{equation}
Same as estimating the series $\sum\limits_{n_1\leq j \leq n_2} R_{3j}(t)$, we have it too
$$
\lim\limits_{n \rightarrow \infty} \int\limits_0^T (T-\tau)^{\alpha(q-\widehat{q})-1}
\left\{\sum_{n_1 \leq j \leq n_2}   F_j^2 (\tau, u(\tau))    \right\}^{1 / 2} d \tau=0.
$$

Next, we estimate the series $\sum\limits_{n_1\leq j \leq n_2} R_{4j}(t)$
$$
\left\|\sum\limits_{n_1\leq j \leq n_2} R_{4j}(t)e_j\right\|_{D\left(L^{-\beta (q-\widehat{q})}\right)}
\leq  \int\limits_0^T \Bigg\{\sum_{i=1}^{\infty} \lambda_i^{-2 \beta(p-\widehat{q})}\Bigg(\sum_{n_1 \leq j \leq n_2} \lambda_j^\beta k(\tau) u_j(\tau)
$$$$
\times (T-\tau)^{\alpha-1} E_{\alpha, \alpha}\left(-\lambda_j^\beta (T-\tau)^\alpha\right)  \Psi_j(T)  E_{\alpha, 1}\left(-\lambda_j^\beta t^\alpha\right)e_j, e_i\Bigg)_{-\beta(q-\widehat{q}), \beta(q-\widehat{q})}^2\Bigg\}^{1 / 2} d \tau
 $$
$$
 \leq  \frac{C_{\kappa} M_2}{\lambda_1^\beta}   t^{-\alpha} \int\limits_0^T \Bigg\{\sum_{n_1 \leq j \leq n_2} \lambda_j^{2 \beta(p+\widehat{q})}  k^2(\tau) u_j^2(\tau)  (T-\tau)^{2\alpha(q-\widehat{q})-2} \lambda_j^{-2 \beta(p+\widehat{q})} \Bigg\}^{1 / 2} d \tau
$$
$$
\leq  \frac{C_{\kappa} M_2}{\lambda_1^\beta} \|k\|_{C[0,T]}  t^{-\alpha} \int\limits_0^T (T-\tau)^{\alpha(q-\widehat{q})-1}  \Bigg\{\sum_{n_1 \leq j \leq n_2} u_j^2(\tau)    \Bigg\}^{1 / 2} d \tau.
$$

Similar to the price of $R_{1j}(t)$ above, we estimate $R_{5j}(t)$
$$
\left\|\sum\limits_{n_1\leq j \leq n_2} R_{4j}(t)e_j\right\|_{D\left(L^{-\beta (q-\widehat{q})}\right)}
$$
$$
\leq   \int\limits_0^t\left\{\sum_{n_1 \leq j \leq n_2} \lambda_j^{2 \beta(p+\widehat{q})}  k^2(\tau) u_j^2(\tau)  (t-\tau)^{2\alpha-2} \left(E_{\alpha, \alpha}\left(-\lambda_j^\beta (t-\tau)^\alpha\right)\right)^2\right\}^{1 / 2} d \tau
$$
$$
\leq M_2 \|k\|_{C[0,T]} \int\limits_0^t (t-\tau)^{\alpha(q-\widehat{q})-1}
\left\{\sum_{n_1 \leq j \leq n_2}   u_j^2 (\tau)    \right\}^{1 / 2} d \tau.
$$
For $0<\tau<T$, we have $u(\tau,\cdot)$  belonging to $L_2(\Omega)$. This follows that the sequence $\left\{\sum_{j \geq n}    u_j^2(\tau)  \right\}^{1 / 2}$, converges pointwise to $0$ as $n$ goes to infinity.
Therefore, the dominated convergence theorem yields that
$$
\lim _{n \rightarrow \infty}  \int\limits_0^t (t-\tau)^{\alpha(q-\widehat{q})-1}
\left\{\sum_{ j \leq n}    u_j^2(\tau)
 \right\}^{1 / 2} d \tau =0,
$$
$$
\lim _{n \rightarrow \infty}  \int\limits_0^T (T-\tau)^{\alpha(q-\widehat{q})-1}
\left\{\sum_{ j \leq n}    u_j^2(\tau)
 \right\}^{1 / 2} d \tau =0.
$$
According to the  Sobolev imbedding
$D\left(L^{\beta(q+\widehat{q})}\right) \hookrightarrow D\left(L^{\beta p}\right)$ and (\ref{eqth96}), we have
$$
\left\|\sum\limits_{n_1\leq j \leq n_2}J_{4j}(t)e_j\right\|_{D\left(L^{-\beta (q-\widehat{q})}\right)}
$$
\begin{equation}\label{eqth123thth}
\leq \frac{C_{\kappa} M_2}{\lambda_1^\beta}
\chi_0  \|\varphi\|_{D\left(L^{\beta p}\right)}  \|k\|_{C[0,T]} T^{\alpha-\alpha \widehat{q}} B(1-\alpha q,\alpha(q-\widehat{q})) t^{-\alpha }
\end{equation}
$$
\left\|\sum\limits_{n_1\leq j \leq n_2}J_{5j}(t)e_j\right\|_{D\left(L^{-\beta (q-\widehat{q})}\right)}
$$
\begin{equation}\label{eqth124}
\leq  M_2   \chi_0  \|\varphi\|_{D\left(L^{\beta p}\right)}  \|k\|_{C[0,T]} T^{\alpha-\alpha \widehat{q}} B(1-\alpha q,\alpha(q-\widehat{q})) t^{-\alpha }.
\end{equation}
From the above, we get
$$
\lim\limits_{n \rightarrow \infty} \int\limits_0^t (t-\tau)^{\alpha(q-\widehat{q})-1}
\left\{\sum_{n_1 \leq j \leq n_2}   u_j^2 (\tau)    \right\}^{1 / 2} d \tau=0,
$$
$$
\lim\limits_{n \rightarrow \infty} \int\limits_0^T (T-\tau)^{\alpha(q-\widehat{q})-1}
\left\{\sum_{n_1 \leq j \leq n_2}   u_j^2 (\tau)    \right\}^{1 / 2} d \tau=0.
$$

We deduce $
\left\|\sum\limits_{n_1\leq j \leq n_2} R_{ij}(t)e_j\right\|_{D\left(L^{-\beta (q-\widehat{q})}\right)},\,\, i=1,2,3,4,5
$ exist finitely.

 In addition, the Sobolev embedding  $L_2(\Omega) \hookrightarrow D\left(L^{-\beta(q-\widehat{q})}\right)$ yields that there exists a positive constant $C_8$ such that
$$
\left\| F(t, \cdot, u(t, \cdot)) \right\|_{D\left(L^{-\beta (q-\widehat{q})}\right)} \leq  C_8 \left\| F(t, \cdot, u(t, \cdot)) \right\|
$$
According to (\ref{eqth121}), we have
$$
\left\| F(t, \cdot, u(t, \cdot)) \right\|_{D\left(L^{-\beta (q-\widehat{q})}\right)} \leq  C_8  C_7 \Upsilon^{*} \chi_0  \|\varphi\|_{D\left(L^{\beta(p+\widehat{q})}\right)} t^{-\alpha q}.
$$

Moreover, the Sobolev embedding $L_2(\Omega) \hookrightarrow D\left(L^{-\beta(q-\widehat{q})}\right)$ yields that there exists a positive constant $M_3$ such that
$$
\left\|k u\right\|_{D\left(L^{-\beta (q-\widehat{q})}\right)} \leq M_3 \left\|k u\right\|_{L_2(\Omega)}.
$$
Hence $\sum_{j=1}^n \partial_t^\alpha u_j(t) e_j$ is a Cauchy sequence and a convergent sequence in $D\left(L^{-\beta (q-\widehat{q})}\right)$. Taking all the above arguments together, we conclude that $\left\|\sum_{j=1}^n \partial_t^\alpha u_j(t) e_j\right\|_{D\left(L^{-\beta (q-\widehat{q})}\right)}
$ finitely exists.
Thus, by combining the above arguments, we get
$$
\left\|\partial_t^\alpha u(t, \cdot)\right\|_{D\left(L^{-\beta (q-\widehat{q})}\right)}\leq  \left\|F(t, \cdot, u(t, \cdot))\right\|_{D\left(L^{-\beta (q-\widehat{q})}\right)}+
 \left\|k(t) u(t, \cdot)\right\|_{D\left(L^{-\beta (q-\widehat{q})}\right)}
$$
$$
+\sum\limits_{i=1}^{5}\left\| \sum\limits_{j=1}^{\infty} R_{ij}(t)e_j \right\|_{D\left(L^{-\beta (q-\widehat{q})}\right)}.
$$
This combines with (\ref{eqth121}) and there exists a constant $C_7>0$ such that
\begin{equation}\label{eqth125}
\left\|\partial_t^\alpha u(t, \cdot)\right\|_{D\left(L^{-\beta (q-\widehat{q})}\right)}\leq  C_7  \|\varphi\|_{D\left(L^{\beta(p+\widehat{q})}\right)} t^{-\alpha },
\end{equation}
as a result of this
$$
\left\|\partial_t^\alpha u\right\|_{L^{\frac{1}{\alpha}-\widehat{r}}\left(0, T ; D\left(L^{-\beta (q-\widehat{q})}\right)\right)} \leq  C_7 \left\| t^{-\alpha}\right\|_{L^{\frac{1}{\alpha}-\widehat{r}}\left(0, T ; \mathbb{R}\right)}  \|\varphi\|_{D\left(L^{\beta(p+\widehat{q})}\right)}.
$$
This is the inequality (\ref{ththeq121}) from Theorem \ref{theorth6}.
\end{proof}

 \begin{theorem}\label{theorth7}
Let $p, q, p', q', \widehat{p}, \widehat{q}, r, \widehat{r}$ be defined by A1), B2), B4). If $\varphi \in D\left(L^{\beta(p+\widehat{q})}\right)$, $F$ satisfies the assumptions C2), and $\Theta(T)<1$, then nonlocal initial problem (\ref{eqth1})--(\ref{eqth3}) has a unique solution $u(t, x)$ satisfying that
$
\partial_t^\alpha u \in  C^\alpha\left((0, T]; D\left(L^{-\beta q}\right)\right)
$
and there exists a constant $C_9>0$ such that
\begin{equation}\label{eqth127}
\left\|\partial_t^\alpha u \right\|_{C^{\alpha}\left((0, T];D\left(L^{-\beta q}\right)\right)} \leq C_9 \|\varphi\|_{D\left(L^{\beta (p+\widehat{q})}\right)}
\end{equation}
\end{theorem}

\begin{proof}
 We prove $\partial_t^\alpha u \in C^\alpha\left((0, T]; D\left(L^{-\beta q}\right) \right)$. We consider $0<t_1<t_2 \leq T$. A similar argument as in (3.38) yields
$$
\partial_t^\alpha u(t_2, x) - \partial_t^\alpha u(t_1, x)
$$$$
=F(t_2, x, u(t_2, x)) - F(t_1, x, u(t_1, x))
-(k(t_2)u(t_2, x) - k(t_1)u(t_1, x))
$$
$$
-\sum_{j=1}^{\infty} \lambda_j^\beta \int\limits_0^{t_1} \int\limits_{t_1-\tau}^{t_2-\tau} F_j(\tau,u(\tau)) \eta^{\alpha-2} E_{\alpha, \alpha-1}\left(-\lambda_j^\beta \eta^\alpha\right) d \eta d \tau e_j(x)
$$
$$
-\sum_{j=1}^{\infty} \lambda_j^\beta \int\limits_{t_1}^{t_2} F_j(\tau,u(\tau)) (t_2-\tau)^{\alpha-1} E_{\alpha, \alpha}\left(-\lambda_j^\beta (t_2-\tau)^{\alpha}\right) d \tau  e_j(x)
$$
$$
+\sum_{j=1}^{\infty} \lambda_j^{2\beta} \Psi_j(T)  \varphi_j  \int\limits_{t_1}^{t_2} \tau^{\alpha-1} E_{\alpha, \alpha}\left(-\lambda_j^\beta \tau^\alpha\right) d \tau e_j(x)
$$
$$
-\sum_{j=1}^{\infty} \lambda_j^{2\beta} \Psi_j(T)   F_j(T,u(T)) \star \left\{T^{\alpha-1} E_{\alpha, \alpha}\left(-\lambda_j^\beta T^\alpha\right)\right\}
$$$$
\times \int\limits_{t_1}^{t_2} \tau^{\alpha-1} E_{\alpha, \alpha}\left(-\lambda_j^\beta \tau^\alpha\right) d \tau e_j(x)
$$
$$
+\sum_{j=1}^{\infty} \lambda_j^{2\beta} \Psi_j(T)   \{k(T) u_j(T)\} \star \left\{T^{\alpha-1} E_{\alpha, \alpha}\left(-\lambda_j^\beta T^\alpha\right)\right\}
$$$$
\times \int\limits_{t_1}^{t_2} \tau^{\alpha-1} E_{\alpha, \alpha}\left(-\lambda_j^\beta \tau^\alpha\right) d \tau e_j(x)
$$
$$
+\sum_{j=1}^{\infty} \lambda_j^\beta \int\limits_0^{t_1} \int\limits_{t_1-\tau}^{t_2-\tau} k(\tau) u_j(\tau) \eta^{\alpha-2} E_{\alpha, \alpha-1}\left(-\lambda_j^\beta \eta^\alpha\right) d \eta d \tau e_j(x)
$$
$$
+\sum_{j=1}^{\infty} \lambda_j^\beta \int\limits_{t_1}^{t_2} k(\tau) u_j(\tau) (t_2-\tau)^{\alpha-1} E_{\alpha, \alpha}\left(-\lambda_j^\beta (t_2-\tau)^{\alpha}\right) d \tau  e_j(x)
$$
\begin{equation}\label{eqth128}
:=F(t_2, x, u(t_2, x)) - F(t_1, x, u(t_1, x))-(k(t_2)u(t_2, x) - k(t_1)u(t_1, x))+\sum\limits_{j=1}^{7} {}^1\mathcal{R}_j.
\end{equation}
where ${}^1\mathcal{R}_j=L^\beta \mathcal{R}_j,\, j=\overline{1,7}$ and $\mathcal{R}_j,\, j=\overline{1,7}$ is defined by (\ref{eqth102}).
Using Sobolev embedding $L_2(\Omega) \hookrightarrow  D\left(L^{-\beta q}\right)$ and the condition C2),  there exists a positive constant $C_9$ such that
$$
\lim _{t_2-t_1 \rightarrow 0} \| F   \left(t_2, ., u\left(t_2, .\right)\right)-F\left(t_1, ., u\left(t_1, .\right)\right) \|_{D\left(L^{-\beta q}\right)}
$$
$$
\leq \lim _{t_2-t_1 \rightarrow 0} C_9 \left\|F\left(t_2, \cdot, u(t_2, \cdot)\right)-F\left(t_1, ., u(t_1,\cdot)\right)\right\|
$$
$$
\leq \lim _{t_2-t_1 \rightarrow 0} C_9 \Upsilon^{*}\left(\left|t_2-t_1\right|+\left\|u(t_2, \cdot)-u(t_1, \cdot)\right\|\right).
$$
Considering the  note $u \in C^{\alpha q}\left((0, T]; L_2(\Omega) \right)$, we get
\begin{equation}\label{thththth129}
 \lim _{t_2-t_1 \rightarrow 0} \| F   \left(t_2, ., u\left(t_2, .\right)\right)-F\left(t_1, ., u\left(t_1, .\right)\right) \|_{D\left(L^{-\beta q}\right)}=0.
\end{equation}

Next, consider the limit of the difference $k(t_2)u(t_2, x) - k(t_1)u(t_1, x)$ when $t_2-t_1 \rightarrow 0$. Let's first examine the following inequality
$$
\left|k(t_2)u(t_2, x) - k(t_1)u(t_1, x)\right|\leq \left|k(t_2)\right|\left|u(t_2, x) -u(t_1, x)\right|+
\left| u(t_1, x)\right|\left|k(t_2) - k(t_1) \right|.
$$
Considering the  note $u \in C^{\alpha q}\left((0, T]; L_2(\Omega) \right)$ and $k(t)\in C[0,T]$, using the above inequality, we get
$$
\lim _{t_2-t_1 \rightarrow 0} \left\|k(t_2)u(t_2, x) - k(t_1)u(t_1, x)\right\|_{D\left(L^{-\beta q}\right)}
$$$$
\leq \lim _{t_2-t_1 \rightarrow 0}   \left\|k(t_2)-k(t_1)\right\| \|u(t_2, x)\|_{D\left(L^{-\beta q}\right)}
$$
\begin{equation}\label{thththth130}
+ \lim _{t_2-t_1 \rightarrow 0}  \|k(t_1)\| \left\|u(t_2, x)-u(t_1, x)\right\|_{D\left(L^{-\beta q}\right)}=0.
\end{equation}

Let us first consider $\left\|{}^1\mathcal{R}_1\right\|_{D\left(L^{-\beta q}\right)}$. From (\ref{est67}) and (\ref{eqth121}), we have
$$
\left\|{}^1\mathcal{R}_1\right\|_{D\left(L^{-\beta q}\right)}
$$$$
\leq \int_0^{t_1}\left\|-\sum_{j=1}^{\infty} \lambda_j^\beta  \int\limits_{t_1-\tau}^{t_2-\tau} F_j(\tau, u(\tau)) \eta^{\alpha-2} E_{\alpha, \alpha-1}\left(-\lambda_j^\beta \eta^\alpha\right) d \eta   e_j\right\|_{D\left(L^{-\beta q}\right)} d \tau
$$
$$
 \leq \int_0^{t_1}\left\{\sum_{j=1}^{\infty} \lambda_j^{2 \beta} \lambda_j^{-2 \beta q} F_j^2(\tau,u(\tau))\left|\int\limits_{t_1-\tau}^{t_2-\tau} \eta^{\alpha-2} E_{\alpha, \alpha-1}\left(-\lambda_j^\beta \eta^\alpha\right)|d \eta\right|^2\right\}^{1 / 2} d \tau
$$$$
  \leq M_2 \int_0^{t_1}\left\{\sum_{j=1}^{\infty} \lambda_j^{2 \beta} \lambda_j^{-2 \beta q} F_j^2(\tau,u(\tau))\left|\int\limits_{t_1-\tau}^{t_2-\tau} \eta^{\alpha-2} \lambda_j^{-\beta p} \eta^{-\alpha p} d \eta\right|^2\right\}^{1 / 2} d \tau
$$
$$
\leq  \frac{M_2}{1-\alpha q}  \int\limits_0^{t_1}  \|F(\tau, \cdot)\| \left( (t_1-\tau)^{\alpha q-1} -(t_2-\tau)^{\alpha q-1}\right) d \tau.
$$
$$
\leq  \frac{M_2 C_7 \Upsilon^{*}}{1-\alpha q}  \chi_0 \|\varphi\|_{D\left(L^{\beta(p+\widehat{q})}\right)}  \int\limits_0^{t_1}  \tau^{-\alpha q} \left( (t_1-\tau)^{\alpha q-1} -(t_2-\tau)^{\alpha q-1}\right) d \tau.
$$
Using (\ref{thesth104}) for the integral participating in the right-hand side of the above inequality, we obtain
\begin{equation}\label{est129}
\left\|{}^1\mathcal{R}_1\right\|_{D\left(L^{-\beta q}\right)}  \leq  \frac{M_2 C_7 \Upsilon^{*}}{\alpha q(1-\alpha q)}  \chi_0 \|\varphi\|_{D\left(L^{\beta(p+\widehat{q})}\right)}\frac{(t_2-t_1)^{\alpha q}}{t_1^{\alpha q}}.
\end{equation}

Let us secondly consider $\left\|{}^1\mathcal{R}_2\right\|_{D\left(L^{-\beta q}\right)}$. We have
$$
\left\|{}^1\mathcal{R}_2\right\|_{D\left(L^{-\beta q}\right)}  \leq
$$$$
\int\limits_{t_1}^{t_2}\left\|\sum_{j=1}^{\infty} \lambda_j^\beta  F_j(\tau, \cdot,u(\tau, \cdot)) (t_2-\tau)^{\alpha-1} E_{\alpha, \alpha}\left(-\lambda_j^\beta (t_2-\tau)^{\alpha}\right)  e_j\right\|_{D\left(L^{-\beta q}\right)} d \tau
$$$$
\leq M_2 \int\limits_{t_1}^{t_2}\left\{\sum_{j=1}^{\infty} \lambda_j^{2 \beta} \lambda_j^{-2 \beta q} F_j^2(\tau, \cdot,u(\tau, \cdot)) \lambda_j^{-2 \beta p}\left(t_2-\tau\right)^{-2 \alpha p}\right\}^{1 / 2}\left(t_2-\tau\right)^{\alpha-1} d \tau
$$
$$
\leq M_2 \int\limits_{t_1}^{t_2}\|F(\tau, \cdot,u(\tau, \cdot))\|\left(t_2-\tau\right)^{\alpha q-1} d \tau.
$$
From (\ref{eqth121}), we have
$$
\left\|{}^1\mathcal{R}_2\right\|_{D\left(L^{-\beta q}\right)}  \leq   M_2 C_7 \Upsilon^{*} \chi_0 \|\varphi\|_{D\left(L^{\beta(p+\widehat{q})}\right)} \int\limits_{t_1}^{t_2} \tau^{-\alpha q} \left(t_2-\tau\right)^{\alpha q-1} d \tau.
$$
Next, let's study the integral $\int\limits_{t_1}^{t_2} \tau^{-\alpha q} \left(t_2-\tau\right)^{\alpha q-1} d \tau$
\begin{equation}\label{ththth130}
\int\limits_{t_1}^{t_2} \tau^{-\alpha q} \left(t_2-\tau\right)^{\alpha q-1} d \tau \leq t_1^{-\alpha q}\int\limits_{t_1}^{t_2}  \left(t_2-\tau\right)^{\alpha q-1} d \tau =\frac{t_1^{-\alpha q}}{\alpha q} (t_2-t_1)^{\alpha q}.
\end{equation}
From this we obtain
\begin{equation}\label{est130}
\left\|{}^1\mathcal{R}_2\right\|_{D\left(L^{-\beta q}\right)}  \leq  \frac{M_2 C_7 \Upsilon^{*}}{\alpha q}  \chi_0  t_1^{-\alpha q}\|\varphi\|_{D\left(L^{\beta(p+\widehat{q})}\right)}  (t_2-t_1)^{\alpha q}.
\end{equation}

Thirdly, we consider the norm $\left\|{}^1\mathcal{R}_3\right\|_{D\left(L^{-\beta q}\right)} $. This norm  is the same as the estimate in (\ref{est69}).

The norm $\left\|{}^1\mathcal{R}_4\right\|_{D\left(L^{-\beta q}\right)} $ can be estimated as follows
$$
\left\|{}^1\mathcal{R}_4\right\|_{D\left(L^{-\beta q}\right)}  \leq  C_{\kappa} M_2^2 \int\limits_0^T \Bigg\{
\sum_{j=1}^{\infty} \lambda_j^{4 \beta} \lambda_j^{-2 \beta q}
F_j^2(\tau, \cdot,u(\tau, \cdot))
$$$$
\lambda_j^{-\beta p} (T-\tau)^{2\alpha-2}
(T-\tau)^{-2\alpha p}\left[\int\limits_{t_1}^{t_2} \eta^{\alpha-1} \lambda_j^{-\beta } \eta^{-\alpha} d \eta\right]^2 \Bigg\}^{1 / 2}  d \tau
$$
$$
\leq  \frac{C_{\kappa} M_2^2 }{\alpha  }   t_1^{-\alpha} \int\limits_0^T \Bigg\{
\sum_{j=1}^{\infty} F_j^2(\tau, \cdot,u(\tau, \cdot)) (T-\tau)^{2\alpha q-2}
 \left[t_2^{\alpha }-t_1^{\alpha } \right]^2 \Bigg\}^{1 / 2}  d \tau
$$
$$
\leq  \frac{C_{\kappa} M_2^2}{\alpha  }   t_1^{-\alpha}  \left(t_2^{\alpha }-t_1^{\alpha }\right)  \int\limits_0^T  \|F(\tau, \cdot,u(\tau, \cdot))\| (T-\tau)^{\alpha q-1}    d \tau.
$$
From (\ref{eqth121}), we obtain an estimate
\begin{equation}\label{est131}
\left\|{}^1\mathcal{R}_4\right\|_{D\left(L^{-\beta q}\right)} \leq  \frac{C_{\kappa} M_2^2}{\alpha  }   t_1^{-\alpha}  \left(t_2^{\alpha }-t_1^{\alpha }\right)  C_7 \Upsilon^{*} \chi_0 \|\varphi\|_{D\left(L^{\beta(p+\widehat{q})}\right)} B(\alpha q,1-\alpha q).
\end{equation}

Next, we get for ${}^1\mathcal{R}_5$
$$
\left\|{}^1\mathcal{R}_5\right\|_{D\left(L^{-\beta q}\right)}  \leq  C_{\kappa} M_2^2 \int\limits_0^T \Bigg\{
\sum_{j=1}^{\infty} \lambda_j^{4 \beta} \lambda_j^{-2 \beta q} k^2(\tau)
u_j^2(\tau)
$$$$
\lambda_j^{-\beta p} (T-\tau)^{2\alpha-2}
(T-\tau)^{-2\alpha p}\left[\int\limits_{t_1}^{t_2} \eta^{\alpha-1} \lambda_j^{-\beta } \eta^{-\alpha} d \eta\right]^2 \Bigg\}^{1 / 2}  d \tau
$$
$$
\leq  \frac{C_{\kappa} M_2^2 }{\alpha  }   t_1^{-\alpha} \int\limits_0^T \Bigg\{
\sum_{j=1}^{\infty} k^2(\tau)
u_j^2(\tau) (T-\tau)^{2\alpha q-2}
 \left[t_2^{\alpha }-t_1^{\alpha } \right]^2 \Bigg\}^{1 / 2}  d \tau
$$
$$
\leq  \frac{C_{\kappa} M_2^2}{\alpha  } \|k\|_{C[0,T]}  t_1^{-\alpha}  \left(t_2^{\alpha }-t_1^{\alpha }\right)  \int\limits_0^T  \|u(\tau, \cdot)\| (T-\tau)^{\alpha q-1}    d \tau.
$$
\begin{equation}\label{est131thth}
\leq  \frac{C_{\kappa} M_2^2}{\alpha  } \|k\|_{C[0,T]}  t_1^{-\alpha}  \left(t_2^{\alpha }-t_1^{\alpha }\right)  \chi_0 \|\varphi\|_{D\left(L^{\beta(p+\widehat{q})}\right)} B(\alpha q,1-\alpha q).
\end{equation}

Now let's estimate ${}^1\mathcal{R}_6$ in the same way
$$
\left\|{}^1\mathcal{R}_6\right\|_{D\left(L^{-\beta q}\right)}
\leq M_2 \|k\|_{C[[0,T]} \int_0^{t_1}\left\{\sum_{j=1}^{\infty}   u_j^2(\tau) \left|\int\limits_{t_1-\tau}^{t_2-\tau} \omega^{\alpha-2}   \omega^{-\alpha p} d \omega\right|^2\right\}^{1 / 2} d \tau
$$

$$
\leq M_2 \|k\|_{C[[0,T]} \int\limits_0^{t_1} \|u(\tau, \cdot)\| \int\limits_{t_1-\tau}^{t_2-\tau} \eta^{\alpha q-2} d \eta  d \tau.
$$
According to (\ref{eqth96}), from the above inequality we obtain the following inequality
$$
\left\|{}^1\mathcal{R}_6\right\|_{D\left(L^{-\beta q}\right)}
\leq   \frac{M_2  \chi_0}{1-\alpha q} \|k\|_{C[[0,T]}  \|\varphi\|_{D\left(L^{\beta p}\right)} \int\limits_0^{t_1} \tau^{-\alpha q} \Big[ (t_1-\tau)^{\alpha q-1} - (t_2-\tau)^{\alpha q-1}\Big]  d \tau.
$$
By using (\ref{thesth104}) in the above inequality, we obtain
\begin{equation}\label{esth132}
\left\|{}^1\mathcal{R}_6\right\|_{D\left(L^{-\beta q}\right)}
\leq   \frac{M_2  \chi_0}{\alpha q (1-\alpha q)} \|k\|_{C[[0,T]}  t_1^{-\alpha q}\|\varphi\|_{D\left(L^{\beta p}\right)} (t_2-t_1)^{\alpha q}.
\end{equation}

Let's estimate ${}^1\mathcal{R}_7:$
 $$
\left\|{}^1\mathcal{R}_7\right\|_{D\left(L^{-\beta q}\right)}  \leq
$$$$
\int\limits_{t_1}^{t_2}\left\|\sum_{j=1}^{\infty} \lambda_j^\beta k(\tau) u_j(\tau) (t_2-\tau)^{\alpha-1} E_{\alpha, \alpha}\left(-\lambda_j^\beta (t_2-\tau)^{\alpha}\right)  e_j\right\|_{D\left(L^{-\beta q}\right)} d \tau
$$$$
\leq M_2 \|k\|_{C[[0,T]} \int\limits_{t_1}^{t_2}\left\{\sum_{j=1}^{\infty} \lambda_j^{2 \beta} \lambda_j^{-2 \beta q} u_j^2(\tau) \lambda_j^{-2 \beta p}\left(t_2-\tau\right)^{-2 \alpha p}\right\}^{1 / 2}\left(t_2-\tau\right)^{\alpha-1} d \tau
$$
\begin{equation*}
\leq M_2  \|k\|_{C[[0,T]}  \int\limits_{t_1}^{t_2}\|u(\tau, \cdot)\|\left(t_2-\tau\right)^{\alpha q-1}  d \tau.
\end{equation*}
By using (\ref{eqth96}) and (\ref{ththth130})
\begin{equation}\label{est136}
\left\|{}^1\mathcal{R}_7\right\|_{D\left(L^{-\beta q}\right)}  \leq \frac{M_2 t_1^{-\alpha q}}{\alpha q} \|k\|_{C[[0,T]}  \chi_0  \|\varphi\|_{D\left(L^{\beta p}\right)}  (t_2-t_1)^{\alpha q}.
\end{equation}

It follows from (\ref{eqth128}),  and  (\ref{thththth129}) - (\ref{est136}) that $\partial_t^\alpha u(t, x)$ is in the class $ C\left((0, T];D\left(L^{-\beta q}\right)\right)$. On the other hand, the estimate (\ref{eqth125}) also holds for $\widehat{q}=0$ and $\widehat{p}=1$, i.e., for $0<t\leq T$ we get
\begin{equation}\label{eqth137}
t^{\alpha } \left\|\partial_t^\alpha u(t, \cdot)\right\|_{D\left(L^{-\beta q}\right)}\leq  C_7  \|\varphi\|_{D\left(L^{\beta p}\right)}.
\end{equation}
According to the Sobolev embedding $ D\left(L^{\beta(p+\widehat{q})}\right) \hookrightarrow  D\left(L^{\beta p}\right)$, we obtain
$$
 \left\|\partial_t^\alpha u \right\|_{C^{\alpha}\left((0, T];D\left(L^{-\beta q}\right)\right)} \leq C_9 \|\varphi\|_{D\left(L^{\beta (p+\widehat{q})}\right)}.
$$
From this, we obtain the result (\ref{eqth127}) from Theorem \ref{theorth7}.
\end{proof}

\begin{lemma}\label{lemmath9}
Let $p, q, p', q', \widehat{p}, \widehat{q}, r, \widehat{r}$ be defined by A1), B2), B4). If $\varphi \in D\left(L^{\beta(p+\widehat{q})}\right)$, $F$ satisfies the assumptions C2), and $\Theta(T)<1$, then nonlocal initial problem (\ref{eqth1})--(\ref{eqth3}) has a unique solution $u(t, x)$ satisfying that
$$
\partial_t^\alpha u \in L^{\frac{1}{\alpha \widehat{q}}-\widehat{r}}\left(0, T ;
D\left(L^{-\beta(p-\widehat{q})}\right)\right) \cap C^\alpha\left((0, T]; D\left(L^{-\beta q}\right)\right)
$$
Moreover, there exists a constant $C_7>0$ such that
\begin{equation}\label{ththeq138}
 \left\|\partial_t^\alpha u\right\|_{L^{\frac{1}{\alpha }-\widehat{r}}\left(0, T ;
D\left(L^{-\beta(q-\widehat{q})}\right)\right)}+ \left\|\partial_t^\alpha u \right\|_{C^{\alpha}\left((0, T];D\left(L^{-\beta q}\right)\right)} \leq  C_{10}  \|\varphi\|_{D\left(L^{\beta(p+\widehat{q})}\right)}.
\end{equation}
\end{lemma}

The proof of the Lemma \ref{lemmath9} is proved by combining the inequalities (\ref{eqth127}) and (\ref{ththeq121}).

\section*{5 Investigation of the inverse problem (\ref{eqth1})--(\ref{eqth4})}

Let us now investigate the inverse problem (\ref{eqth1})--(\ref{eqth4}).  Let $k(t)$ be an arbitrary function of $C[0,T]$.

Now, let's check the continuity of the fractional operator $L^\beta u(t, x)$ in equation (\ref{eqth1}).

According to equation (\ref{eqth1}), we obtain
\begin{equation}\label{ththeq143}
L^\beta u(t, x)=F(t, x, u(t, x))-
\partial_t^\alpha u(t, x)-k(t) u(t, x)
\end{equation}
Using Theorems \ref{theorth5}, \ref{theorth7} and equality  (\ref{eqth127}) we estimate $L^\beta u(t, x)$
$$
\|L^\beta u(t, x)\|_{C^{\alpha q}\left([0,T];D\left(L^{-\beta q'}\right)\right)} \leq \Upsilon \|u(t, x)\|_{C^{\alpha q}\left([0,T];D\left(L^{-\beta q'}\right)\right)}
$$$$
+\|\partial_t^\alpha u(t, x)\|_{C^{\alpha q}\left([0,T];D\left(L^{-\beta q'}\right)\right)}+\|k\|_{C[0,T]} \|u(t, x)\|_{C^{\alpha q}\left([0,T];D\left(L^{-\beta q'}\right)\right)}
$$

$$
\leq  M_2 \left(\Upsilon+\|k\|_{C[0,T]} \right)\Bigg[ \frac{T^{\alpha(p-q-p')} \chi_0  \Upsilon}{\alpha q(1-\alpha q')} +  \frac{ M_6  t_1^{-\alpha q}}{\alpha q \lambda_1^{\beta p'}}  \chi_0  \Upsilon+ \frac{C_{\kappa}  }{\alpha q' \lambda_1^{2 \beta p}} T^{\alpha (p-p')}
$$
$$
 +\frac{ C_{\kappa} M_2 }{\alpha q  \lambda_1^{\beta q}} T^{\alpha (p-p')}   \chi_0  \Upsilon B(\alpha q',1-\alpha q)
+
\frac{\|k\|_{C[[0,T]}t_1^{-\alpha q}}{ (1-\alpha q \lambda_1^{\beta (p+q')}}
\chi_0  +\frac{\|k\|_{C[[0,T]} }{ \lambda_1^{\beta (p+q')}} t_1^{-\alpha q} \chi_0
$$
$$
+\frac{ C_{\kappa} M_2 }{\alpha q  \lambda_1^{\beta q}} T^{\alpha (p-p')}   \chi_0  \|k\|_{C[0,T]}   B(\alpha q',1-\alpha q) \Bigg]\|\varphi\|_{D\left(L^{\beta p}\right)}+ C_9 \|\varphi\|_{D\left(L^{\beta (p+\widehat{q})}\right)}.
$$

On the other hand, it is easy to see that
$$
L^\beta u(t, x_0)=L^\beta\mathcal{G}_2(t, x_0) \varphi+L^\beta\mathcal{G}_1(t, x_0) F +L^\beta\mathcal{G}_3(t, x_0) F
$$
\begin{equation}\label{thththeq145}
-L^\beta\mathcal{G}_3(t, x_0) \{k,  u\}-L^\beta\mathcal{G}_1(t, x_0)\{k,  u\}:=\mathcal{Q}_1+\mathcal{Q}_2+\mathcal{Q}_3+\mathcal{Q}_4+\mathcal{Q}_5. \end{equation}

The following lemma is valid.
\begin{lemma}\label{lemma7}
 Let $\varphi \in D\left(L^{\gamma_0}\right)$ and $k \in C[0,T]$. Then there exists a positive constant $C_11$ such that
 $$
 \|L^\beta u(t, x_0)\|_{C[0,T]} \leq  C_{11}  \Bigg(C_\kappa+\frac{\Upsilon}{\lambda_1^{\beta p}}\chi_0  B(\alpha q, 1-\alpha q)+\frac{M_2 C_\kappa \Upsilon}{\lambda_1^{\beta p}}\chi_0  B(\alpha q, 1-\alpha q)
 $$
\begin{equation}\label{eqtthth145}
+\frac{M_2 C_\kappa\|k\|_{C[0,T]}+ \|k\|_{C[0,T]}}{\lambda_1^{\beta p}}   \chi_0  B(\alpha q, 1-\alpha q)\Bigg)\|\varphi\|_{D\left(L^{\gamma_0}\right)},
\end{equation}
where $\gamma_0>\beta+\frac{\mathbb{N}}{2},$ $C_{11}$ is dependent on $\Omega, \alpha, \beta, \gamma_0, \mathbb{N}, \lambda_1$.
\end{lemma}
\begin{proof}
 An inequality similar to the estimate in (\ref{eqtthth145}) was derived in [30]. However, the smoothness assumptions differ from those in [30], so we provide a detailed proof of inequality (\ref{eqtthth145}).

Recall the following inequality for the fractional power $L^\beta$ of $L$:
$$
\|u\|_{H^{2\beta}(\Omega)} \leq c_0 \|L^\beta u\|_{L_2(\Omega)}
$$
where constant $c_0>0$ depends only on $\beta$ and $\lambda_1$ (see., [[31], p. 208])

Let $2 \varepsilon=\gamma_0-\beta-\frac{\mathbb{N}}{2}$ and $\gamma_0>\beta+\frac{\mathbb{N}}{2}$. Furthermore we set $\theta=\frac{\mathbb{N}}{4}+\varepsilon$. Then by $H^{2 \theta} \subset C(\bar{\Omega})$ due to $4 \theta>\mathbb{N}$, we have
\begin{equation}\label{thththeq146}
\left\|e_n\right\|_{C(\bar{\Omega})} \leq C(\Omega)\left\|e_n\right\|_{H^{2 \theta}(\Omega)} \leq C_\Omega\left\|L^\theta e_n\right\|_{L^2(\Omega)} \leq C_\Omega \lambda_n^\theta .
\end{equation}

For $\mathcal{Q}_1$, by inequalities in Proposition \ref{thproth3}, and noticing that $\lambda_n=O\left(n^{2 / \mathbb{N}}\right)$ (see [[32], pp. 355,356]), we have
$$
|\mathcal{Q}_1| \leq \left|L^\beta\mathcal{G}_2(t, x_0) \varphi \right| =  \left|\sum_{j=1}^{\infty} \lambda_j^\beta \varphi_j \Psi_j(T)E_{\alpha, 1}\left(-\lambda_j^\beta t^\alpha\right) e_j(x_0)  \right|
$$
$$
\leq C_\kappa \sum_{j=1}^{\infty} \lambda_j^\beta |\varphi_j | |E_{\alpha, 1}\left(-\lambda_j^\beta t^\alpha\right)| |e_j(x_0)|\leq C_\kappa M_2 \sum_{j=1}^{\infty}\lambda_j^\beta |\varphi_j |  \frac{1}{1+\lambda_j^\beta t^\alpha}  |e_j(x_0)|
$$
$$
\leq C_\kappa M_2 C_\Omega \sum_{j=1}^{\infty}\lambda_j^\beta \lambda_j^{\gamma_0} |(\varphi,e_j) |  \frac{1}{1+\lambda_j^\beta t^\alpha} \lambda_j^{-\gamma_0+\theta}   \leq C_\kappa M_2 C_\Omega \sum_{j=1}^{\infty} \lambda_j^{\gamma_0} |(\varphi,e_j) |  \lambda_j^{-(\gamma_0-\beta-\theta)}
$$
$$
\leq C_\kappa M_2 C_\Omega \left(\sum_{j=1}^{\infty} \lambda_j^{2\gamma_0} |\varphi_j |^2 \right)^{\frac{1}{2}} \left(\sum_{j=1}^{\infty} \lambda_j^{-2(\gamma_0-\beta-\theta)}   \right)^{\frac{1}{2}}
$$
$$
\leq C_\kappa M_2 C_\Omega \|\varphi\|_{D\left(L^{\gamma_0}\right)} \left(\sum_{j=1}^{\infty} j^{\frac{-4(\gamma_0-\beta-\theta)}{\mathbb{N}}}   \right)^{\frac{1}{2}}.
$$
By the choice of $\theta$, we have
$\frac{4(\gamma_0-\beta-\theta)}{\mathbb{N}}=\frac{4(\gamma_0-\beta-\frac{\mathbb{N}}{4}-\varepsilon)}{\mathbb{N}}=\frac{\mathbb{N}+4(\gamma_0-\beta-\frac{\mathbb{N}}{2})-4\varepsilon}{\mathbb{N}}
=
\frac{\mathbb{N}+4\varepsilon}{\mathbb{N}}>1$, which implies
$$
\sum_{j=1}^{\infty} j^{\frac{-4(\gamma_0-\beta-\theta)}{\mathbb{N}}}<C\left(\gamma_0, \mathbb{N}\right).
$$
Therefore,
\begin{equation}\label{eqththeq146}
\left|\mathcal{Q}_1\right| \leq   C_\kappa M_2 C_\Omega  C\left(\gamma_0, \mathbb{N}\right) \|\varphi\|_{D\left(L^{\gamma_0}\right)}.
\end{equation}
Further,  by inequalities in Proposition \ref{thproth3} and (\ref{thththeq146}), we have the following estimate for $\mathcal{Q}_2:$
$$
\left|\mathcal{Q}_2 \right|
=\sum_{j=1}^{\infty} \left|\lambda_j^\beta \int\limits_0^t   F_j(\tau,\cdot,u(\tau,\cdot)) (t-\tau)^{\alpha-1} E_{\alpha, \alpha}\left(-\lambda_j^\beta (t-\tau)^\alpha\right)   d \tau e_j(x_0)  \right|
$$
$$
\leq   \sum_{j=1}^{\infty} \left|\lambda_j^\beta \int\limits_0^t   F_j(\tau,\cdot,u(\tau,\cdot)) (t-\tau)^{\alpha-1} E_{\alpha, \alpha}\left(-\lambda_j^\beta (t-\tau)^\alpha\right)   d \tau e_j(x_0)  \right|
$$
$$
\leq M_2 C_\Omega \sum_{j=1}^{\infty}  \lambda_j^\beta \int\limits_0^t   | F_j(\tau,\cdot,u(\tau,\cdot))|  (t-\tau)^{\alpha-1} \frac{1}{1+\lambda_j^\beta (t-\tau)^\alpha}     d \tau    \lambda_j^{\theta}
$$
$$
\leq \frac{1}{\lambda_1^{\beta p}} M_2 C_\Omega \int\limits_0^t (t-\tau)^{\alpha q-1} \sum_{j=1}^{\infty} \lambda_j^{\gamma_0}   |F_j(\tau,\cdot,u(\tau,\cdot))| \lambda_j^{-\left(\gamma_0-\beta-\theta\right)} d \tau
$$
$$
\leq \frac{M_2 C_\Omega}{\lambda_1^{\beta p}} \int\limits_0^t (t-\tau)^{\alpha q-1} \left(\sum_{j=1}^{\infty} \lambda_j^{2\gamma_0} |F_j(\tau,\cdot,u(\tau,\cdot))|^2 \right)^{\frac{1}{2}} \left(\sum_{j=1}^{\infty} \lambda_j^{-2(\gamma_0-\beta-\theta)}   \right)^{\frac{1}{2}} d \tau
$$
$$
\leq \frac{M_2 C_\Omega }{\lambda_1^{\beta p}} C\left(\gamma_0, \mathbb{N}\right) \int\limits_0^t (t-\tau)^{\alpha q-1} \|F_j(\tau,\cdot,u(\tau,\cdot))\|_{D\left(L^{\gamma_0}\right)}
 d \tau
$$
\begin{equation}\label{eqththeq147}
\leq \frac{M_2 C_\Omega \Upsilon}{\lambda_1^{\beta p}} C\left(\gamma_0, \mathbb{N}\right) \chi_0   B(\alpha q, 1-\alpha q) \|\varphi\|_{D\left(L^{\gamma_0}\right)}.
\end{equation}
Next we estimate $\mathcal{Q}_3.$ By Proposition \ref{thproth3}, we have
$$
|\mathcal{Q}_3| =\left|L^\beta \mathcal{G}_2(t, x) \mathcal{G}_1(T, x)F\right|
$$
$$
=\left|\sum_{j=1}^{\infty} \lambda_j^\beta \Psi_j(T)  \int\limits_0^T   F_j(\tau,\cdot,u(\tau,\cdot)) (T-\tau)^{\alpha-1} E_{\alpha, \alpha}\left(-\lambda_j^\beta (T-\tau)^\alpha\right)  d \tau  E_{\alpha, 1}\left(-\lambda_j^\beta t^\alpha\right)  e_j(x_0) \right|
$$
$$
\leq   \frac{M_2^2 C_\kappa C_\Omega }{\lambda_1^{\beta p}}  \int\limits_0^T (T-\tau)^{ \alpha q-1}  \sum_{j=1}^{\infty} \lambda_j^{\gamma_0}   |F_j(\tau,\cdot,u(\tau,\cdot))| \lambda_j^{-\left(\gamma_0-\beta-\theta\right)}  d \tau
$$
$$
\leq   \frac{M_2^2 C_\kappa C_\Omega }{\lambda_1^{\beta p}}  \int\limits_0^T (T-\tau)^{ \alpha q-1}  \left(\sum_{j=1}^{\infty} \lambda_j^{2\gamma_0} |F_j(\tau,\cdot,u(\tau,\cdot))|^2 \right)^{\frac{1}{2}} \left(\sum_{j=1}^{\infty} \lambda_j^{-2(\gamma_0-\beta-\theta)}   \right)^{\frac{1}{2}} d \tau
$$
$$
\leq \frac{M_2^2 C_\kappa C_\Omega }{\lambda_1^{\beta p}} C\left(\gamma_0, \mathbb{N}\right) \int\limits_0^T (T-\tau)^{\alpha q-1} \|F(\tau,\cdot,u(\tau,\cdot))\|_{D\left(L^{\gamma_0}\right)}
 d \tau
$$
\begin{equation}\label{eqththeq148}
\leq \frac{M_2^2 C_\kappa C_\Omega \Upsilon}{\lambda_1^{\beta p}} C\left(\gamma_0, \mathbb{N}\right) \chi_0   B(\alpha q, 1-\alpha q) \|\varphi\|_{D\left(L^{\gamma_0}\right)}.
\end{equation}
Similarly, we get estimate for $\mathcal{Q}_4$
$$
|\mathcal{Q}_4| =\Bigg|\sum_{j=1}^{\infty} \lambda_j^\beta \Psi_j(T)  \int\limits_0^T k(\tau)  u_j(\tau,\cdot) (T-\tau)^{\alpha-1} E_{\alpha, \alpha}\left(-\lambda_j^\beta (T-\tau)^\alpha\right)  d \tau
$$$$
\times E_{\alpha, 1}\left(-\lambda_j^\beta t^\alpha\right)  e_j(x_0) \Bigg|
\leq   \frac{M_2^2 C_\kappa C_\Omega }{\lambda_1^{\beta p}}  \int\limits_0^T (T-\tau)^{ \alpha q-1}  \sum_{j=1}^{\infty} \lambda_j^{\gamma_0}   | k(\tau)  u_j(\tau,\cdot)| \lambda_j^{-\left(\gamma_0-\beta-\theta\right)}  d \tau
$$
$$
\leq   \frac{M_2^2 C_\kappa C_\Omega }{\lambda_1^{\beta p}}  \int\limits_0^T (T-\tau)^{ \alpha q-1}  | k(\tau)| \left(\sum_{j=1}^{\infty} \lambda_j^{2\gamma_0} |u_j(\tau,\cdot)|^2 \right)^{\frac{1}{2}} \left(\sum_{j=1}^{\infty} \lambda_j^{-2(\gamma_0-\beta-\theta)}   \right)^{\frac{1}{2}} d \tau
$$
$$
\leq \frac{M_2^2 C_\kappa C_\Omega }{\lambda_1^{\beta p}} C\left(\gamma_0, \mathbb{N}\right) \|k\|_{C[0,T]} \int\limits_0^T (T-\tau)^{\alpha q-1} \|u(\tau,\cdot)\|_{D\left(L^{\gamma_0}\right)}
 d \tau
$$
\begin{equation}\label{eqththeq148thth}
\leq \frac{M_2^2 C_\kappa C_\Omega }{\lambda_1^{\beta p}} C\left(\gamma_0, \mathbb{N}\right) \|k\|_{C[0,T]} \chi_0   B(\alpha q, 1-\alpha q) \|\varphi\|_{D\left(L^{\gamma_0}\right)}.
\end{equation}

Finally, we estimate $\mathcal{Q}_5$ using the methods from the above estimates
 $$
\left|\mathcal{Q}_5 \right|
=\sum_{j=1}^{\infty} \left|\lambda_j^\beta \int\limits_0^t   k(\tau)u_j(\tau,\cdot) (t-\tau)^{\alpha-1} E_{\alpha, \alpha}\left(-\lambda_j^\beta (t-\tau)^\alpha\right)   d \tau e_j(x_0)  \right|
$$
$$
\leq   \sum_{j=1}^{\infty} \left|\lambda_j^\beta \int\limits_0^t   k(\tau)u_j(\tau,\cdot) (t-\tau)^{\alpha-1} E_{\alpha, \alpha}\left(-\lambda_j^\beta (t-\tau)^\alpha\right)   d \tau e_j(x_0)  \right|
$$
$$
\leq \frac{1}{\lambda_1^{\beta p}} M_2 C_\Omega \|k\|_{C[0,T]} \int\limits_0^t (t-\tau)^{\alpha q-1} \sum_{j=1}^{\infty} \lambda_j^{\gamma_0}   |u_j(\tau,\cdot)| \lambda_j^{-\left(\gamma_0-\beta-\theta\right)} d \tau
$$
$$
\leq \frac{M_2 C_\Omega}{\lambda_1^{\beta p}} \|k\|_{C[0,T]} \int\limits_0^t (t-\tau)^{\alpha q-1} \left(\sum_{j=1}^{\infty} \lambda_j^{2\gamma_0} |u_j(\tau,\cdot)|^2 \right)^{\frac{1}{2}} \left(\sum_{j=1}^{\infty} \lambda_j^{-2(\gamma_0-\beta-\theta)}   \right)^{\frac{1}{2}} d \tau
$$
$$
\leq \frac{M_2 C_\Omega }{\lambda_1^{\beta p}} C\left(\gamma_0, \mathbb{N}\right) \|k\|_{C[0,T]} \int\limits_0^t (t-\tau)^{\alpha q-1} \|u_j(\tau,\cdot)\|_{D\left(L^{\gamma_0}\right)}
 d \tau
$$
\begin{equation}\label{eqththeq149}
\leq \frac{M_2 C_\Omega  }{\lambda_1^{\beta p}} C\left(\gamma_0, \mathbb{N}\right) \chi_0   B(\alpha q, 1-\alpha q) \|k\|_{C[0,T]}\|\varphi\|_{D\left(L^{\gamma_0}\right)}.
\end{equation}
Using (\ref{eqththeq146}) - (\ref{eqththeq149}), we obtain an estimate for $L^\beta u(t, x_0)$
$$
\|L^\beta u(t, x_0)\|_{C[0,T]} \leq   C_\kappa M_2 C_\Omega  C\left(\gamma_0, \mathbb{N}\right) \|\varphi\|_{D\left(L^{\gamma_0}\right)}
$$$$
+
\frac{M_2 C_\Omega \Upsilon}{\lambda_1^{\beta p}} C\left(\gamma_0, \mathbb{N}\right) \chi_0   B(\alpha q, 1-\alpha q) \|\varphi\|_{D\left(L^{\gamma_0}\right)}
$$
$$
+\frac{M_2^2 C_\kappa C_\Omega \Upsilon}{\lambda_1^{\beta p}} C\left(\gamma_0, \mathbb{N}\right) \chi_0   B(\alpha q, 1-\alpha q) \|\varphi\|_{D\left(L^{\gamma_0}\right)}
$$
$$
+\frac{M_2^2 C_\kappa C_\Omega }{\lambda_1^{\beta p}} C\left(\gamma_0, \mathbb{N}\right)  \chi_0   B(\alpha q, 1-\alpha q) \|k\|_{C[0,T]} \|\varphi\|_{D\left(L^{\gamma_0}\right)}
$$
$$
+\frac{M_2 C_\Omega  }{\lambda_1^{\beta p}} C\left(\gamma_0, \mathbb{N}\right) \chi_0   B(\alpha q, 1-\alpha q) \|k\|_{C[0,T]}\|\varphi\|_{D\left(L^{\gamma_0}\right)}.
$$
From this, we obtain the estimate (\ref{eqtthth145}) from Lemma \ref{lemma7} and so complete the proof of this lemma.
\end{proof}

We make the following assumptions:\\
(D1) $h(t) \in C^1[0, T]$ and satisfy the conditions $|h(t)|\geq h_0>0$ where $h_0$ is given number;\\
(D2) $h(T)=\kappa h(0)+\varphi(x_0)$.\\
(D3)  $\exp{\left\{\frac{T^\alpha}{\Gamma(\alpha+1)}\left(\Upsilon+\|k\|_{C[0, T]}\right)\right\}}<1+|\kappa - 1|$

The following theorem holds.
\begin{theorem}\label{theorem8}
Let the conditions of the theorems \ref{theorth3} and conditions D1), D2), D3) be satisfied. Then the problem of finding a solution of (\ref{eqth1})--(\ref{eqth4}) is equivalent to the problem of determining the function $k(t) \in C[0,T]$ satisfying
\begin{equation}\label{eqth140}
k(t)= k_0(t) - \frac{1}{h(t)}\sum_{j=1}^{\infty} \lambda_j^\beta \left(u(t;k), e_j\right) e_j(x_0),
\end{equation}
where
\begin{equation}\label{eqth141}
k_0(t)=\frac{1}{h(t)}\left(F(t, x_0, h(t))
-  \partial_t^\alpha h(t)\right).
\end{equation}
\end{theorem}

\begin{proof}
 We split the proof into two steps.

\textit{Step 1}: Equivalent formulation as an integral equation.

Set $x = x_0$ in equation (\ref{eqth1}), the procedure yields
\begin{equation}\label{eqth142}
\partial_t^\alpha u(t, x_0)+L^\beta u(t, x_0)+k(t)u(t, x_0)=F(t, x_0, u(t, x_0)).
\end{equation}
Taking into account the condition $\partial_t^\alpha h(t)\in C[0;T]$ which imply that $h(t)\in C^1[0,T]$, and fractional differentiating both sides of (\ref{eqth4}) with respect to t give
\begin{equation}\label{eqth143}
\partial_t^\alpha u(t, x_0)=\partial_t^\alpha h(t), 0\leq t\leq T.
\end{equation}
Using (\ref{eqth4}), (\ref{eqth142}) and (\ref{eqth143}), we obtain
\begin{equation*}
k(t)=\frac{1}{h(t)}\left(F(t, x_0, h(t))
-  \partial_t^\alpha h(t)\right) - \frac{1}{h(t)}L^\beta u(t, x_0).
\end{equation*}
According to the definition (\ref{eqth4444}) of the fractional power of the operator $L$, we get
\begin{equation*}
k(t)=\frac{1}{h(t)}\left(F(t, x_0, h(t))
-  \partial_t^\alpha h(t)\right) - \frac{1}{h(t)}\sum_{j=1}^{\infty} \lambda_j^\beta \left(u(t,\cdot), e_j\right) e_j(x_0).
\end{equation*}
The functions $u(t)$ depend on $k(t)$, i.e. $u(t;k).$
After simple converting, we obtain
integral equation (\ref{eqth140}) for determining $k(t)$ of Theorem \ref{theorem8}.

\textit{Step 2.} Suppose now that $k \in C[0, T]$ satisfy  (\ref{eqth140}). In order to prove that $k$ is the solution to the inverse problem (\ref{eqth1})--(\ref{eqth4}), it suffices to show that (\ref{eqth4}). By the equation (\ref{eqth1}), we have (\ref{eqth142}). \\
Together with (\ref{eqth140}) and conditions (D1), (D2), we obtain that $\mu(t):= u(t,x_0)-h(t)$ satisfies
\begin{equation}\label{eqth144}
\left\{\begin{array}{l}
\partial_{t}^{\alpha} \mu(t)+k(t) \mu(t)=F(t, x_0, \mu(t)+ h(t))-F(t, x_0, h(t)), \quad t \in(0, T], \\
\mu(T)-\kappa \mu(0)=0.
\end{array}\right.
\end{equation}
Therefore, we have
$$
\mu(t)=
\frac{1}{(\kappa - 1)\Gamma(\alpha)} \int_0^T (T - s)^{\alpha - 1} \left(F(s, x_0, \mu(s)+ h(s))-F(s, x_0, h(s))\right) ds
$$
$$
-\frac{1}{(\kappa - 1)\Gamma(\alpha)} \int_0^T (T - s)^{\alpha - 1} k(s) \mu(s) ds
$$
$$
+\frac{1}{\Gamma(\alpha)} \int_0^t(t-s)^{\alpha-1} \left(F(s, x_0, \mu(s)+ h(s))-F(s, x_0, h(s))\right) ds
$$
$$
-\frac{1}{\Gamma(\alpha)} \int_0^t(t-s)^{\alpha-1} k(s) \mu(s) ds.
$$
According to the condition C1), for $k(t) \in C[0,T]$, we obtain the estimate
$$
\|\mu\|_{C[0, t]} \leq
\frac{\Upsilon+\|k\|_{C[0, T]}}{|\kappa - 1|\Gamma(\alpha)}  \int_0^T (T-s)^{\alpha - 1} |\mu(s)| ds
+\frac{\Upsilon+\|k\|_{C[0, T]}}{ \Gamma(\alpha)} \int_0^t(t-s)^{\alpha-1}  |\mu(s)| ds.
$$
According to the generalized Volterra Fredholm integral inequality in Theorem \ref{ththththth1} and condition (D3), we have $\|\mu\|_{C[0, t]}=0$ for all $t \in[0, T]$, which implies  $u(t,x_0)-h(t)=0$, $0 \leq t \leq T$, i.e., the condition (\ref{eqth4}) is satisfied. This completes the proof of Theorem \ref{theorem8}.\end{proof}

 \section*{6 Well-posedness of the inverse problem}

\begin{theorem}\label{theorem9}
Let the conditions of the theorems 4 and 5 and conditions D1), D2), D3) be satisfied. Then  there exists a sufficiently small $T > 0$ such that the inverse problem (\ref{eqth1})--(\ref{eqth4}) has a unique solution $k(t) \in C[0,T]$.
\end{theorem}

We can now prove the existence of a solution to our inverse problem. To prove the existence of a solution, we use the fixed point argument. First, we define the function set
$$
B_{R}^{T}=\left\{\overline{u} \in L^{\frac{1}{\alpha q^{\prime}}-r}\left(0, T ; D\left(L^{\beta (p-p')}\right)\right),\,\, \overline{k} \in C[0, T]:\quad \overline{u}(T, x)=\kappa \overline{u}(0, x)+\varphi(x),   \right.
$$$$\left.
\textbf{B} \overline{u}(t, x)=0, \quad(t, x) \in (0,T) \times \partial \Omega,\,\, \|\overline{u}\|_{L^{\frac{1}{\alpha q^{\prime}}-r}\left(0, T ; D\left(L^{\beta (p-p')}\right)\right)}+\|\overline{k}\|_{C[0, T]} \leq R\right\}.
$$
For given $\overline{u},\,\, \overline{k} \in B_{R}^{T}$, we also define the mapping $\Psi: B_{R}^{T} \rightarrow B_{R}^{T}$  such that $(\overline{u},\overline{k})\rightarrow (u,k)$  through
\begin{equation}\label{eqth150}
k(t)= \frac{1}{h(t)}\left(F(t, x_0, h(t))
-  \partial_t^\alpha h(t)\right) - \frac{1}{h(t)}L^\beta \overline{u}(t, x_0),
\end{equation}
and initial boundary value problem
\begin{equation}\label{eqth151}
\begin{cases}
\partial_t^\alpha u(t, x)+L^\beta u(t, x)+\overline{k}(t) \overline{u}(t, x)=F(t, x, \overline{u}(t, x)), \,\, (t, x) \in (0,T) \times \Omega,\\
\textbf{B} u(t, x)=0, \quad(t, x) \in (0,T) \times \partial \Omega,\\
u(T, x)=\kappa u(0, x)+\varphi(x), \quad x \in \Omega.
\end{cases}
\end{equation}
By H\''older's inequality, we have
$$
\|\overline{k}(t) \overline{u}(t, x)\|^2_{D\left(L^{-\beta q'}\right)} \leq
\sum_{n=1}^{\infty} \lambda_n^{-2 \beta q'}\left|\overline{k}(t) \overline{u}(t, x)\right|^2
$$$$
\leq \|\overline{k}\|^2_{C[0,t]} \sum_{n=1}^{\infty} \lambda_n^{-2 \beta q'}\left|\overline{u}(t, x)\right|^2 \leq \|\overline{k}\|^2_{C[0,t]}\|\overline{u}\|^2_{L^{\frac{1}{\alpha q^{\prime}}-r}\left(0, t ; D\left(L^{\beta (p-p')}\right)\right)}.
$$
From this we obtain the following
$$
\|\overline{k}(t) \overline{u}(t, x)\|_{L^{\frac{1}{\alpha q^{\prime}}-r}\left(0, T ; D\left(L^{\beta (p-p')}\right)\right)}    \leq \|\overline{k}\|_{C[0,T]}\|\overline{u}\|_{L^{\frac{1}{\alpha q^{\prime}}-r}\left(0, T ; D\left(L^{\beta (p-p')}\right)\right)}\leq R^2.
$$
By conditions C1), C2), we find that
$$
\|F(t, x, \overline{u}(t, x))\|_{L^{\frac{1}{\alpha q^{\prime}}-r}\left(0, T ; D\left(L^{\beta (p-p')}\right)\right)}    \leq \Upsilon \|\overline{u}\|_{L^{\frac{1}{\alpha q^{\prime}}-r}\left(0, T ; D\left(L^{\beta (p-p')}\right)\right)}\leq \Upsilon R
$$
Using these results, we have
$$\overline{k}(t) \overline{u}(t, x) \in  L^{\frac{1}{\alpha q^{\prime}}-r}\left(0, T ; D\left(L^{\beta (p-p')}\right)\right),\,\,
F(t, x, \overline{u}(t, x))\in  L^{\frac{1}{\alpha q^{\prime}}-r}\left(0, T ; D\left(L^{\beta (p-p')}\right)\right).
$$
By Theorem \ref{theorth3}, the unique solution $u(t, x) \in  L^{\frac{1}{\alpha q^{\prime}}-r}\left(0, T ; D\left(L^{\beta (p-p')}\right)\right)$ of problem (\ref{eqth151}), given by
$$
u(t, x)=\mathcal{G}_2(t, x) \varphi +\mathcal{G}_1(t, x) F(t, x, \overline{u}(t, x))
$$
\begin{equation}\label{eqth153}
+\mathcal{G}_3(t, x) F(t, x, \overline{u}(t, x))
-\mathcal{G}_3(t, x) \{\overline{k} \overline{u}\}
-\mathcal{G}_1(t, x)\{\overline{k} \overline{u}\},
\end{equation}
satisfies
$$
\|u(t, x)\|_{ L^{\frac{1}{\alpha q^{\prime}}-r}\left(0, T ; D\left(L^{\beta (p-p')}\right)\right)}\leq \frac{C_\kappa M_2}{ \lambda_1^{\beta p'}} \|\varphi\|_{D\left(L^{\beta (p)}\right)}
$$
$$
+
\frac{M_2 \Upsilon}{\alpha q \lambda_1^{\beta p'}}T^{\alpha q} \|\overline{u}\|_{L^{\frac{1}{\alpha q^{\prime}}-r}\left(0, T ; D\left(L^{\beta (p-p')}\right)\right)}+\frac{M^2_2 C_\kappa \Upsilon}{\alpha q \lambda_1^{\beta p'}}T^{\alpha q} \|\overline{u}\|_{L^{\frac{1}{\alpha q^{\prime}}-r}\left(0, T ; D\left(L^{\beta (p-p')}\right)\right)}
$$
$$
+\frac{M^2_2 C_\kappa T^{\alpha q}}{\alpha q \lambda_1^{\beta p'}} \|\overline{k}\|_{C[0,T]} \|\overline{u}\|_{L^{\frac{1}{\alpha q^{\prime}}-r}\left(0, T ; D\left(L^{\beta (p-p')}\right)\right)}
$$
\begin{equation}\label{eqth154}
+\frac{M_2  }{\alpha q \lambda_1^{\beta p'}}T^{\alpha q} \|\overline{k}\|_{C[0,T]}\|\overline{u}\|_{L^{\frac{1}{\alpha q^{\prime}}-r}\left(0, T ; D\left(L^{\beta (p-p')}\right)\right)}.
\end{equation}
Then (\ref{eqth150}) defines the function $k(t)$ in terms of $u$. Furthermore, by (\ref{eqth150}), we have
$$
|k(t)| \leq  \left|\frac{1}{h(t)}\left(F(t, x_0, h(t))
-  \partial_t^\alpha h(t)\right) - \frac{1}{h(t)}L^\beta \overline{u}(t, x_0)\right|
$$
\begin{equation}\label{eqth1555}
\leq \frac{1}{h_0}(\Upsilon+1) \|h\|_1+\frac{1}{h_0}\left|L^\beta \overline{u}(t, x_0)\right|,
\end{equation}
where $\|h\|_1=\max\left\{\|h\|_{C[0,T]}, \,\, \|\partial_t^\alpha h\|_{C[0,T]}\right\}.$

Using the inequality (\ref{eqtthth145}), we obtain from (\ref{eqth1555})
$$
|k(t)|  \leq \frac{(\Upsilon+1)}{h_0} \|h\|_1+ C_{11}  \Bigg(C_\kappa+\frac{\Upsilon}{\lambda_1^{\beta p}}\chi_0  B(\alpha q, 1-\alpha q)+\frac{M_2 C_\kappa \Upsilon}{\lambda_1^{\beta p}}\chi_0  B(\alpha q, 1-\alpha q)
$$
\begin{equation}\label{thth156}
+\frac{M_2 C_\kappa\|k\|_{C[0,T]}+ \|k\|_{C[0,T]}}{\lambda_1^{\beta p}}   \chi_0  B(\alpha q, 1-\alpha q)\Bigg)\|\varphi\|_{D\left(L^{\gamma_0}\right)}.
\end{equation}
Thus, the mapping
$$
\Psi: B_{R}^{T} \rightarrow L^{\frac{1}{\alpha q^{\prime}}-r}\left(0, T ; D\left(L^{\beta (p-p')}\right)\right) \times C[0,T],\quad (\overline{u},\overline{k})\rightarrow (u,k)
$$
given by (\ref{eqth150}) and (\ref{eqth151}), is well defined.

The next Theorem shows that $\Psi$ is a contraction map on $B_{R}^{T}$

\begin{theorem}\label{theorem10}
Let the conditions of the theorems \ref{theorth3}-\ref{theorth6} and conditions D1)- D3) be satisfied. Then there exists a sufficiently small $T_0$ and a suitable large $R$ such that the inverse problem (\ref{eqth1})--(\ref{eqth4}) has a unique solution $(u,k) \in L^{\frac{1}{\alpha q^{\prime}}-r}\left(0, T ; D\left(L^{\beta (p-p')}\right)\right) \times C[0,T]$ for all $T \in (0, T_0]$.
\end{theorem}

In the following we show that is a contraction map on $B_{R}^{T}$ for suitable $T$ and $R$. To
this end, we first prove  $\Psi(B_{R}^{T} ) \subset B_{R}^{T}$ for sufficiently small $T_0$ and suitably large $R$.  To simplify the calculations, we restrict $T \in (0, 1]$.  From Theorem \ref{theorth3} and (\ref{eqth154}),
we have
$$
\|u(t, x)\|_{ L^{\frac{1}{\alpha q^{\prime}}-r}\left(0, T ; D\left(L^{\beta (p-p')}\right)\right)}\leq \frac{C_\kappa M_2}{ \lambda_1^{\beta p'}} \|\varphi\|_{D\left(L^{\beta (p)}\right)}
$$
$$
+
\frac{M_2 \Upsilon}{\alpha q \lambda_1^{\beta p'}}T^{\alpha q} \|\overline{u}\|_{L^{\frac{1}{\alpha q^{\prime}}-r}\left(0, T ; D\left(L^{\beta (p-p')}\right)\right)}+\frac{M^2_2 C_\kappa \Upsilon}{\alpha q \lambda_1^{\beta p'}}T^{\alpha q} \|\overline{u}\|_{L^{\frac{1}{\alpha q^{\prime}}-r}\left(0, T ; D\left(L^{\beta (p-p')}\right)\right)}
$$
$$
+\frac{M_2  }{\alpha q \lambda_1^{\beta p'}}T^{\alpha q} \|\overline{k}\|_{C[0,T]}\|\overline{u}\|_{L^{\frac{1}{\alpha q^{\prime}}-r}\left(0, T ; D\left(L^{\beta (p-p')}\right)\right)}
$$
\begin{equation}\label{eqth162}
\leq \frac{C_\kappa M_2}{ \lambda_1^{\beta p'}} \|\varphi\|_{D\left(L^{\beta (p)}\right)} +
\frac{M_2 \Upsilon}{\alpha q \lambda_1^{\beta p'}}T^{\alpha q} R  +\frac{M^2_2 C_\kappa \Upsilon}{\alpha q \lambda_1^{\beta p'}}T^{\alpha q} R +\frac{M_2  }{\alpha q \lambda_1^{\beta p'}}T^{\alpha q}R^2.
\end{equation}
On the other hand, by (\ref{thth156}), we have
$$
\|k\|_{C[0,T]}  \leq \frac{(\Upsilon+1)}{h_0} \|h\|_1+ C_{11}  \Bigg(C_\kappa+\frac{\Upsilon}{\lambda_1^{\beta p}}\chi_0  B(\alpha q, 1-\alpha q)+\frac{M_2 C_\kappa \Upsilon}{\lambda_1^{\beta p}}\chi_0  B(\alpha q, 1-\alpha q)
$$
\begin{equation}\label{thth163}
+\frac{M_2 C_\kappa R+ R}{\lambda_1^{\beta p}}   \chi_0  B(\alpha q, 1-\alpha q)\Bigg)\|\varphi\|_{D\left(L^{\gamma_0}\right)}.
\end{equation}

Then, adding up (\ref{eqth162}) and (\ref{thth163}) leads to
\begin{equation}\label{thth159}
\|(u,k)\|_{L^{\frac{1}{\alpha q^{\prime}}-r}\left(0, T ; D\left(L^{\beta (p-p')}\right)\right)\times C[0,T]}\leq \mathcal{P} +
\frac{M_2 }{\alpha q \lambda_1^{\beta p'}}T^{\alpha q} R \left(\Upsilon+M_2 C_\kappa \Upsilon+R\right),
\end{equation}
where
$$
\mathcal{P}=\frac{C_\kappa M_2}{ \lambda_1^{\beta p'}} \|\varphi\|_{D\left(L^{\beta (p)}\right)}+ \frac{(\Upsilon+1)}{h_0} \|h\|_1+ C_{11}  \Bigg(C_\kappa+\frac{\Upsilon}{\lambda_1^{\beta p}}\chi_0  B(\alpha q, 1-\alpha q)
$$
$$
+\frac{M_2 C_\kappa \Upsilon}{\lambda_1^{\beta p}}\chi_0  B(\alpha q, 1-\alpha q)+\frac{M_2 C_\kappa R+ R}{\lambda_1^{\beta p}}   \chi_0  B(\alpha q, 1-\alpha q)\Bigg)\|\varphi\|_{D\left(L^{\gamma_0}\right)}
$$
and therefore satisfies
$$
\lim\limits_{T\rightarrow +0} \left(
\frac{M_2 }{\alpha q \lambda_1^{\beta p'}}T^{\alpha q} R \bigg(\Upsilon+M_2 C_\kappa \Upsilon+R\bigg)\right)=0.
$$
Now we take $R$, such that $R=2\mathcal{P}$ with the constant $\mathcal{P}$ in (\ref{thth159}). Then there exists $T_0 > 0$ such that
\begin{equation}\label{thth160}
\|(u,k)\|_{L^{\frac{1}{\alpha q^{\prime}}-r}\left(0, T ; D\left(L^{\beta (p-p')}\right)\right)\times C[0,T]}\leq R
\end{equation}
for all $T \in \left(0, T_0\right]$. That is, $\Psi$ maps $(B_{R}^{T})$ into itself for each fixed $T \in\left(0, \min \left\{1, T_0\right\}\right]$.

Next, we check the second condition of contractive mapping $\Psi$. Let $(u, k)=\Psi(\bar{u}, \bar{k})$ and $(U, K)=\Psi(\bar{U}, \bar{K})$. Then we obtain that $(u-U, k-K)$ satisfies that
\begin{equation}\label{eqth161}
(k-K)(t)=  - \frac{1}{h(t)}L^\beta (\overline{u}-\overline{U})(t, x_0),
\end{equation}
and initial boundary value problem
\begin{equation}\label{eqth167}
\begin{cases}
\partial_t^\alpha (u-U)(t, x)+L^\beta [u-U](t, x)+ (\overline{k}-\overline{K})\overline{u}+(\overline{u}-\overline{U})  \overline{K},\\
=F(t, x, \overline{u}(t, x))-F(t, x, \overline{U}(t, x)), \,\, (t, x) \in (0,T) \times \Omega,\\
\textbf{B} (u-U)(t, x)=0, \quad(t, x) \in (0,T) \times \partial \Omega,\\
(u-U)(T, x)=\kappa (u-U)(0, x), \quad x \in \Omega.
\end{cases}
\end{equation}

Again by Theorems \ref{theorth3} and \ref{theorth4}, the unique solution $u(t, x) \in  L^{\frac{1}{\alpha q^{\prime}}-r}\left(0, T ; D\left(L^{\beta (p-p')}\right)\right)$ of problem (\ref{eqth167}), given by
$$
(u-U)(t, x)= \mathcal{G}_1(t, x) \left(F(t, x, \overline{u}(t, x))-F(t, x, \overline{U}(t, x))\right)
$$$$
+\mathcal{G}_3(t, x) \left(F(t, x, \overline{u}(t, x))-F(t, x, \overline{U}(t, x))\right)
-\mathcal{G}_3(t, x) \left(\overline{k}(t)\overline{u}(t, x)-\overline{K}(t)\overline{U}(t, x)\right)
$$
\begin{equation}\label{eqth133}
-\mathcal{G}_1(t, x) \left(\overline{k}(t)\overline{u}(t, x)-\overline{K}(t)\overline{U}(t, x)\right).
\end{equation}
In formula (\ref{eqth133}), we use the following
$$
\mathcal{G}_3(t, x) \left(\overline{k}(t)\overline{u}(t, x)-\overline{K}(t)\overline{U}(t, x)\right)=\mathcal{G}_3(t, x)\{(\overline{u}-\overline{U})  \overline{K}\}+\mathcal{G}_3(t, x)\{(\overline{k}-\overline{K})\overline{u}\},
$$
$$
\mathcal{G}_1(t, x) \left(\overline{k}(t)\overline{u}(t, x)-\overline{K}(t)\overline{U}(t, x)\right)=\mathcal{G}_1(t, x)\{(\overline{u}-\overline{U})  \overline{K}\}+\mathcal{G}_1(t, x)\{(\overline{k}-\overline{K})\overline{u}\}.
$$
Using the above relations, we estimate (\ref{eqth133})
$$
\|u-U\|_{ L^{\frac{1}{\alpha q^{\prime}}-r}\left(0, T ; D\left(L^{\beta (p-p')}\right)\right)}\leq
\frac{M_2 \Upsilon (C_\kappa+M_2)}{\alpha q \lambda_1^{\beta p'}}T^{\alpha q} \|\overline{u}-\overline{U}\|_{L^{\frac{1}{\alpha q^{\prime}}-r}\left(0, T ; D\left(L^{\beta (p-p')}\right)\right)}
$$
$$
+\frac{M_2 (C_\kappa+M_2)}{\alpha q \lambda_1^{\beta p'}}T^{\alpha q} \|\overline{k}\|_{C[0,T]} \|\overline{u}-\overline{U}\|_{L^{\frac{1}{\alpha q^{\prime}}-r}\left(0, T ; D\left(L^{\beta (p-p')}\right)\right)}
$$
$$
+\frac{M_2 (C_\kappa+M_2)}{\alpha q \lambda_1^{\beta p'}}T^{\alpha q} \|\overline{k}-\overline{K}\|_{C[0,T]}\|\overline{u}\|_{L^{\frac{1}{\alpha q^{\prime}}-r}\left(0, T ; D\left(L^{\beta (p-p')}\right)\right)}
$$
$$
+\frac{M_2  }{\alpha q \lambda_1^{\beta p'}}T^{\alpha q} \|\overline{k}\|_{C[0,T]}\|\overline{u}-\overline{U}\|_{L^{\frac{1}{\alpha q^{\prime}}-r}\left(0, T ; D\left(L^{\beta (p-p')}\right)\right)}
$$
\begin{equation*}
+\frac{M_2  }{\alpha q \lambda_1^{\beta p'}}T^{\alpha q} \|\overline{k}-\overline{K}\|_{C[0,T]}\|\overline{u}\|_{L^{\frac{1}{\alpha q^{\prime}}-r}\left(0, T ; D\left(L^{\beta (p-p')}\right)\right)}
\end{equation*}
and from this
$$
\|u-U\|_{ L^{\frac{1}{\alpha q^{\prime}}-r}\left(0, T ; D\left(L^{\beta (p-p')}\right)\right)}\leq  \max\Bigg\{
\frac{M_2 \Upsilon (C_\kappa+M_2)}{\alpha q \lambda_1^{\beta p'}}T^{\alpha q}, \frac{M_2 R (C_\kappa+M_2)}{\alpha q \lambda_1^{\beta p'}}T^{\alpha q},
$$
\begin{equation}\label{eqth164}
\frac{M_2  R}{\alpha q \lambda_1^{\beta p'}}T^{\alpha q}, \frac{M_2  R}{\alpha q \lambda_1^{\beta p'}}T^{\alpha q}\Bigg\}\left[\|\overline{u}-\overline{U}\|_{L^{\frac{1}{\alpha q^{\prime}}-r}\left(0, T ; D\left(L^{\beta (p-p')}\right)\right)}+ \|\overline{k}-\overline{K}\|_{C[0,T]}\right].
\end{equation}

Similarly, according to  estimate (\ref{thth156}) and equality (\ref{eqth161})
\begin{equation}\label{eqth165}
\|k-K\|_{C[0,T]}  \leq \frac{1}{h_0} \mathcal{Y} \Bigg[\|\overline{u}-\overline{U}\|_{L^{\frac{1}{\alpha q^{\prime}}-r}\left(0, T ; D\left(L^{\beta (p-p')}\right)\right)}+\|\overline{k}-\overline{K}\|_{C[0,T]}\Bigg],
\end{equation}
where
$$
\mathcal{Y}=\max\Bigg\{
\frac{M_2 C_\kappa C\left(\gamma_0, \mathbb{N}\right)}{\alpha q \lambda_1^{\beta p'}},
$$$$
\frac{M_2^2 C_\Omega C_\kappa }{\alpha q \lambda_1^{\beta p'}}C\left(\gamma_0, \mathbb{N}\right),
\frac{M_2^2 C_\Omega C_\kappa R}{\alpha q \lambda_1^{\beta p'}}C\left(\gamma_0, \mathbb{N}\right),
\frac{M_2 C_\Omega R}{\alpha q \lambda_1^{\beta p'}}C\left(\gamma_0, \mathbb{N}\right)\Bigg\}
$$

Therefore, by (\ref{eqth164}) and (\ref{eqth165}), we have
$$
\|(u-U,k-K)\|_{L^{\frac{1}{\alpha q^{\prime}}-r}\left(0, T ; D\left(L^{\beta (p-p')}\right)\right)\times C[0,T]}
$$
$$
\leq \Bigg(\max\Bigg\{
\frac{M_2 \Upsilon (M_2 C_\kappa+M_2)}{\alpha q \lambda_1^{\beta p'}}T^{\alpha q}, \frac{M_2  R}{\alpha q \lambda_1^{\beta p'}}T^{\alpha q}, \frac{M_2  R}{\alpha q \lambda_1^{\beta p'}}T^{\alpha q}\Bigg\}+\frac{1}{h_0} \mathcal{Y}\Bigg)
$$
\begin{equation}\label{thth166}
\times \|(\overline{u}-\overline{U},\overline{k}-\overline{K})\|_{L^{\frac{1}{\alpha q^{\prime}}-r}\left(0, T ; D\left(L^{\beta (p-p')}\right)\right)\times C[0,T]}.
\end{equation}
Hence, we can choose a sufficiently small $T_{2}$ such that
$$
\max\Bigg\{
\frac{M_2 \Upsilon (M_2 C_\kappa+1)}{\alpha q \lambda_1^{\beta p'}}T^{\alpha q}, \frac{M_2  R}{\alpha q \lambda_1^{\beta p'}}T^{\alpha q}, \frac{M_2  R}{\alpha q \lambda_1^{\beta p'}}T^{\alpha q}\Bigg\}+\frac{1}{h_0} \mathcal{Y} < \frac{1}{2}
$$
for all $T\in (0; T_{2}]$ to obtain
$$
\|(u-U,k-K)\|_{L^{\frac{1}{\alpha q^{\prime}}-r}\left(0, T ; D\left(L^{\beta (p-p')}\right)\right)\times C[0,T]} \qquad \qquad \qquad
$$
\begin{equation}\label{thth167}
\qquad \qquad \qquad \leq \frac{1}{2} \|(\overline{u}-\overline{U},\overline{k}-\overline{K})\|_{L^{\frac{1}{\alpha q^{\prime}}-r}\left(0, T ; D\left(L^{\beta (p-p')}\right)\right)\times C[0,T]}.
\end{equation}
Estimates (\ref{thth160}) and (\ref{thth167}) show that $\Psi$ is a contraction map on $B_R^T$ for all $T\in (0,T_2]$, if we choose
$T\leq \min\{1,T_1,T_2\}$.


\begin{thebibliography}{99}



\bibitem{KSTr2006} Kilbas A.A., Srivastava H.M., Trujillo J.J.
Theory and Application of Fractional Differential Equations.
North-Holland Mathematical Studies, Amsterdam: Elsevier, 2006.

\bibitem{TVE2010} Tarasov V.E.
Fractional Dynamics: Applications of Fractional Calculus to Dynamics of Particles.
Fields and Media, New York: Springer, 2010.

\bibitem{PodI1999} Podlubny I.
Fractional Differential Equations.
San Diego: Academic Press, 1999.

\bibitem{SKYM2011} Sakamoto K., Yamamoto M.
Initial value/boundary value problems for fractional diffusion-wave equations and applications
to some inverse problems.
Journal of Mathematical Analysis and Applications, 382(1), 426-447, 2011.

\bibitem{ASHA2013} Ashyralyev A.
Well-posedness of fractional parabolic equations.
Boundary Value Problems, 2013, 31, 2013.

\bibitem{YKYM2021} Yavar K., Yamamoto M.
Well-posedness for weak and strong solutions of non-homogeneous initial boundary value
problems for fractional diffusion equations.
Fractional Calculus and Applied Analysis, 24, 168-201, 2021.

\bibitem{KAYM2018} Kubica A., Yamamoto M.
Initial-boundary value problems for fractional diffusion equations with time-dependent coefficients.
Fractional Calculus and Applied Analysis, 21(2), 276-311, 2018.

\bibitem{GRLYYM2015} Gorenflo R., Luchko Y., Yamamoto M.
Time-fractional diffusion equation in the fractional Sobolev spaces.
Fractional Calculus and Applied Analysis, 18, 799-820, 2015.

\bibitem{MAH2017} Mu J., Ahmad B., Huang Sh.
Existence and regularity of solutions to time-fractional diffusion equations.
Computers and Mathematics with Applications, 73(6), 985-996, 2017.

\bibitem{DABCA2015} Andrade De.B., Carvalho A.N., Carvalho-Neto P.M., Marin-Rubio P.
Semilinear fractional differential equations: global solutions, critical nonlinearities and comparison results.
Topological Methods in Nonlinear Analysis, 45(2), 439-467, 2015.


\bibitem{LYRWYMZL2013} Luchko Y., Rundell W., Yamamoto M., Zuo L.
Uniqueness and reconstruction of an unknown semilinear term in a time-fractional reaction diffusion equation.
Inverse Problems, 29, 065019, 2013.

\bibitem{TNZO2020} Tuan N.H., Ngoc T.B., Zhou Y., O'Regan D.
On existence and regularity of a terminal value problem for the time fractional diffusion equation.
Inverse Problems, 36(5), 055011, 2020.

\bibitem{TNHLNT2019} Tuan N.H., Huynh Le.N., Ngoc T.B., Zhou Y.
On a backward problem for nonlinear fractional diffusion equations.
Applied Mathematics Letters, 92(1), 76-84, 2019.

\bibitem{TNLNTT2016} Tuan N.H., Long L.D., Nguyen V.T., Tran T.
On a final value problem for the time-fractional diffusion equation with inhomogeneous source.
Inverse Problems in Science and Engineering, 25(9), 1367-1395, 2016.

\bibitem{ZQCM2011} Chen Zh.Q., Meerschaert M.M., Nane E.
Space-time fractional diffusion on bounded domains.
Journal of Mathematical Analysis and Applications, 393, 479-488, 2011.

\bibitem{JJLK2016} Jia J., Li K.
Maximum principles for a time-space fractional diffusion equation.
Applied Mathematics Letters, 62, 23-28, 2016.

\bibitem{GBHST2015} Guswanto B.H., Suzuki T.
Existence and uniqueness of mild solutions for fractional semilinear differential equations.
Electronic Journal of Differential Equations, 168, 1-16, 2015.

\bibitem{TSUS2013} Tatar S., Ulusoy S.
A uniqueness result in an inverse problem for a space-time fractional diffusion equation.
Electronic Journal of Differential Equations, 258, 1-9, 2013.

\bibitem{DDKRAA2025} Durdiev D.K., Rahmonov A.
Global solvability of inverse coefficient problem for one fractional diffusion equation with initial non-local and integral overdetermination conditions.
Fractional Calculus and Applied Analysis, 28(1), 117-145, 2025.

\bibitem{DDKJJ2025} Durdiev D.K., Jumaev J.J., Turdiev H.H.
Inverse problem for determining time dependent coefficient and source functions in a time-fractional diffusion equation.
Journal of Mathematical Sciences, 289(4), 475-486, 2025.

\bibitem{LGZDJXY2013} Li G., Zhang D., Jia X., Yamamoto M.
Simultaneous inversion for the space-dependent diffusion coefficient and the fractional order in the time fractional diffusion equation.
Inverse Problems, 29, 065014, 2013.

\bibitem{TWJGW2014} Wei T., Wang J.G.
A modified quasi-boundary value method for the backward time-fractional diffusion problem.
ESAIM: Mathematical Modelling and Numerical Analysis, 78, 95-111, 2014.

\bibitem{RAA2024} Rahmonov A.A.
Recovering the time-dependent coefficient in fractional wave equation.
Uzbek Mathematical Journal, 68(2), 125-140, 2024.

\bibitem{JnJ2020} Janno J.
Determination of time-dependent sources and parameters of nonlocal diffusion and wave equations from final data.
Fractional Calculus and Applied Analysis, 23, 1678-1701, 2020.


\bibitem{MLMG1902} Mittag-Leffler M.G.
Sur l'integrale de Laplace-Abel.
Comp. Rend. Acad. Sci. Paris, 135, 937-939, 1902.

\bibitem{MLMG1903} Mittag-Leffler M.G.
Une generalization de l'integrale de Laplace-Abel.
Comp. Rend. Acad. Sci. Paris, 136, 537-539, 1903.

\bibitem{MLMG19033} Mittag-Leffler M.G.
Sur la nouvelle fonction $E_\alpha(x)$.
Comp. Rend. Acad. Sci. Paris, 137, 554-558, 1903.

\bibitem{MLMG1904} Mittag-Leffler M.G.
Sopra la funzione $E_\alpha(x)$.
Rendiconti R. Accademia dei Lincei (Ser. 5), 13, 3-5, 1904.


\bibitem{ZheB2014} Zheng B.
Explicit bounds derived by some new inequalities and applications in fractional integral equations.
Journal of Inequalities and Applications, 2014(4), 2014.

\bibitem{BWU2014} Bin Wu, Siyuan Wu.
Existence and uniqueness of an inverse source problem for a fractional integrodifferential equation.
Computers and Mathematics with Applications, 68(10), 1123-1136, 2014.

\bibitem{PAA2012} Pazy A.
Semigroups of Linear Operators and Applications to Partial Differential Equations.
Springer Science and Business Media, Vol. 44, New York, 2012.

\bibitem{MikhVP1976} Evans, L.C. Partial Differential Equations. volume 19. American Mathematical
Society, Second Edition, USA 2010.



\end{thebibliography}
\end{document}